\documentclass[article,onefignum,onetabnum]{siamart190516}

\usepackage{lipsum}
\usepackage{amsfonts}
\usepackage{graphicx}
\ifpdf
  \DeclareGraphicsExtensions{.eps,.pdf,.png,.jpg}
\else
  \DeclareGraphicsExtensions{.eps}
\fi

\newsiamthm{assumption}{Assumption}
\crefname{assumption}{Assumption}{Assumptions}
\Crefname{assumption}{Assumption}{Assumptions}

\newsiamthm{remark}{Remark}

\crefname{section}{Section}{Sections}
\crefname{subsection}{Section}{Sections}
\Crefname{subsection}{Section}{Sections}

\headers{PinT for hyperbolic systems}{O. A. Krzysik, H. De Sterck, R. D. Falgout, J. B. Schroder}

\title{Parallel-in-time solution of hyperbolic PDE systems via characteristic-variable block preconditioning\thanks{A published version of this preprint is available at \url{https://doi.org/10.1137/24M1673310}.
\funding{Los Alamos Laboratory report number LA-UR-25-20668.
This work was performed under the auspices of the U.S. Department of Energy by Lawrence Livermore National Laboratory under Contract DE-AC52-07NA27344 (LLNL-JRNL-2000224).
This work was supported in part by the U.S. Department of Energy, Office of Science, Office of Advanced Scientific Computing Research, Applied Mathematics program, and by NSERC of Canada.
}}}

\author{O. A. Krzysik\thanks{Department of Applied Mathematics, University of Waterloo, Waterloo, Ontario, Canada. \textit{Present address:} Theoretical Division, Los Alamos National Laboratory, Los Alamos, New Mexico, USA
  (\email{okrzysik@lanl.gov}, \url{https://orcid.org/0000-0001-7880-6512}).}
\and
H. De Sterck\thanks{Department of Applied Mathematics, University of Waterloo, Waterloo, Ontario, Canada} 
  (\email{hans.desterck@uwaterloo.ca}, \url{https://orcid.org/0000-0002-1641-932X}).
\and R. D. Falgout\thanks{Center for Applied Scientific Computing, Lawrence Livermore National Laboratory, Livermore, California, USA 
  (\email{falgout2@llnl.gov}, \url{https://orcid.org/0000-0003-4884-0087}).}
  \and J. B. Schroder\thanks{Department of Mathematics and Statistics, University of New Mexico, Albuquerque, New Mexico, USA
  (\email{jbschroder@unm.edu}, \url{https://orcid.org/0000-0002-1076-9206}).}
}

\usepackage{amsopn}
\DeclareMathOperator{\diag}{diag} 
\DeclareMathOperator{\blkdiag}{block-diag} 

\usepackage{xcolor}
\usepackage{bm} 

\usepackage{enumitem} 

\usepackage{comment} 

\usepackage{mathtools} 
\mathtoolsset{centercolon} 

\usepackage[linesnumbered,ruled,vlined,algo2e]{algorithm2e}

\SetCommentSty{mycommfont}
\SetKwRepeat{Do}{do}{while}%

\usepackage{amssymb} 
\usepackage{makecell} 
\usepackage{multirow} 
\usepackage{makecell} 
\usepackage{pifont}      

\usepackage{boldline}  

\usepackage{grffile} 

\Crefname{subsection}{Section}{Sections}

\usepackage{tikz}
\usetikzlibrary{calc}
\usepackage{graphbox}

\DeclareFontFamily{U}{mathx}{\hyphenchar\font45}
\DeclareFontShape{U}{mathx}{m}{n}{
      <5> <6> <7> <8> <9> <10>
      <10.95> <12> <14.4> <17.28> <20.74> <24.88>
      mathx10
      }{}
\DeclareSymbolFont{mathx}{U}{mathx}{m}{n}
\DeclareFontSubstitution{U}{mathx}{m}{n}
\DeclareMathAccent{\widecheck}{0}{mathx}{"71}
\DeclareMathAccent{\wideparen}{0}{mathx}{"75}

\newcommand{\wh}[1]{\widehat{#1}} 								  
\newcommand{\wc}[1]{\widecheck{#1}}
\newcommand{\wt}[1]{\widetilde{#1}}

\usepackage{grffile} 

\makeatletter
\renewcommand\tableofcontents{%
\@starttoc{toc}%
}
\makeatother

\usepackage[us,12hr]{datetime} 

\usepackage{xcolor}
\usepackage{marginnote}
\makeatletter
\@mparswitchfalse%
\normalmarginpar
\makeatother

\allowdisplaybreaks

\begin{document}

%
%
%
%


\maketitle

\begin{abstract}
	We consider the parallel-in-time solution of hyperbolic partial differential equation (PDE) systems in one spatial dimension, both linear and nonlinear. In the nonlinear setting, the discretized equations are solved with a preconditioned residual iteration based on a global linearization.
	The linear(ized) equation systems are approximately solved parallel-in-time using a block preconditioner applied in the characteristic variables of the underlying linear(ized) hyperbolic PDE. This change of variables is motivated by the observation that inter-variable coupling between characteristic variables is  weak, at least locally where spatio-temporal variations in the eigenvectors of the associated flux Jacobian are sufficiently small, while that between the original variables is not. 
	For an $\ell$-dimensional system of PDEs, applying the preconditioner consists of solving a sequence of $\ell$ scalar linear(ized)-advection-like problems, each being associated with a different characteristic wave-speed in the underlying linear(ized) PDE.
	We approximately solve these linear advection problems using multigrid reduction-in-time (MGRIT); however, any other suitable parallel-in-time method could be used.
	Numerical examples are shown for the (linear) acoustics equations in heterogeneous media, and for the (nonlinear) shallow water equations and Euler equations of gas dynamics with shocks and rarefactions.
	For many test problems the solver converges in just a handful of iterations, and with mesh-independent convergence rates.
\end{abstract}



\section{Introduction}
\label{sec:introduction}

Parallel-in-time methods for solving evolutionary differential equations have grown in prominence over the past two decades, with many effective strategies developed capable of yielding substantial speed-ups over sequential time-stepping \cite{Gander2015,Ong_Schroder_2020}.
However, an aspect that limits the practicability of these methods is their lack of robust convergence speed for hyperbolic problems, and for advection-dominated problems more broadly.
In particular, two of the most well-known algorithms, Parareal \cite{Lions_etal_2001} and multigrid reduction-in-time (MGRIT) \cite{Falgout_etal_2014}, are well-documented to suffer in this regard \cite{Ruprecht_Krause_2012,
Steiner_etal_2015,
Dobrev_etal_2017,
Nielsen_etal_2018,
Hessenthaler_etal_2018,
Ruprecht_2018,
Schmitt_etal_2018,
Schroder_2018,
Howse_etal_2019,
DeSterck_etal_2019,
DeSterck_etal_2021,
Danieli_MacLachlan_2023}.
Other parallel-in-time strategies, namely ``diagonalization-in-time,'' do exist which do not seem to suffer (at least to the same extent) in this regard, e.g.,
\cite{Goddard-Wathen-2019-all-at-once-wave,
Gander-etal-2019-wave-direct,
Gander-Wu-2020-wave-para-diag,
Liu-Wu-2020-alpha-wave,
Danieli-Wathen-2021-wave-GMRES,
Liu-Wu-2022-sinc-nystrom,
Liu-etal-2022-1st-and-2nd-order,
Hon-SC-2023-block-precond-wave,
Wu-Zhou-2021,HopeCollins-etal-2024}; however, their efficacy for general hyperbolic problems with spatio-temporally varying flux Jacobians remains unclear \cite{Wu-Zhou-2021,HopeCollins-etal-2024}.

Recently we have made significant inroads towards developing more efficient MGRIT and Parareal solvers for hyperbolic partial differential equations (PDEs).
Underpinning our methodology is a modified semi-Lagrangian coarse-grid discretization for linear advection problems \cite{DeSterck_etal_2023_MOL,
DeSterck_etal_2023_SL,
DeSterck-etal-2023-nonlin-scalar}. 
This coarse-grid discretization is derived so as to better approximate the fine-grid discretization over a subset of problematic smooth Fourier modes known as characteristic components \cite{DeSterck_etal_2025_LFA}.
The methodology is motivated from that in \cite{Yavneh_1998}, wherein poor coarse-grid correction of characteristic components was addressed in the context of spatial multigrid for steady state advection-dominated PDEs. 

Given the difficulty of effectively applying MGRIT and Parareal to scalar hyperbolic PDEs, there is, unsurprisingly, little literature on their application to hyperbolic systems. Specifically, literature on the subject appears limited to: Acoustics \cite{Ruprecht_Krause_2012}, shallow water \cite{Nielsen_etal_2018,Danieli_MacLachlan_2023}, Euler \cite{Danieli_MacLachlan_2023}, and (a partially hyperbolic formulation of) linear elasticity \cite{Hessenthaler_etal_2018,DeSterck_etal_2019}.
However, none of these works appear to provide a solution methodology that fundamentally addresses the non-robust convergence speed issues for hyperbolic problems, e.g., the solvers tend to deteriorate as the grid is refined or as the time domain size increases. This motivates the solver we develop in this paper, which exhibits fast, mesh-independent convergence rates for several challenging test problems.
We do note that MGRIT- and Parareal-like methods have been applied with various levels of success to advection-diffusion-like PDE systems, e.g., dissipative shallow water equation systems \cite{Haut-Wingate-2014-asymptotic,Peddle-etal-2019-asymptotic-para,
Steinstraesser-etal-2024-SWE-MGRIT}, and Navier--Stokes   \cite{Steiner_etal_2015,Christopher-etal-2020-amr,Guzik-etal-2023-turbulence}.
We mention also that certain non-MGRIT/Parareal parallel-in-time methods have been applied to hyperbolic systems, including the acoustics equations \cite{Doerfler-etal-2019}, and the linearized shallow water equations \cite{Haut-etal-2015-exp-skew}.

In this work we consider both linear and nonlinear hyperbolic systems, extending various aspects of our previous parallel-in-time work for scalar hyperbolic PDEs, both linear \cite{DeSterck_etal_2025_LFA,
DeSterck_etal_2021,
DeSterck_etal_2023_MOL,
DeSterck_etal_2023_SL}, and most recently nonlinear \cite{DeSterck-etal-2023-nonlin-scalar}.
In particular, here nonlinear hyperbolic systems are solved using an extension of the global linearization strategy developed in \cite{DeSterck-etal-2023-nonlin-scalar} for nonlinear hyperbolic scalar PDEs.
The key novelty of this work is the development of a block preconditioner for solving linear(ized) hyperbolic systems with MGRIT. Applying this preconditioner requires space-time linear advection solves, which we approximate with MGRIT, relying on modified coarse-grid semi-Lagrangian discretizations to do so; however, the preconditioner may be used in conjunction with any other suitable parallel-in-time method in place of MGRIT.

A key aspect of our solver is first mapping to characteristic variables before applying block preconditioning techniques, which is motivated by the observation that inter-variable coupling between characteristic variables is significantly weaker than that between primitive variables, so that block preconditioning in the former is significantly more effective than in the latter.
In this sense, our block preconditioning strategy is related to that proposed by Reynolds et al. \cite{Reynolds-etal-2010-char-prec} for solving linearized stiff hyperbolic systems in the implicit time-stepping context. Specifically, the preconditioner in \cite{Reynolds-etal-2010-char-prec} maps to characteristic variables and solves only for the stiffness-inducing waves before mapping back to the original variables.
We further note that the current work is not the first to consider space-time block preconditioning for PDE systems, with the recent work in \cite{Danieli-etal-2022-block-prec,Schroder-etal-2023-block-prec-MHD} extending incompressible MHD and fluid dynamics preconditioners from the time-stepping context to the space-time setting.
In comparison to \cite{Danieli-etal-2022-block-prec,Schroder-etal-2023-block-prec-MHD}, we consider hyperbolic PDEs, and our proposed approach is designed to be applied to existing sequential time-stepping simulations, whereas the approaches in \cite{Danieli-etal-2022-block-prec,Schroder-etal-2023-block-prec-MHD} are structurally different in that they assume an all-at-once space-time discretization.

While parallel-in-time literature for hyperbolic systems is sparse, second-order wave equations have received considerably more attention, especially within the last five years
\cite{Dai_Maday_2013,
Gander_Guttel_2013,
Goddard-Wathen-2019-all-at-once-wave,
Gander-etal-2019-wave-direct,
Gander-Wu-2020-wave-para-diag,
Liu-Wu-2020-alpha-wave,
Danieli-Wathen-2021-wave-GMRES,
Southworth-etal-2021-mgrit-ana,
Liu-Wu-2022-sinc-nystrom,
Liu-etal-2022-1st-and-2nd-order,
Hon-SC-2023-block-precond-wave}.
Second-order wave equations are closely related to first-order hyperbolic systems, and in some circumstances may even be equivalent under reformulation (see \cref{sec:acoustics}).
Most of the aforementioned approaches for second-order wave equations target temporal parallelism by diagonalization-in-time
\cite{Goddard-Wathen-2019-all-at-once-wave,
Gander-etal-2019-wave-direct,
Gander-Wu-2020-wave-para-diag,
Liu-Wu-2020-alpha-wave,
Danieli-Wathen-2021-wave-GMRES,
Liu-Wu-2022-sinc-nystrom,
Liu-etal-2022-1st-and-2nd-order,
Hon-SC-2023-block-precond-wave}.
However, as mentioned previously, the applicability and efficacy of diagonalization-in-time approaches for general nonlinear hyperbolic conservation laws remains unclear \cite{Wu-Zhou-2021,HopeCollins-etal-2024}.

The remainder of this paper is organized as follows.
\Cref{sec:acoustics} introduces the acoustics equations, and \cref{sec:acoustics-pint} presents our parallel-in-time solver for them, including with numerical results. 
Relevant background for nonlinear hyperbolic systems is presented in \cref{sec:cons}, and \cref{sec:cons-pint} extends the acoustics-specific solver to this more general nonlinear setting.
Concluding remarks are given in \cref{sec:conclusion}.
The MATLAB code used to generate the results in this manuscript can be found at \url{https://github.com/okrzysik/pit-nonlinear-hyperbolic} (v1.0.0).

\section{Acoustics: Introduction and background}
\label{sec:acoustics}

In this section we provide background on the acoustics equations, including their wave-propagation nature (\cref{sec:acoustics-char-var}) and their discretization (\cref{sec:acoustics-disc}). 

The acoustics equations are a prototypical linear hyperbolic system, and are typically presented in non-conservative form, which is the formulation we use here (see \cite{LeVeque2002-nonlin-elasticity} for their conservative formulation).
In one spatial dimension they are
\begin{align} \label{eq:acoustic} \tag{acoustics}
\frac{\partial \bm{q}}{\partial t} + A_0 \frac{\partial \bm{q}}{\partial x} 
=
\frac{\partial}{\partial t} 
\begin{bmatrix}
p \\
u
\end{bmatrix} 
+  
\begin{bmatrix}
0 & K_0 \\
1/\rho_0 & 0
\end{bmatrix}
\frac{\partial}{\partial x}
\begin{bmatrix}
p \\
u
\end{bmatrix} 
=
\begin{bmatrix}
0 \\
0
\end{bmatrix}.
\end{align}
This system arises in the context of gas dynamics where $p$ and $u$ represent pressure and velocity perturbations, respectively, linearized about some steady state solution of an underlying nonlinear hyperbolic system; also, systems with equivalent or closely related structure arise in electromagnetism and elasticity \cite{LeVeque2002-nonlin-elasticity,
LeVeque_2004,
Hesthaven-Warburton-2008-DG,
Dafermos-2010-book,
Hesthaven_2017}.

In \eqref{eq:acoustic}, $\rho_0 = \rho_0(x) > 0$ and $K_0 = K_0(x) > 0$ are given parameters corresponding to the density and bulk modulus, respectively, of the underlying material. These material parameters are related to the \textit{sound speed} $c_0 = c_0(x) > 0$ and the \textit{impedance} $Z_0 = Z_0(x) > 0$ by
\begin{align} \label{eq:c0-Z0}
c_0 = \sqrt{K_0 / \rho_0}, 
\quad
\textrm{and}
\quad
Z_0 = \sqrt{K_0 \rho_0}.
\end{align}
Where convenient, we include a ``0'' subscript on material parameters to distinguish them from the unknown variables in \eqref{eq:acoustic}.
The material parameters $\rho_0$, $K_0$, $c_0$ and $Z_0$, may vary in space, but we do not consider variations in time.
In particular, spatial variation in the impedance, $Z_0$, leads to a much richer solution structure in \eqref{eq:acoustic} than constant $Z_0$, as we discuss in the next section.

Note that the coupled \textit{first-order} system \eqref{eq:acoustic} is equivalent to a pair of decoupled \textit{second-order} wave equations. 
Specifically, assuming that $p$ and $u$ are twice differentiable, then \eqref{eq:acoustic} can be manipulated to yield the equations
\begin{align} \label{eq:second-order-wave}
\frac{\partial^2 p}{\partial t^2} - K_0 \frac{\partial}{\partial x} \left( \frac{1}{\rho_0} \frac{\partial p}{\partial x} \right) = 0,
\quad
\textrm{and}
\quad
\frac{\partial^2 u}{\partial t^2} - \frac{1}{\rho_0} \frac{\partial}{\partial x} \left( K_0 \frac{\partial u}{\partial x} \right) = 0.
\end{align}
We do not work directly with these wave equations, however, since the first-order systems formulation more naturally extends to nonlinear hyperbolic conservation laws.

The coefficient matrix $A_0$ in \eqref{eq:acoustic} satisfies $A_0 = R_0 \Lambda_0 R^{-1}$ with
\begin{align} \label{eq:acoustic-eigen}
\Lambda_0 
= 
\begin{bmatrix}
-c_0 & 0 \\
0 & c_0
\end{bmatrix},
\quad
R_0 
=
\begin{bmatrix}
-Z_0 & Z_0 \\
1 & 1
\end{bmatrix},
\quad
R_0^{-1} = \frac{1}{2Z_0}
\begin{bmatrix}
-1 & Z_0 \\
  1 & Z_0
\end{bmatrix}.
\end{align}
Since $A_0$ has distinct real eigenvalues, system \eqref{eq:acoustic} is strictly hyperbolic.

\subsection{Characteristic variables}
\label{sec:acoustics-char-var}

Here we briefly review the wave-propagation nature of \eqref{eq:acoustic} since this is crucial to the development of our solvers for it.
To this end, let us reconsider \eqref{eq:acoustic} in \textit{characteristic variables}, which are defined by $\bm{w} := R_0^{-1} \bm{q} = (w^1, w^2)^\top$. 
Left multiplying \eqref{eq:acoustic} by $R_0^{-1}$ gives
\begin{align} \label{eq:acoustic-char}
R_0^{-1}
\left[
 \frac{\partial \bm{q}}{\partial t} 
+ 
(R_0 \Lambda_0 R_0^{-1}) \frac{\partial \bm{q}}{\partial x} 
\right]
=
\frac{\partial \bm{w}}{\partial t} 
+ 
\Lambda_0 R_0^{-1} 
\frac{\partial}{\partial x} 
\big( 
R_0 \bm{w}
\big)
= \bm{0}.
\end{align}
If the derivative of $Z_0$, i.e., $Z_0' = Z'_0(x)$, exists, the $x$ derivative in \eqref{eq:acoustic-char} can be expanded to yield a system of advection-reaction equations:
\begin{align} \label{eq:acoust-char-adv-rea}
\frac{\partial \bm{w}}{\partial t} 
+ 
\Lambda_0
\frac{\partial \bm{w}}{\partial x} 
= 
B_0 \bm{w},
\hspace{1ex}
\textrm{or}
\hspace{1ex}
\frac{\partial }{\partial t} 
\begin{bmatrix}
w^1 \\
w^2
\end{bmatrix}
+
\begin{bmatrix}
-c_0 & 0 \\
0 & c_0
\end{bmatrix} 
\frac{\partial }{\partial x} 
\begin{bmatrix}
w^1 \\
w^2
\end{bmatrix}
= 
\frac{c_0 Z'_0}{2 Z_0}
\begin{bmatrix}
w^1 - w^2 \\
w^1 - w^2
\end{bmatrix},
\end{align}
where $B_0 := -\Lambda_0 R_0^{-1} \frac{\partial R_0}{\partial x}$.
That is, $w^1$ and $w^2$ are \textit{coupled} left- and right-travelling waves. The coupling strength is, in some sense, proportional to $Z'_0$: Smooth changes in $Z_0$ induce coupling of the characteristic variables.
Notice that if $Z_0' \equiv 0$ then the two waves decouple globally.
If $Z_0'$ does not exist, \eqref{eq:acoustic-char} cannot be simplified as in \eqref{eq:acoust-char-adv-rea}; however, the characteristic variables still propagate as coupled left- and right-travelling waves, with coupling strength still an increasing function of changes in the impedance (see \cite[Chapter 9.9]{LeVeque_2004}).

\subsection{Godunov discretization}
\label{sec:acoustics-disc}

We discretize \eqref{eq:acoustic} on a spatial mesh consisting of $n_x$ finite-volume (FV) cells of equal width $h$.
The $i$th FV cell is $x \in [x_{i - 1/2}, x_{i+1/2}]$, $i \in \{ 1, \ldots, n_x \}$.
The temporal domain $t \in [0, T]$ is discretized with a grid of $n_t$ points, $\{ t_n \}_{n = 0}^{n_t -1 }$, $t_n = n \delta t$, with constant spacing $\delta t$.

To discretize \eqref{eq:acoustic} we use a first-order accurate Godunov method \cite[Chapter 9]{LeVeque_2004}.
The pressure, velocity, sound speed, and impedance are all assumed constant within each FV cell, with values denoted by $p_i$, $u_i$, $c_i$, and $Z_i$, respectively, in cell $i$.\footnote{For notational simplicity, when referring to discretized values of $c_0$ and $Z_0$ we drop their ``0'' subscripts. Note also that implementationally the constant value in each cell is taken as the function's value in the cell center, i.e., $c_i = c_0(x_i)$, $Z_{i} = Z_0(x_i)$.}
The discretization advances the solution from $t_n \to t_{n+1}$ in cell $i$ according to
\begin{subequations} \label{eq:acoustics-god}
\begin{align} 
\begin{bmatrix}
p_i \\
u_i
\end{bmatrix}^{n+1}
&=
\begin{bmatrix}
p_i \\
u_i
\end{bmatrix}^{n}
- 
\frac{\delta t}{h}
\left(
c_i 
{\cal W}^2_{i - 1/2}
-
c_i 
{\cal W}^1_{i+1/2}
\right),
\quad
\textrm{where}
\\
{\cal W}^1_{i-1/2}
&=
\frac{-(p_i - p_{i-1}) + Z_i(u_i - u_{i-1})}{Z_{i - 1} + Z_{i}}
\begin{bmatrix}
-Z_{i-1} \\
1
\end{bmatrix},
\\
{\cal W}^2_{i - 1/2}
&=
\frac{(p_i - p_{i-1}) + Z_{i-1}(u_i - u_{i-1})}{Z_{i - 1} + Z_{i}}
\begin{bmatrix}
Z_i \\
1
\end{bmatrix}.
\end{align}
\end{subequations}
The general principle underlying this discretization is to update cell averages by accounting for changes due to the waves propagating into each cell across the time step; for a given cell, this requires solving a Riemann problem at each of its interfaces.

Moving forward, it is useful to represent \eqref{eq:acoustics-god} in terms of an update for the stacked spatial vector $\bm{q}^n := (\bm{p}^n, \bm{u}^n)^{\top} \in \mathbb{R}^{2 n_x}$, with $\bm{p}^n = (p_1^n, \ldots, p_{n_x}^n)^\top, \bm{u}^n = (u_1^n, \ldots, u_{n_x}^n)^\top$.
In particular, we express this one-step update in the form
\begin{align} \label{eq:acoustics-god-Phi}
\bm{q}^{n+1}
=
\Phi
\bm{q}^{n},
\quad
\Phi 
=
\begin{bmatrix}
I & 0 \\
0 & I
\end{bmatrix}
-
\frac{\delta t}{h}
\begin{bmatrix}
{\cal C}_0 & 0 \\
0 & {\cal C}_0
\end{bmatrix}
\begin{bmatrix}
{\cal L}_{pp} & {\cal L}_{pu} \\
{\cal L}_{up} & {\cal L}_{uu} 
\end{bmatrix}
\equiv
\begin{bmatrix}
\Phi_{pp} & \Phi_{pu} \\
\Phi_{up} & \Phi_{uu}
\end{bmatrix},
\end{align}
with ${\cal C}_0 := \diag ( c_1, \ldots, c_{n_x} )$, and tridiagonal ${\cal L}_{ij} \in \mathbb{R}^{n_x \times n_x}$ defined by the stencils
\begin{subequations} \label{eq:acoustics-L-stencils}
\begin{align}
\big[ {\cal L}_{pp} \big]_i
&=
\left[
- \frac{Z_i}{Z_{i-1} + Z_{i}}, \,
   \frac{Z_i}{Z_{i-1} + Z_{i}} + \frac{Z_{i}}{Z_{i} + Z_{i+1}}, \,
- \frac{Z_{i}}{Z_{i} + Z_{i+1}}
\right],
\\
\big[ {\cal L}_{pu} \big]_i
&=
\left[
- \frac{Z_{i-1} Z_i}{Z_{i-1} + Z_{i}}, \,
   \frac{Z_{i-1} Z_i}{Z_{i-1} + Z_{i}} - \frac{Z_{i} Z_{i+1}}{Z_{i} + Z_{i+1}}, \,
  \frac{Z_{i} Z_{i+1}}{Z_{i} + Z_{i+1}}
\right],
\\
\big[ {\cal L}_{up} \big]_i 
&=
\left[
- \frac{1}{Z_{i-1} + Z_{i}}, \,
   \frac{1}{Z_{i-1} + Z_{i}} - \frac{1}{Z_{i} + Z_{i+1}}, \,
   \frac{1}{Z_{i} + Z_{i+1}}
\right],
\\
\big[ {\cal L}_{uu} \big]_i
&=
\left[
- \frac{Z_{i-1}}{Z_{i-1} + Z_{i}}, \,
   \frac{Z_{i-1}}{Z_{i-1} + Z_{i}} + \frac{Z_{i+1}}{Z_{i} + Z_{i+1}}, \,
- \frac{Z_{i+1}}{Z_{i} + Z_{i+1}}
\right],
\end{align}
\end{subequations}
where we have used stencil notation: We denote by $[Q]_i$ the stencil for the  $i$th row of the matrix $Q \in \mathbb{R}^{n_x \times n_x}$, with, e.g., $[ \cdot ]$ denoting a diagonal entry, and $[ \cdot, \cdot, \cdot ]$ denoting sub-, main-, and super-diagonal entries.

\section{Acoustics: Parallel-in-time solution}
\label{sec:acoustics-pint}

In this section, we consider the parallel-in-time solution of the fully discretized problem 
\begin{align}
\bm{q}^{n+1} \label{eq:godunov-ts}
=
\Phi 
\bm{q}^n,
\quad
n = 0, 1, \ldots, n_t-2,
\end{align}
corresponding to the Godunov discretization \eqref{eq:acoustics-god-Phi} of \eqref{eq:acoustic}.
To this end, it is convenient to re-express \eqref{eq:godunov-ts} as the global linear system
\begin{align} \label{eq:godunov-all-at-once}
{\cal A} \bm{q} 
\equiv
\begin{bmatrix}
I & \\
-\Phi & I \\
& \ddots & \ddots \\
& & -\Phi & I
\end{bmatrix}
\begin{bmatrix}
\bm{q}^0 \\
\bm{q}^1 \\
\vdots \\
\bm{q}^{n_t-1} 
\end{bmatrix}
=
\begin{bmatrix}
\bm{q}^0 \\
\bm{0} \\
\vdots \\
\bm{0}
\end{bmatrix}
=:
\bm{b},
\end{align}
where ${\cal A} \in \mathbb{R}^{2 n_x n_t \times 2 n_x n_t}$, $\bm{q} \in \mathbb{R}^{2 n_x n_t}$, and $\bm{b} \in \mathbb{R}^{2 n_x n_t}$ are the space-time discretization matrix, space-time solution vector, and right-hand side vector, respectively.

The crux of our solver for \eqref{eq:godunov-all-at-once} is a residual correction scheme with block preconditioner applied in the characteristic variables.
Next \cref{sec:acoustics-outer-iter} introduces and motivates the outer residual update iteration, and then \cref{sec:acoustics-god-char-var} considers the Godunov discretization in characteristic-variable space. \Cref{sec:acoustics-prec} introduces block preconditioners with numerical results following in \cref{sec:acoustics-num-res}.

\subsection{The outer iteration with characteristic-variable block preconditioning}
\label{sec:acoustics-outer-iter}

Here we describe our solver for \eqref{eq:godunov-all-at-once} based on a residual correction scheme with characteristic-variable block preconditioning. 
This block preconditioning strategy is the key novelty introduced in this paper.
Pseudo code for the solver is given in \cref{alg:char-prec}, and the main steps are detailed more thoroughly below.\footnote{The pseudo code in \cref{alg:char-prec} and the steps below denote the time-stepping operator $\Phi$ and the eigenvectors ${\cal R}_0$ of the flux Jacobian as time dependent, i.e., $\Phi = \Phi^n$ and ${\cal R}_0 = {\cal R}_0^n$, even though those arising from \eqref{eq:acoustic} are not. We make this slight generalization here since it will be useful later in \cref{sec:cons-pint} when considering more general linear(ized) hyperbolic systems that have temporally varying flux Jacobians.}
\begin{algorithm}[h!]
\KwIn{Hyperbolic system space-time discretization matrix ${\cal A}$; initial iterate $\bm{q}_0 \approx \bm{q}$; right-hand side $\bm{b}$; number of iterations \texttt{maxit}}
\KwOut{Approximate solution $\bm{q}_k \approx \bm{q}$}
$k \gets 0$\tcp*{Iteration counter}
\While{$k < \mathtt{maxit}$}{
  	${\bm{q}}_{k} \gets \textrm{relax on }{\cal A}(\bm{q}_k) \approx \bm{b}$\tcp*{F-relaxation}\label{ln:lin-F-relax}
	${\bm{r}}_{k} \gets \bm{b} - {\cal A} ({\bm{q}}_{k})$\tcp*{Compute residual}\label{ln:lin-res-comp}
	$\wh{\bm{r}}_{k} \gets \underset{n = 0, \ldots, n_t - 1}{\blkdiag} (({\cal R}_0^n)^{-1}) \bm{r}_{k}$\tcp*{Compute char. residual}\label{ln:lin-res-char-comp}
	\tcc{Approximately solve ${\wh{{\cal A}} \wh{\bm{e}}_k = \wh{\bm{r}}_k}$ using block preconditioner $\wh{{\cal P}}$ or $\wt{{\cal P}}$ of \cref{sec:acoustics-prec}. Diagonal blocks of $\wh{{\cal P}}$ and $\wt{{\cal P}}$ can be approximately inverted parallel-in-time.}
	\uIf{using $\wh{{\cal P}}$ from \eqref{eq:prec-def}}{ 
      $\wh{\bm{e}}_k \gets \wh{{\cal P}}^{-1} \wh{\bm{r}}_k$\tcp*{Uses block diagonal of $\wh{{\cal A}}$}\label{ln:lin-ehat}
    } 
    \uElseIf{using $\wt{{\cal P}}$ from \eqref{eq:prec-approx-def}}{
      $\wh{\bm{e}}_k \gets \wt{{\cal P}}^{-1} \wh{\bm{r}}_k$\tcp*{Approximates block diagonal of $\wh{{\cal A}}$}\label{ln:lin-etilde}
    }\label{ln:lin-end-of-solve}
	$\bm{q}_{k+1} \gets \bm{q}_k + \underset{n = 0, \ldots, n_t - 1}{\blkdiag}({\cal R}_0^n) \wh{\bm{e}}_k$\tcp*{Compute new iterate}\label{ln:new-iterate}
	$k \gets k + 1$\tcp*{Increment iteration counter}
	}
	  \caption{Characteristic-based block preconditioned iteration for ${\cal A} \bm{q} = \bm{b}$. 
  \label{alg:char-prec}}
\end{algorithm}
At iteration $k$ of the algorithm the approximate solution to \eqref{eq:godunov-all-at-once} is denoted by $\bm{q}_k \approx \bm{q}$, with associated algebraic residual and error given by
\begin{align}
\bm{r}_k = \bm{b} - {\cal A} \bm{q}_k,
\quad
\textrm{and}
\quad
\bm{e}_k = \bm{q} - \bm{q}_k.
\end{align}
A preconditioned residual correction scheme would compute $\bm{q}_{k + 1} = \bm{q}_k + P^{-1} \bm{r}_k$, where $P^{-1} \bm{r}_k \approx {\cal A}^{-1} \bm{r}_k = \bm{e}_k$, but we will carry out this correction in characteristic variables with a block preconditioner that can be applied efficiently using MGRIT.
The method requires an integer $m > 1$ to induce a CF-splitting of the time grid $\{ t_n \}_{n = 0}^{n_t - 1}$ as in MGRIT \cite{Falgout_etal_2014}, i.e., such that every $m$th time point, $t_0, t_m, t_{2m}, \ldots,$ is a C-point and all other points are F-points.
The algorithm also requires F-relaxation as it is used in MGRIT \cite{Falgout_etal_2014};  F-relaxation updates the solution in parallel across each block of consecutive $m-1$ F-points by time-stepping the solution from their preceding C-point, i.e., $\bm{q}^{j m+1}, \ldots, \bm{q}^{j m + m - 1}$ is updated based on $\bm{q}^{j m}$ in parallel for each $j \in \{0, 1, \ldots, \lfloor (n_t-1)/m \rfloor \}$.
%
%

\underline{Lines \ref{ln:lin-F-relax} and \ref{ln:lin-res-comp}:} Perform an F-relaxation on the system \eqref{eq:godunov-all-at-once} and then compute $\bm{r}_k$.
We now have a residual equation 
\begin{align} \label{eq:acoustics-res}
{\cal A} \bm{e}_k = \bm{r}_k,
\quad
\textrm{or}
\quad
\bm{e}_k^{n+1} 
&=
\Phi^n 
\bm{e}_k^n
+
\bm{r}_k^{n+1},
\quad
n = 0, 1, \ldots, n_t-2.
\end{align}

\underline{Lines \ref{ln:lin-res-char-comp}--\ref{ln:lin-end-of-solve}:} To solve \eqref{eq:acoustics-res}, first transform it to characteristic variable space.
To this end, we introduce the block matrices $({\cal R}_0^n)^{-1}$, ${\cal R}_0^n \in \mathbb{R}^{2 n_x \times 2 n_x}$ corresponding to spatially discretizing the left and right eigenvector matrices of $A_0$ in \eqref{eq:acoustic-eigen}:
\begin{align} \label{eq:acoustic-calR}
({\cal R}_0^n)^{-1}
:=
\frac{1}{2}
 \begin{bmatrix}
-{\cal Z}_0^{-1} & I \\
{\cal Z}_0^{-1} & I
\end{bmatrix},
\quad
\textrm{and}
\quad
{\cal R}_0^n
:=
 \begin{bmatrix}
-{\cal Z}_0 & {\cal Z}_0 \\
I & I
\end{bmatrix},
\end{align}
where ${\cal Z}_0 := \diag( Z_1, \ldots, Z_{n_x} ) \in \mathbb{R}^{n_x \times n_x}$ such that $\bm{w}_k^n = ({\cal R}_0^n)^{-1} \bm{q}_k^n$ is the current approximation of the characteristic variables.
Left multiplying the $n$th equation in \eqref{eq:acoustics-res} by $({\cal R}_0^{-1})^n$ gives
\begin{align} \label{eq:acoustics-res-char}
\wh{\bm{e}}_k^{n+1}
&=
\wh{\Phi}^n
\wh{\bm{e}}_k^{n}
+
\wh{\bm{r}}_k^{n+1},
\quad
n = 0, 1, \ldots, n_t-2.
\end{align}
Here quantities in characteristic space are denoted with hats, i.e.,
\begin{align} \label{eq:wh-Phi}
\wh{\bm{e}}^n_k := ({\cal R}_0^n)^{-1} \bm{e}^n_k,
\quad
\wh{\bm{r}}^n_k := ({\cal R}_0^n)^{-1} \bm{r}^n_k,
\quad
\wh{\Phi}^n
:=
({\cal R}_0^n)^{-1}
\Phi^n 
{\cal R}_0^n
=
\begin{bmatrix}
\wh{\Phi}_{11}^n & \wh{\Phi}_{12}^n \\
\wh{\Phi}_{21}^n & \wh{\Phi}_{22}^n
\end{bmatrix}.
\end{align}
The vectors $\wh{\bm{r}}_k^n$ and $\wh{\bm{e}}_k^n$ represent the algebraic residual and error, respectively, of $\bm{w}_k^n = ({\cal R}_0^n)^{-1} \bm{q}_k^n$.
Furthermore, $\wh{\Phi}^n$ is the time-stepping operator for the Godunov discretization \eqref{eq:acoustics-god-Phi} re-expressed in characteristic-variable space.

We re-order the global system \eqref{eq:acoustics-res-char} so that all time points for each characteristic variable are blocked together:
\begin{align} \label{eq:A-wh-blocked}
\wh{{\cal A}} \wh{\bm{e}}_k = \wh{\bm{r}}_k,
\quad
\textrm{or}
\quad
\begin{bmatrix}
\wh{{\cal A}}_{11} & \wh{{\cal A}}_{12} \\
\wh{{\cal A}}_{21} & \wh{{\cal A}}_{22}
\end{bmatrix}
\begin{bmatrix}
\wh{\bm{e}}_{k,1} \\[0.75ex]
\wh{\bm{e}}_{k,2}
\end{bmatrix}
=
\begin{bmatrix}
\wh{\bm{r}}_{k,1} \\[0.75ex]
\wh{\bm{r}}_{k,2}
\end{bmatrix},
\end{align}
where the blocks $\wh{{\cal A}}_{ii}, \wh{{\cal A}}_{ij} \in \mathbb{R}^{n_x n_t \times n_x n_t}$, $j \neq i$, are given by
\begin{align} 
\wh{{\cal A}}_{ii}
=
\begin{bmatrix}
I \\
-\wh{\Phi}_{ii}^0 & I \\
& \ddots & \ddots \\
\end{bmatrix},
\quad
\wh{{\cal A}}_{ij}
=
\begin{bmatrix}
0 \\
-\wh{\Phi}_{ij}^0 & 0 \\
& \ddots & \ddots \\
\end{bmatrix}.
\end{align}

Then approximately solve \eqref{eq:A-wh-blocked} using block preconditioning techniques.
To this end, let $\wh{{\cal P}} \approx \wh{{\cal A}}$ be a $2 \times 2$ block preconditioner, with $\wh{{\cal P}}^{-1}$ amenable to parallel-in-time implementation.
Further details on this preconditioner are given in \cref{sec:acoustics-prec}, including an approximation to it denoted as $\wt{{\cal P}}$ in \cref{alg:char-prec}.
The error is approximated via application of the preconditioner to the residual:
\begin{align} \label{eq:prec-update}
\begin{bmatrix}
\wh{\bm{e}}_{k, 1} \\[0.75ex]
\wh{\bm{e}}_{k, 2}
\end{bmatrix}
\gets
\wh{{\cal P}}^{-1}
\begin{bmatrix}
\wh{\bm{r}}_{k, 1} \\[0.75ex]
\wh{\bm{r}}_{k, 2}
\end{bmatrix},
\end{align}

\underline{Line \ref{ln:new-iterate}:} Compute the new iterate in primitive-variable space:
\begin{align} \label{eq:acoustics-final-update}
\bm{q}_{k+1}^n \gets 
\bm{q}_k^n
+ 
{\cal R}_0 \wh{\bm{e}}_k^n,
\quad
n = 0, 1, \ldots, n_t - 1.
\end{align}
Now we provide more detailed comments on certain aspects of this outer iteration.

\textit{Solving the residual equation \eqref{eq:acoustics-res} in characteristic-variable space:} 
In characteristic space, inter-variable coupling is, in general, \textit{relatively weaker} than intra-variable coupling.
We recall from \cref{sec:acoustics-char-var} that, at the PDE level,  characteristic variables decouple locally wherever the impedance is constant---this is investigated at the discrete level next in \cref{sec:acoustics-god-char-var}.
In contrast, inter-variable coupling in primitive space is not weak relative to intra-variable coupling, and has a complicated skew-symmetric-like structure---see the stencils for ${\cal L}_{pu}$ and ${\cal L}_{up}$ in \eqref{eq:acoustics-L-stencils}.
Thus, we anticipate block-preconditioning strategies will be more effective in characteristic space than they would be in primitive space.
In fact, our numerical experiments in Supplementary Materials Section \ref{SMsec:block-prec-prim} indicate that block preconditioning in primitive variables is completely ineffective, with simple block-preconditioned residual schemes strongly diverging, even on problems with relatively simple material parameters $c_0$ and $Z_0$.
Note also that we anticipate having effective parallel-in-time preconditioners in characteristic space, because, as described in \cref{sec:acoustics-char-var}, the characteristic variables propagate as waves, and therefore, inversion of the diagonal blocks $\wh{{\cal A}}_{ii}$ in \eqref{eq:A-wh-blocked} should correspond approximately to space-time linear advection solves, for which we have previously developed effective MGRIT methods \cite{DeSterck_etal_2023_SL,DeSterck-etal-2023-nonlin-scalar}. 

\textit{Mapping to characteristic space.} Solving system \eqref{eq:A-wh-blocked} requires the characteristic residuals $\{ \wh{\bm{r}}_{k}^n \}_{n = 0}^{n_t - 1}$. From \eqref{eq:wh-Phi}, $\wh{\bm{r}}_{k}^n = {\cal R}_{0}^{-1} \bm{r}_{k}^n$; however, by virtue of the F-relaxation in Line \ref{ln:lin-F-relax} of \cref{alg:char-prec}, $\bm{r}_0^n$ is zero all all F-points, so that the characteristic transformation $\wh{\bm{r}}_{k}^n = {\cal R}_{0}^{-1} \bm{r}_{k}^n$ need only be performed at C-points.
\begin{remark}
The transformation of quantities into characteristic-variable space and back is inexpensive relative to the computation of the F-relaxation, residual computation, and the application of the preconditioner that occur in \cref{alg:char-prec}. This is because the transformation need only be done at C-points, and because the blocks in the $2 \times 2$ matrices $({\cal R}_0^n)^{-1}$ and ${\cal R}_0^n$ are diagonal (see \eqref{eq:acoustic-calR}). Furthermore, this is a completely local-in-time calculation that does not require any communication when the algorithm is implemented in a parallel environment.
\end{remark}

\textit{Correction in primitive space \eqref{eq:acoustics-final-update}.} This correction need only be done at C-points because the F-relaxation in Line \ref{ln:lin-F-relax} of \cref{alg:char-prec} on the next iteration will immediately overwrite all F-point values of $\bm{q}_k$ based only on the value of $\bm{q}_k$ at C-points.

\textit{Residual calculation.} The outer iteration does not technically require the F-relaxation in Line \ref{ln:lin-F-relax} of \cref{alg:char-prec}. Rather, all that is required is to compute the residual $\bm{r}_k$. We elect to compute this residual via first performing an F-relaxation because it may have computational savings when combined with our parallel-in-time implementation of $\wh{{\cal P}}$---which itself uses an inner F-relaxation. Note also that if no F-relaxation is done, then the two prior points no longer hold: The residual needs to be mapped to characteristic space at all time points, and the error correction needs to be applied at all time points.
%

\subsection{Godunov discretization in characteristic variables}
\label{sec:acoustics-god-char-var}

To develop effective preconditioners for \eqref{eq:A-wh-blocked} we need to understand the structure of its blocks, and, hence, that of $\wh{\Phi}$ in \eqref{eq:wh-Phi}, i.e., the Godunov discretization in characteristic-variable space. 
The below lemma and following corollary express the stencils for $\wh{\Phi}$.

\begin{lemma}[Godunov discretization in characteristic variables]
\label{lem:god-char-stencils}
Let $\Phi \in \mathbb{R}^{2 n_x \times 2 n_x}$ be the time-stepping operator for the Godunov discretization \eqref{eq:acoustics-god} of \eqref{eq:acoustic}, and $\wh{\Phi} := {\cal R}_0^{-1} \Phi {\cal R}_0 \in \mathbb{R}^{2 n_x \times 2 n_x}$ be the time-stepping operator for the same discretization expressed in characteristic variables as in \eqref{eq:wh-Phi}.
Then, the stencils of blocks of $\wh{\Phi}$ are
\begin{subequations} \label{eq:acoustics-god-Phi-hat-stencils}
\begin{alignat}{2}
\label{eq:whPhi11-lem}
\big[ \wh{\Phi}_{11} \big]_{i}
&=
\bigg(
\big[ 1 \big] 
+ 
c_i \frac{\delta t}{h} 
\left[0, \, -1, \, 1 \right]
\bigg)
&&+
c_i \frac{\delta t}{h} 
\left[0, \, 0, \,  \frac{Z_{i+1} - Z_{i}}{Z_{i+1} + Z_{i}} \right]
\\
\label{eq:whPhi22-lem}
\big[ \wh{\Phi}_{22} \big]_{i}
&=
\bigg(
\big[ 1 \big] 
- 
c_i \frac{\delta t}{h} 
\left[ -1, \, 1, \,  0 \right]
\bigg)
&&+
c_i \frac{\delta t}{h} 
\left[\frac{Z_{i-1} - Z_{i}}{Z_{i-1} + Z_{i}}, \, 0, \,  0 \right],
\\
\label{eq:whPhi12-lem}
\big[ \wh{\Phi}_{12} \big]_{i}
&=
&&-
c_i \frac{\delta t}{h} 
\left[ \frac{Z_{i+1} - Z_{i}}{Z_{i+1} + Z_{i}} \right], 
\\
\label{eq:whPhi21-lem}
\big[ \wh{\Phi}_{21} \big]_{i}
&=
&&-
c_i \frac{\delta t}{h} 
\left[\frac{Z_{i-1} - Z_{i}}{Z_{i-1} + Z_{i}} \right].
\end{alignat}
\end{subequations} 
\end{lemma}

\begin{proof}
The proof works by straightforward, albeit somewhat tedious, algebraic manipulation of the stencils in \eqref{eq:acoustics-god-Phi}. Details are given in \cref{app:proofs}.
\end{proof}

The structure here can be seen more easily if a smoothness assumption is imposed on the impedance, as in the following corollary.

\begin{corollary}[Impedance-linearized Godunov discretization in characteristic variables]
\label{cor:god-Z-linearized}
Suppose the impedance $Z(x)$ is twice differentiable in cells $i-1,i,i+1$.
Then, the stencils of blocks of $\wh{\Phi}$ given in \eqref{eq:acoustics-god-Phi-hat-stencils} satisfy
\begin{subequations} \label{eq:acoustics-god-Phi-hat-stencils-linearized}
\begin{alignat}{2}
\label{eq:whPhi11-cor}
\big[ \wh{\Phi}_{11} \big]_{i}
&=
\left(
\big[ 1 \big] 
+ 
c_i \frac{\delta t}{h} 
\left[0, \, -1, \, 1 \right] 
\right)
&&+
c_i \frac{\delta t}{h} 
\frac{h}{2 Z_i}
\big[
0, 0, Z'_i + {\cal O} (h)
\big],
\\
\label{eq:whPhi22-cor}
\big[ \wh{\Phi}_{22} \big]_{i}
&=
\left(
\big[ 1 \big] 
- 
c_i \frac{\delta t}{h} 
\left[-1, \, 1, \,  0 \right]
\right)
&&-
c_i \frac{\delta t}{h} 
\frac{h}{2 Z_i}
\big[
Z'_i + {\cal O} (h), 0, 0
\big],
\\
\label{eq:whPhi12-cor}
\big[ \wh{\Phi}_{12} \big]_{i}
&=
&&-c_i \frac{\delta t}{h} 
\frac{h}{2 Z_i}
\big[ 
Z'_i
+
{\cal O} (h)
\big], 
\\
\label{eq:whPhi21-cor}
\big[ \wh{\Phi}_{21} \big]_{i}
&=
&&+
c_i \frac{\delta t}{h} 
\frac{h}{2 Z_i}
\big[ 
Z'_i
+
{\cal O} (h)
\big].
\end{alignat}
\end{subequations} 
\end{corollary}

\begin{proof}
The result follows from Taylor expansion of the $Z$-dependent terms in \eqref{eq:acoustics-god-Phi-hat-stencils} about the point $x_i$, recalling the shorthand $Z_j = Z(x_j)$.
\end{proof}

Clearly, there are two different scales present in these stencils.
We refer to stencil components in \eqref{eq:acoustics-god-Phi-hat-stencils} and \eqref{eq:acoustics-god-Phi-hat-stencils-linearized} as ``order zero'' if they are independent of mesh parameters, or if they are of size $c_i \tfrac{\delta t}{h} \lesssim 1$, noting that the CFL number is a fixed constant independent of the mesh resolution.
Otherwise, we refer to terms proportional to ${\cal O} \big( Z_{i} - Z_{i \pm 1} \big)$ in \eqref{eq:acoustics-god-Phi-hat-stencils} or ${\cal O} \big( h Z'_i \big)$ in \eqref{eq:acoustics-god-Phi-hat-stencils-linearized}  as ``order one'' since they are of size ${\cal O}(h)$ when $Z(x)$ is sufficiently smooth.

Consider first the inter-variable coupling terms given by the stencils $\big[ \wh{\Phi}_{12} \big]_i$ in \eqref{eq:whPhi12-lem}/\eqref{eq:whPhi12-cor} and $\big[ \wh{\Phi}_{21} \big]_i$ in \eqref{eq:whPhi21-lem}/\eqref{eq:whPhi21-cor}.
Notice that if the impedance is constant across cells $i-1,i,i+1$, then $\big[ \wh{\Phi}_{12} \big]_i = \big[ \wh{\Phi}_{21} \big]_i = 0$, so that there is no inter-variable coupling, consistent with what occurs at the PDE level (see \cref{sec:acoustics-char-var}).
Otherwise, if the impedance varies smoothly across these cells, the inter-variable coupling is relatively weak, being order one.
We also note that the structure of the inter-variable coupling is simple, consisting only of a diagonal connection.

Now consider the intra-variable coupling terms given by the stencils $\big[ \wh{\Phi}_{11} \big]_i$ in \eqref{eq:whPhi11-lem}/\eqref{eq:whPhi11-cor} and $\big[ \wh{\Phi}_{22} \big]_i$ in \eqref{eq:whPhi22-lem}/\eqref{eq:whPhi22-cor}.
The order-one components have a similar structure as in $\big[ \wh{\Phi}_{12} \big]_i$, $\big[ \wh{\Phi}_{21} \big]_i$, although the connections are not diagonal in this case.
Unlike the inter-variable stencils, $\big[ \wh{\Phi}_{11} \big]_i$ and $\big[ \wh{\Phi}_{22} \big]_i$ have order-zero connections.
%
Let us denote the order-zero connections in $\big[ \wh{\Phi}_{11} \big]_i$ and $\big[ \wh{\Phi}_{22} \big]_i$ as $\big[ \wt{\Phi}_{11} \big]_i$ and $\big[ \wt{\Phi}_{22} \big]_i$, respectively, and we note they have a special structure:
\begin{subequations} \label{eq:god-whPhi-adv}
\begin{align} 
\label{eq:god-whPhi-11}
\big[ \wt{\Phi}_{11} \big]_{i}
:=
[1] + c_i \frac{\delta t}{h} [0, -1, 1]
\equiv
\textrm{Godunov-discretization}
\bigg(
\frac{\partial}{\partial t} - c_0 \frac{\partial}{\partial x}
\bigg),
\\
\label{eq:god-whPhi-22}
\big[ \wt{\Phi}_{22} \big]_{i}
:=
[1] - c_i \frac{\delta t}{h} [-1, 1, 0]
\equiv
\textrm{Godunov-discretization}
\bigg(
\frac{\partial}{\partial t} + c_0 \frac{\partial}{\partial x}
\bigg).
\end{align}
\end{subequations}
That is, these order-zero stencils for the characteristic variables are equivalent to discretizing the above linear advection operators (assuming $c_0 > 0$) using the same Godunov methodology as used for the acoustics discretization in \cref{sec:acoustics-disc}; see \cite[Chapter 9.3]{LeVeque_2004} for their derivation.
%

\subsection{Block preconditioners}
\label{sec:acoustics-prec}

Now we specify the $2 \times 2$ block preconditioners $\wh{{\cal P}}^{-1} \approx \wh{{\cal A}}^{-1}$ applied in \eqref{eq:prec-update} to approximate the error in characteristic space.
Specifically, we consider block Jacobi (diagonal) and block Gauss--Seidel (lower triangular) preconditioners, given by
\begin{align} \label{eq:prec-def}
\wh{{\cal P}}_{D}^{-1}
=
\begin{bmatrix}
\wh{{\cal A}}_{11} & 0 \\
0 & \wh{{\cal A}}_{22}
\end{bmatrix}^{-1},
\quad
\textrm{and}
\quad
\wh{{\cal P}}_{L}^{-1}
=
\begin{bmatrix}
\wh{{\cal A}}_{11} & 0 \\
\wh{{\cal A}}_{21} & \wh{{\cal A}}_{22}
\end{bmatrix}^{-1},
\end{align}
respectively.
The dominant cost in applying either of these preconditioners is the inversion of the diagonal blocks $\wh{{\cal A}}_{ii}$.
Based on \cref{lem:god-char-stencils} and \cref{cor:god-Z-linearized} and the subsequent discussion, applying $\wh{{\cal A}}_{11}^{-1}$ and $\wh{{\cal A}}_{22}^{-1}$ corresponds approximately to space-time linear advection solves with wave-speeds $-c_0$ and $+c_0$.
The algorithm outlined in \cref{sec:acoustics-outer-iter} will be parallel-in-time provided these diagonal blocks are inverted parallel-in-time. 
We remark that, in practice, there is usually no need to exactly solve these inner problems: Typically in block preconditioning, one can replace exact inner solves with (sufficiently accurate) inexact solves without significantly   deteriorating convergence of the outer iteration.
For example, exact inverses of $\wh{{\cal A}}_{ii}$ could be replaced with approximate solves from any effective iterative parallel-in-time method.

Further motivated by the heuristic of using inexact inner solves, we introduce a pair of preconditioners based on approximating the diagonal blocks in \eqref{eq:prec-def}.
Specifically, we use the approximations in \eqref{eq:god-whPhi-adv} to consider
\begin{align} \label{eq:prec-approx-def}
\wt{{\cal P}}_{D}^{-1}
&=
\begin{bmatrix}
\wt{{\cal A}}_{11} & 0 \\
0 & \wt{{\cal A}}_{22}
\end{bmatrix}^{-1},
\textrm{ and }
\wt{{\cal P}}_{L}^{-1}
=
\begin{bmatrix}
\wt{{\cal A}}_{11} & 0 \\
\wh{{\cal A}}_{21} & \wt{{\cal A}}_{22}
\end{bmatrix}^{-1},
\textrm{ where }
\wt{{\cal A}}_{ii}
=
\begin{bmatrix}
I \\
-\wt{\Phi}_{ii} & I \\
& \ddots & \ddots \\
\end{bmatrix}.
\end{align}
Here the approximate time-stepping operators $\wt{\Phi}_{11} \approx \wh{\Phi}_{11}$ and $\wt{\Phi}_{22} \approx \wh{\Phi}_{22}$ are the linear advection Godunov discretizations in \eqref{eq:god-whPhi-adv}, and coincide with those of $\wh{\Phi}_{11}$ and $\wh{\Phi}_{22}$ in \eqref{eq:acoustics-god-Phi-hat-stencils} in cases of constant impedance.
Recall that these Godunov discretizations are consistent with the lowest-order terms in $\wh{\Phi}_{11}$ and $\wh{\Phi}_{22}$.
Moreover, our numerical tests consider approximately inverting $\wt{{\cal A}}_{ii}$ using MGRIT.
This overall preconditioning strategy, i.e., using $\wt{{\cal P}}$ with inexact inner inverses, is motivated by our existing parallel-in-time approaches for linear advection \cite{DeSterck_etal_2023_MOL,DeSterck_etal_2023_SL,DeSterck-etal-2023-nonlin-scalar}.

Numerical results are considered in the next section.
First, we note that while our solver outlined in \cref{alg:char-prec} is effective for many problems, it is not without limitations. 
For example, our numerical results indicate that its performance degrades when applied to problems with increasingly complicated impedance variation.
\begin{remark}[Schur-complement preconditioning]
\label{rem:schur-comp}
More general than the block preconditioners in \eqref{eq:prec-def} are those based on the Schur complements of $\wh{{\cal A}}$. 
For example, 
\begin{align} \label{eq:schur-block-prec}
\begin{bmatrix}
\wh{{\cal A}}_{11} & 0 \\
0 & \wc{{\cal S}}_{22}
\end{bmatrix}^{-1},
\quad
\begin{bmatrix}
\wh{{\cal A}}_{11} & 0 \\
\wh{{\cal A}}_{21} & \wc{{\cal S}}_{22}
\end{bmatrix}^{-1},
\quad
\wc{{\cal S}}_{22} \approx \wh{{\cal S}}_{22} :=
\wh{{\cal A}}_{22} 
-
\wh{{\cal A}}_{21} \wh{{\cal A}}_{11}^{-1} \wh{{\cal A}}_{12},
\end{align}
directly generalize those in \eqref{eq:prec-def}, with $\wh{{\cal S}}_{22}$ the (2,2) Schur complement of $\wh{{\cal A}}$ in \eqref{eq:A-wh-blocked}.
The effectiveness of the preconditioners \eqref{eq:schur-block-prec} is characterized by the quality of the preconditioned Schur complement, $\wc{{\cal S}}_{22}^{-1} \wh{{\cal S}}_{22}$, at least in the block lower triangular case \cite{Southworth-etal-2020-2x2}.
Notice that preconditioners \eqref{eq:schur-block-prec} reduce to those in \eqref{eq:prec-def} when $\wc{{\cal S}}_{22}$ is chosen as $\wh{{\cal A}}_{22}$, with preconditioned Schur complement, 
$
\wc{{\cal S}}_{22}^{-1} \wh{{\cal S}}_{22}
=
\wh{{\cal A}}_{22}^{-1}
\big(
\wh{{\cal A}}_{22} 
-
\wh{{\cal A}}_{21} \wh{{\cal A}}_{11}^{-1} \wh{{\cal A}}_{12}
\big)
=
I -
\big(
\wh{{\cal A}}_{22}^{-1} \wh{{\cal A}}_{21} 
\big)
\big(
\wh{{\cal A}}_{11}^{-1} \wh{{\cal A}}_{12}
\big)
$.
It is not too difficult to show that the Schur complement $\wh{{\cal S}}_{22}$ is a dense block lower triangular matrix, and that the choice of $\wc{{\cal S}}_{22} = \wh{{\cal A}}_{22}$ approximates $\wh{{\cal S}}_{22}$ by truncating everything below its block subdiagonal.
We do not theoretically analyze the impact of choosing $\wc{{\cal S}}_{22} = \wh{{\cal A}}_{22}$ further here.
%
In future work we hope to improve the robustness of our linear solver with respect to impedance variation by developing better Schur complement approximations.
\end{remark}
%
%

\subsection{Numerical results}
\label{sec:acoustics-num-res}

\begin{figure}[b!]
\centerline{
\includegraphics[width=0.35\textwidth]{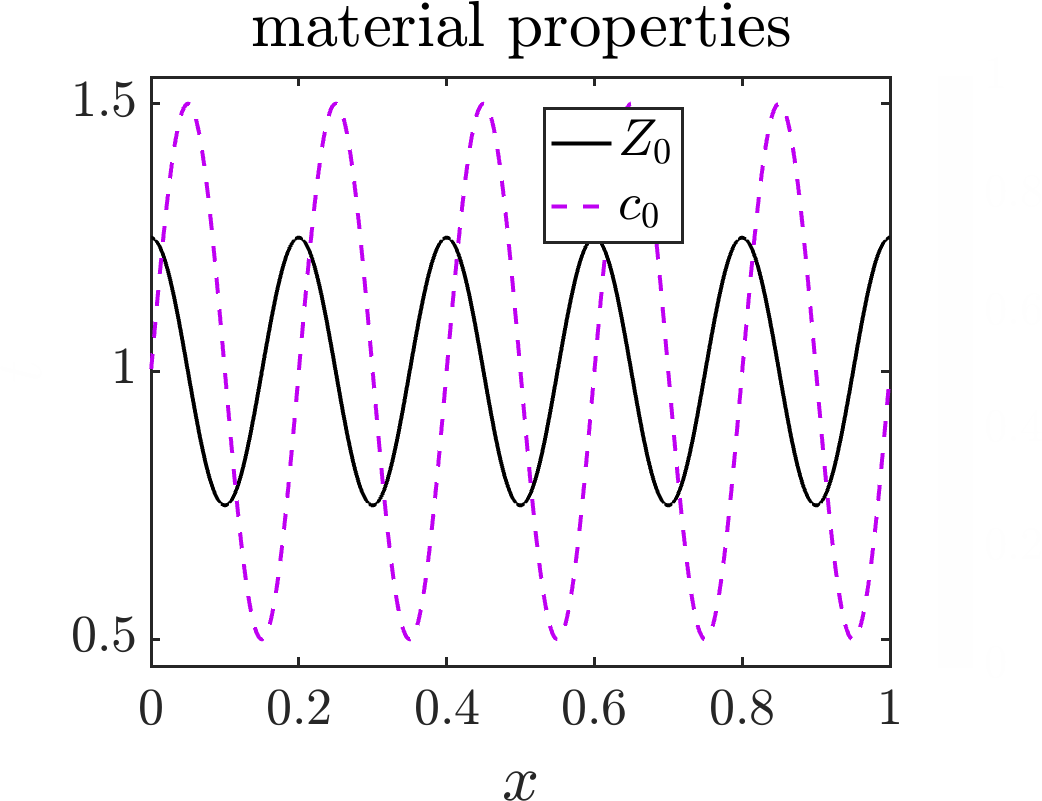}
\hspace{-5ex}
\includegraphics[width=0.35\textwidth]{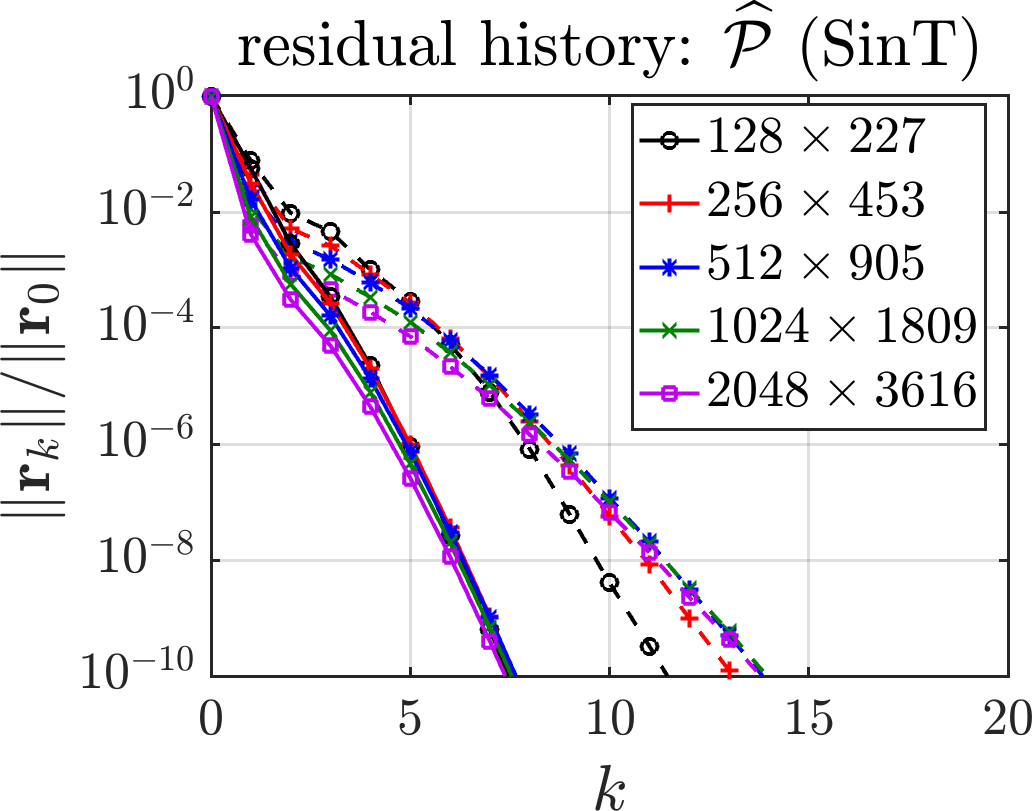}
\hspace{-1.5ex}
\includegraphics[width=0.35\textwidth]{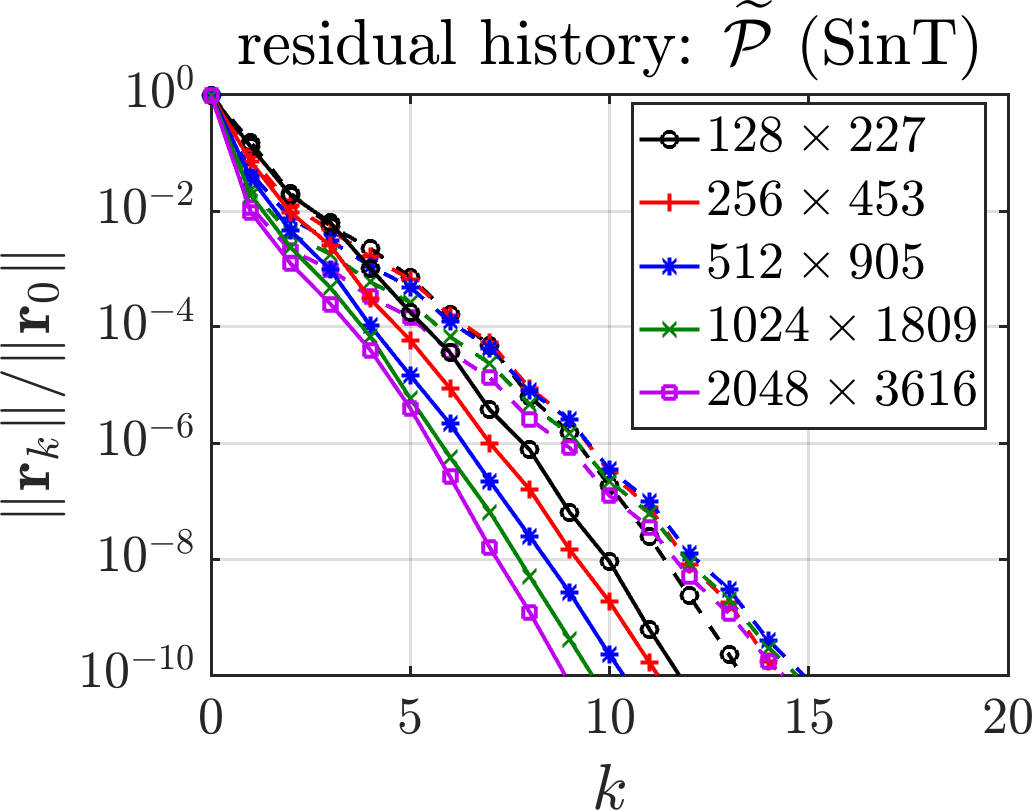}
}
\vspace{1ex}
\centerline{
\includegraphics[width=0.35\textwidth]{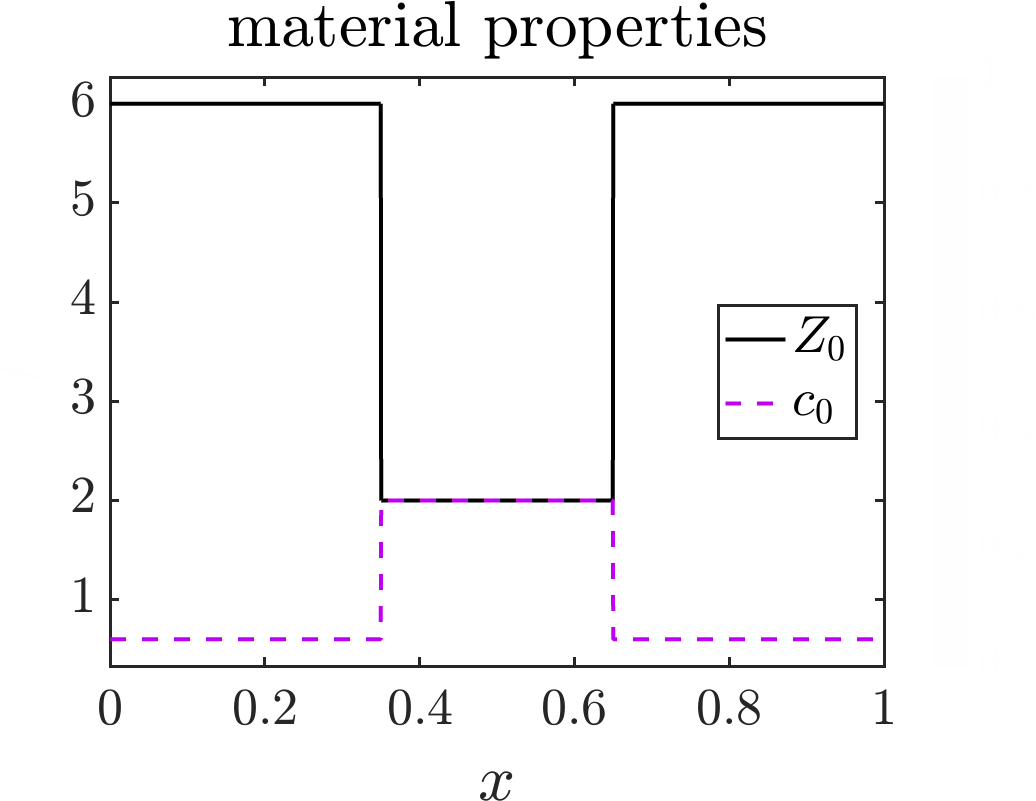}
\hspace{-5ex}
\includegraphics[width=0.35\textwidth]{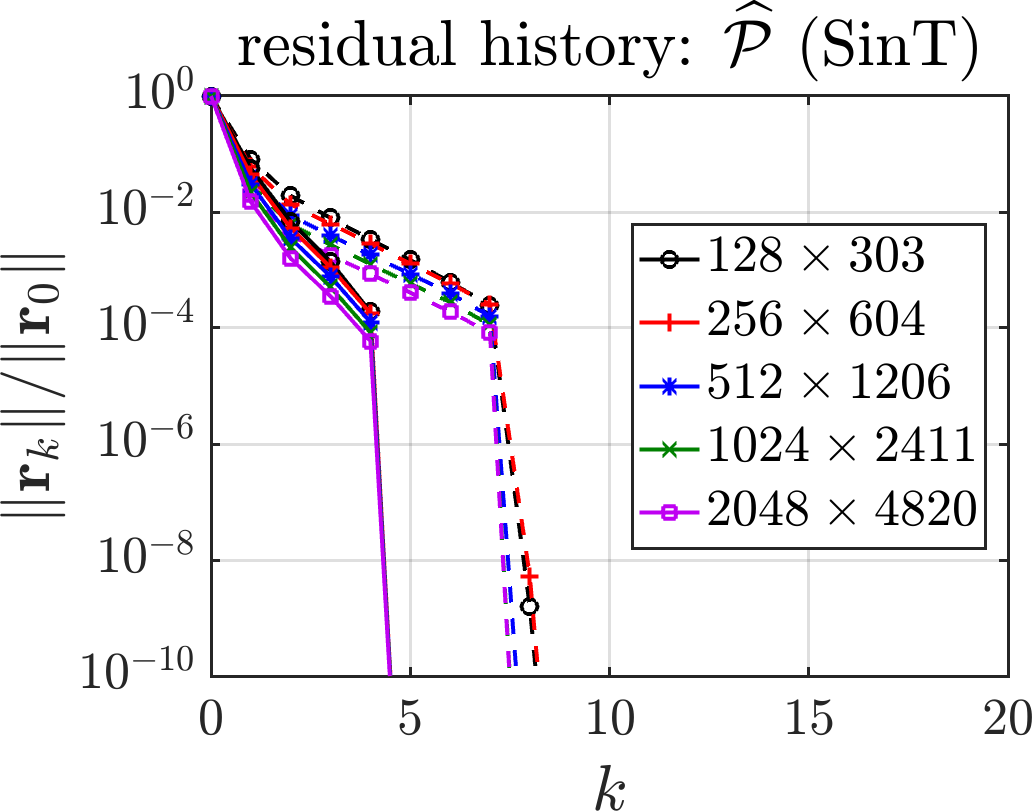}
\hspace{-1.5ex}
\includegraphics[width=0.35\textwidth]{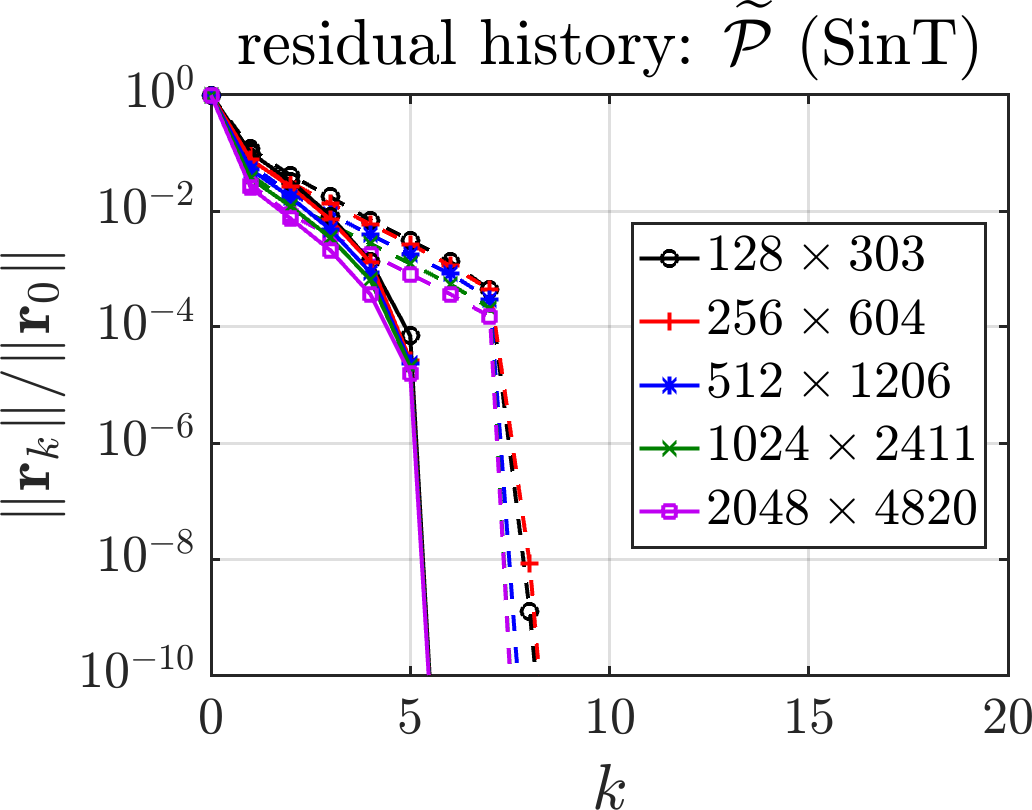}
}
\vspace{1ex}
\centerline{
\includegraphics[width=0.35\textwidth]{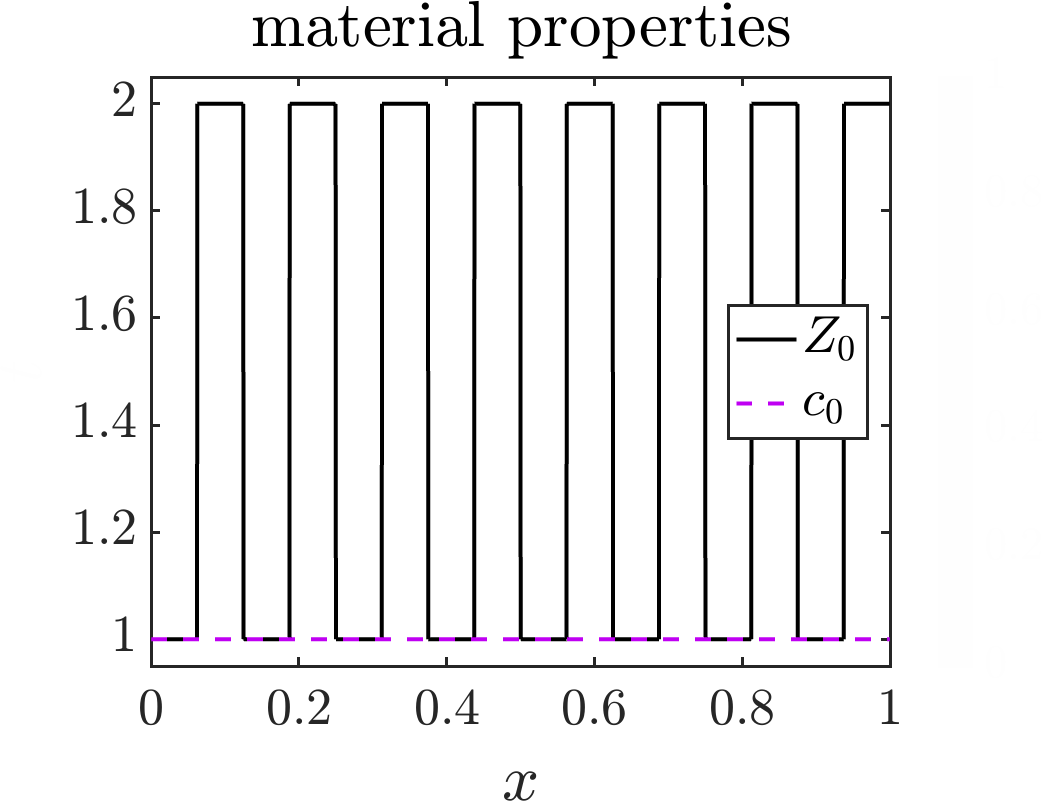}
\hspace{-5ex}
\includegraphics[width=0.35\textwidth]{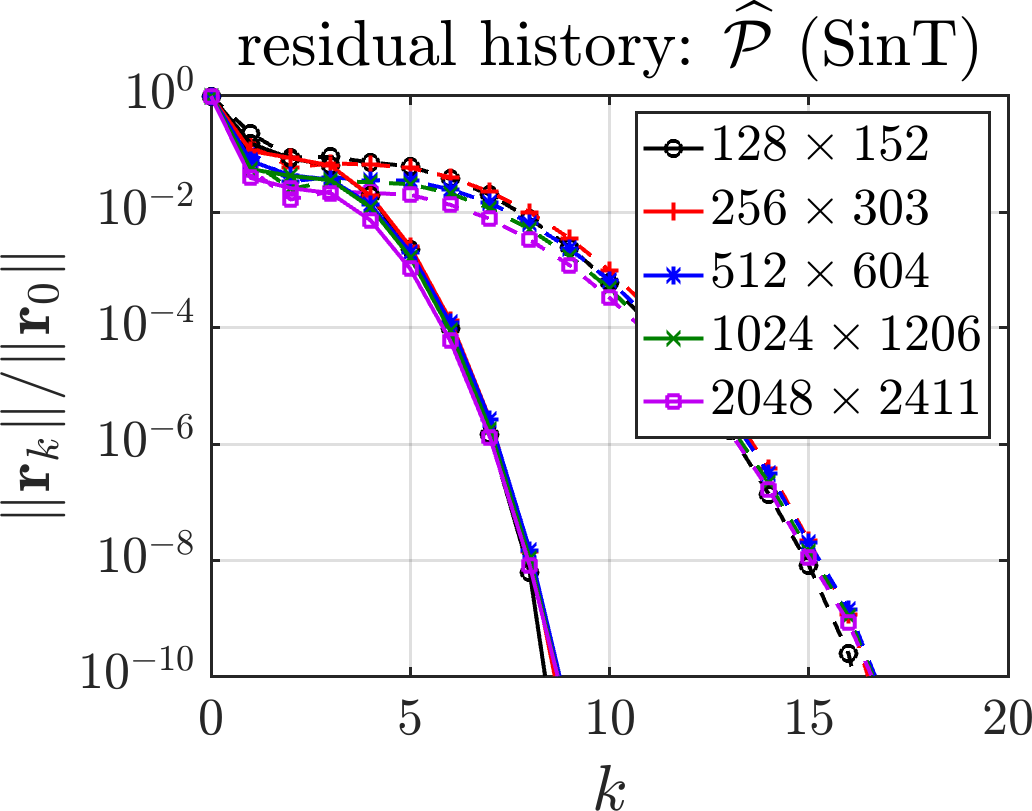}
\hspace{-1.5ex}
\includegraphics[width=0.35\textwidth]{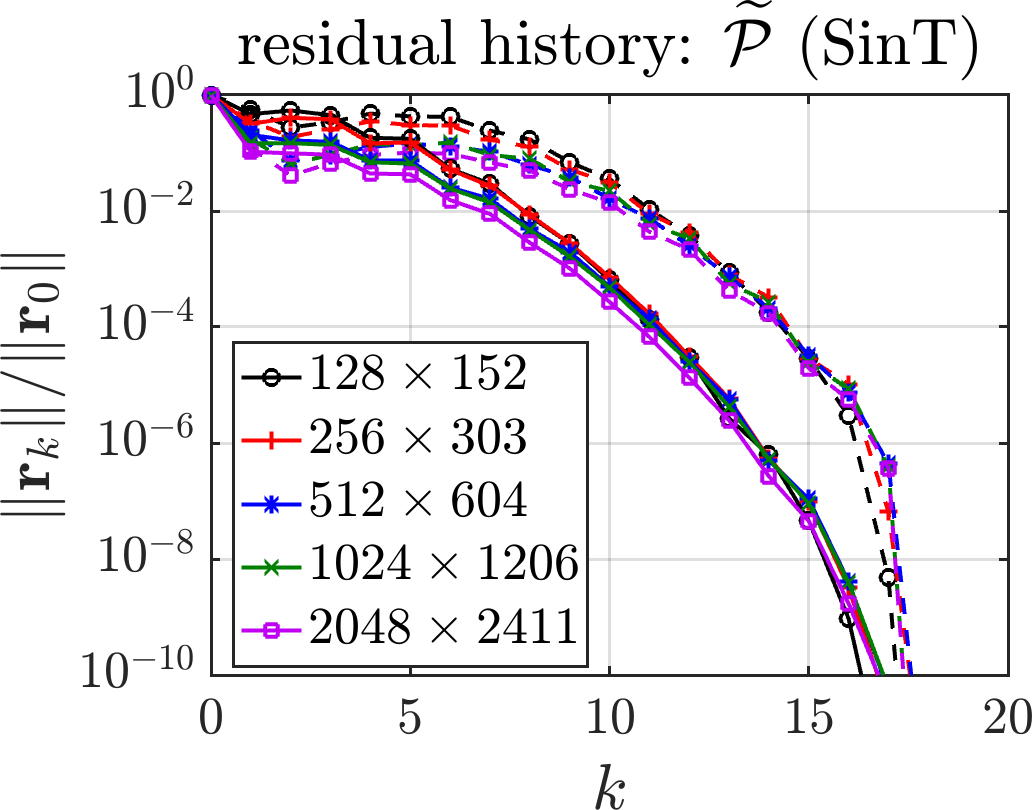}
}
\vspace{1ex}
\centerline{
\includegraphics[width=0.35\textwidth]{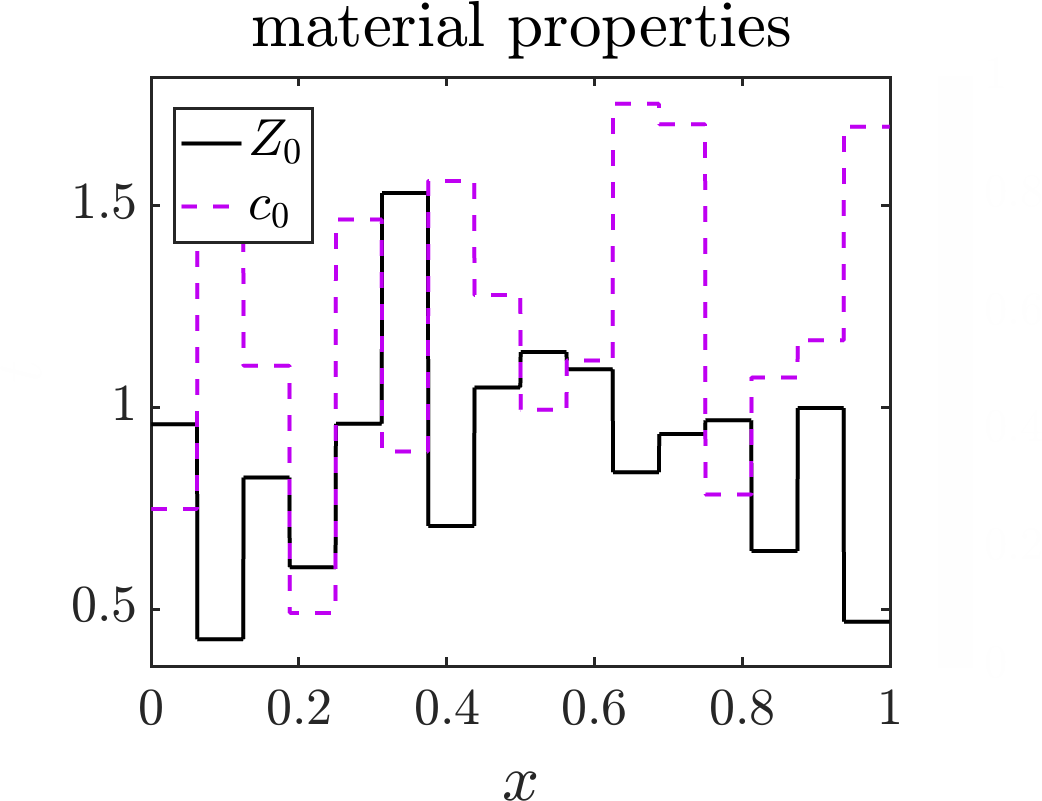}
\hspace{-5ex}
\includegraphics[width=0.35\textwidth]{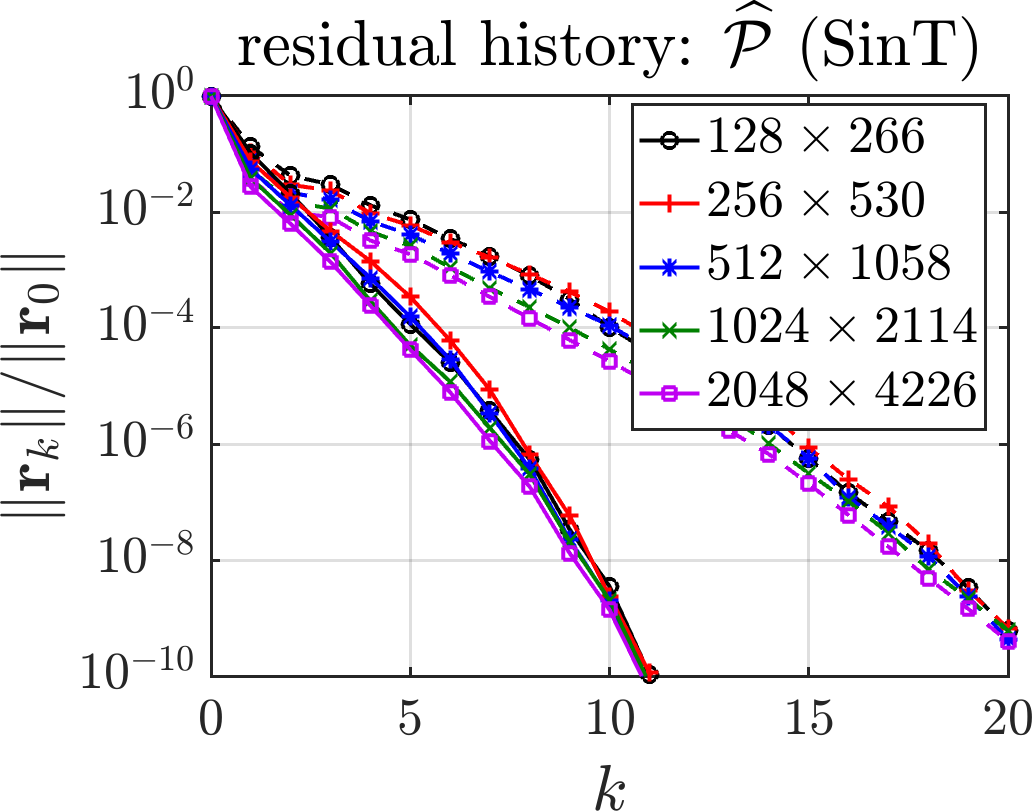}
\hspace{-1.5ex}
\includegraphics[width=0.35\textwidth]{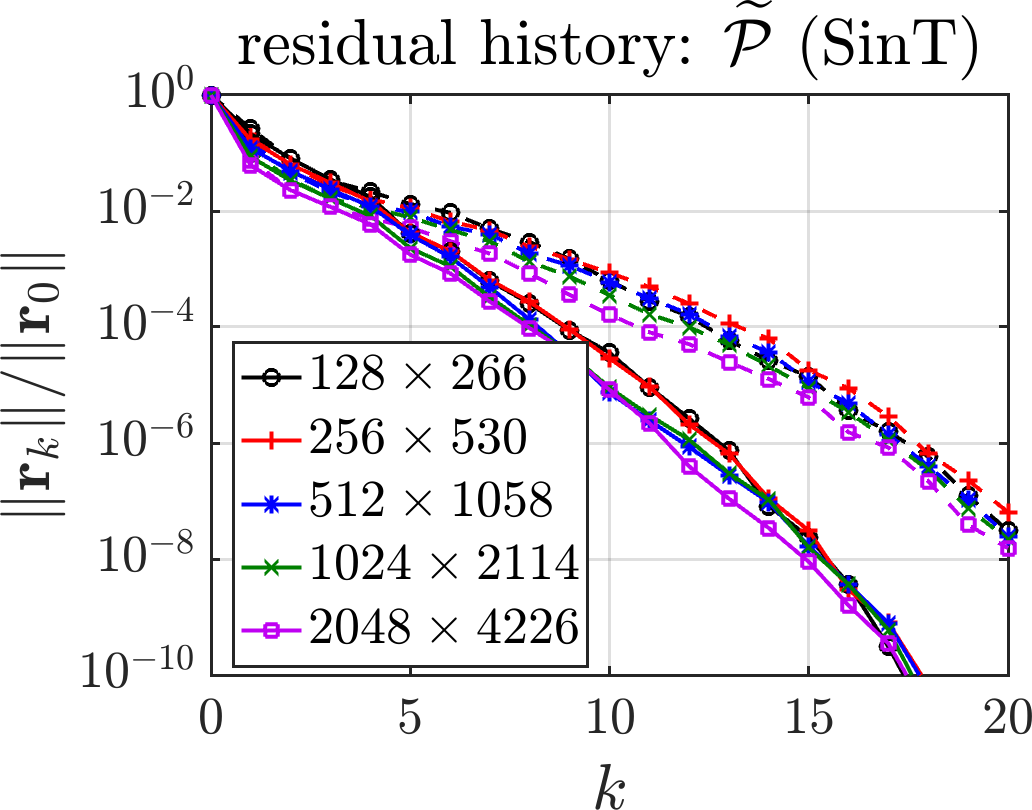}
}
\caption{\textbf{Left:} Material properties \eqref{eq:mat-param-2}, \eqref{eq:mat-param-3}, \eqref{eq:mat-param-4}, \eqref{eq:mat-param-5} ordered from top to bottom. 
\textbf{Middle:} Residual history using preconditioners $\wh{{\cal P}}_L$ (solid lines) and $\wh{{\cal P}}_D$ (dashed lines) in \eqref{eq:prec-def}.
\textbf{Right:} Residual history using preconditioners $\wt{{\cal P}}_L$ (solid lines) and $\wt{{\cal P}}_D$ (dashed lines) in \eqref{eq:prec-approx-def} which are based on \underline{approximate} diagonal blocks.
All preconditioners are applied exactly via sequential time-stepping (SinT).
Legend entries correspond to space-time mesh resolutions of $n_x \times n_t$.
}
\label{fig:acoustic-prec-exact}
\end{figure}

We consider solving \eqref{eq:acoustic} on the space-time domain $(x,t) \in (0, 1) \times (0, 1]$, and subject to the initial conditions
\begin{align}
u(x, 0) = 0,
\quad
p(x, 0) 
=
\begin{cases}
\tfrac{1}{4} \big[ 7 - 3 \cos ( 10 \pi x - 4 \pi  ) \big], 
\quad &x \in (0.4, 0.6), \\
1, 
\quad &\textrm{else}.
\end{cases}
\end{align}
We consider five different media, with sound speeds and impedances given by:
\begin{subequations} \label{eq:acoustics-mat-params-examples}
\begin{alignat}{2}
\label{eq:mat-param-1}
c_0(x) &= 1 + \tfrac{1}{2} \sin (10 \pi x), 
\quad 
&&Z_0(x) = 1, \\
\label{eq:mat-param-2}
c_0(x) &= 1 + \tfrac{1}{2} \sin (10 \pi x), 
\quad 
&&Z_0(x) =1 + \tfrac{1}{4} \cos(10 \pi x), \\
\label{eq:mat-param-3}
c_0(x) &= 
\begin{cases}
2, \quad &x \in (0.35, 0.65), \\
0.6, \quad &\textrm{else},
\end{cases}
\quad\quad
&&Z_0(x) = 
\begin{cases}
2, \quad &x \in (0.35, 0.65), \\
6, \quad &\textrm{else},
\end{cases}
\\
\label{eq:mat-param-4}
c_0(x) &= 
1,
\quad\quad
&&Z_0(x) = 
\begin{cases}
1, \quad & \textrm{mod} ( \lfloor 16 x \rfloor, 2 ) = 0, \\
2, \quad & \textrm{else}. 
\end{cases}
\\
\label{eq:mat-param-5}
c_0(x) &= 
\textrm{rand}_1 ( \lfloor 16 x \rfloor ),
\quad\quad
&&Z_0(x) = 
\textrm{rand}_2 ( \lfloor 16 x \rfloor ).
\end{alignat}
\end{subequations}
Plots of the material parameters and space-time contours of the associated pressures $p(x,t)$ are shown in the left and middle columns of \cref{fig:acoustic-prec-inexact}, respectively---note that the equations in \eqref{eq:mat-param-5} denote a randomly layered material with 16 layers.
For \eqref{eq:mat-param-1} the characteristic variables decouple globally since $Z_0$ is constant, while for all other material parameters they do not.
The parameters \eqref{eq:mat-param-1}--\eqref{eq:mat-param-3} are used in \cite{Bale_etal_2003}, and \eqref{eq:mat-param-4} and \eqref{eq:mat-param-5} are simplified parameters from \cite{LeVeque_2004} and \cite{Fogarty-LeVeque-1999}, respectively. 

Shown in \cref{fig:acoustic-prec-exact,fig:acoustic-prec-inexact} are two sets of tests corresponding to the exact, sequential-in-time, application of the preconditioners $\wh{{\cal P}}$ and $\wt{{\cal P}}$ in \eqref{eq:prec-def} and \eqref{eq:prec-approx-def}, and their inexact, parallel-in-time application via one MGRIT V-cycle, respectively.
In our tests, the time-step size is chosen such that $\delta t = 0.85 h / \max_x c_0(x) $.
The initial space-time iterate $\bm{q}_0$ is chosen with entries drawn randomly from a standard normal distribution, and the convergence metric reported is the 2-norm of the space-time residual of \eqref{eq:godunov-all-at-once} relative to its initial value.
The coarsening factor used to create the CF-splitting for the F-relaxation in Line \ref{ln:lin-F-relax} of \cref{alg:char-prec} is $m = 8$.

Consider first \cref{fig:acoustic-prec-exact} to understand best-case convergence of the algorithm wherein the preconditioners $\wh{\cal {P}}$ and $\wt{\cal {P}}$ are applied exactly.
The simplest material parameters in \eqref{eq:mat-param-1} are omitted from these tests because the iteration converges to the exact solution in a single iteration due to the characteristic variables globally decoupling.
Remarkably, we see that in almost all cases the convergence rate seems  independent of problem size.
Unsurprisingly, we see that in all cases the lower triangular preconditioner (solid) converges in fewer iterations than the diagonal preconditioner (dashed).
We also note there is a mild trend of convergence deteriorating when moving from preconditioners $\wh{\cal {P}}$ using the true diagonal blocks to preconditioners $\wt{\cal {P}}$ using approximate diagonal blocks.
This deterioration is most significant for the lower triangular preconditioners in the bottom two rows, corresponding to material parameters \eqref{eq:mat-param-4}, \eqref{eq:mat-param-5}.

\begin{figure}
\centerline{
\includegraphics[width=0.35\textwidth]{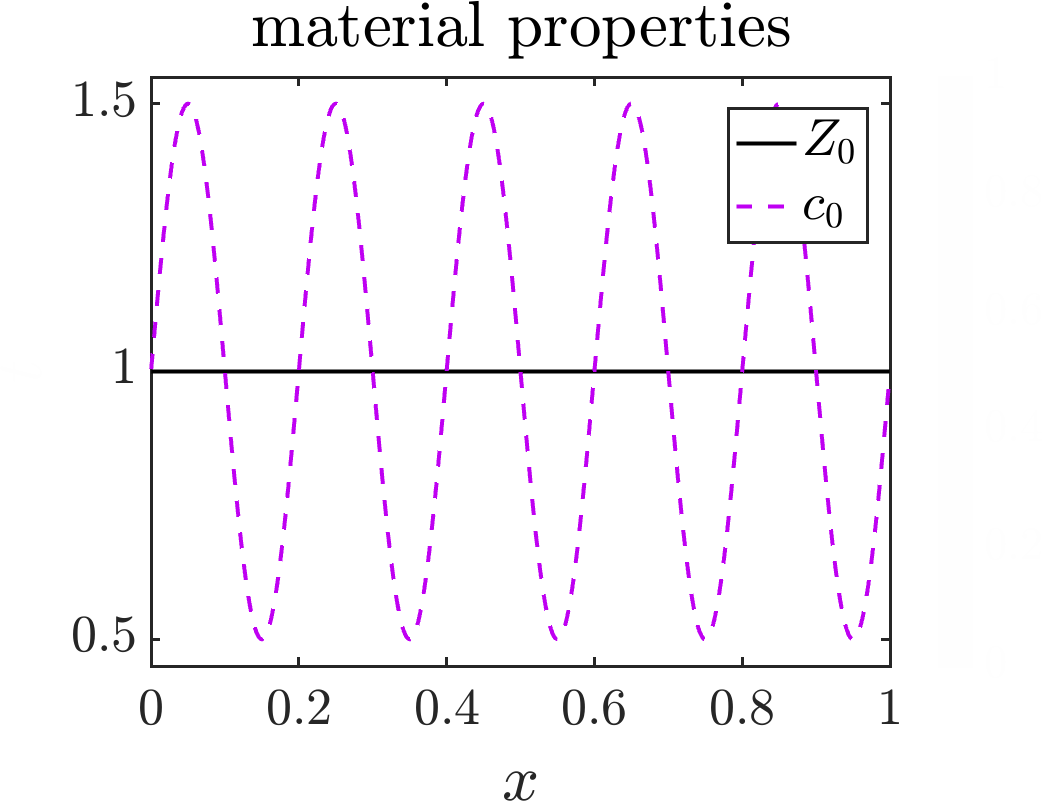}
\hspace{-5ex}
\includegraphics[width=0.35\textwidth]{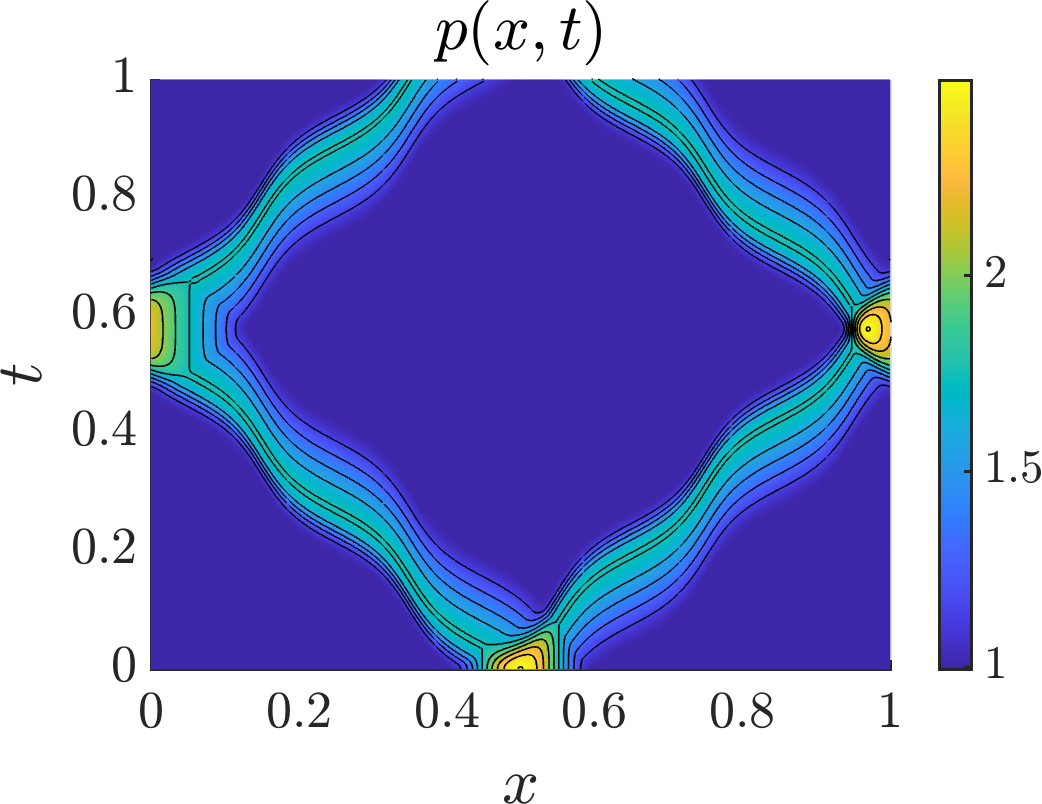}
\hspace{-1ex}
\includegraphics[width=0.35\textwidth]{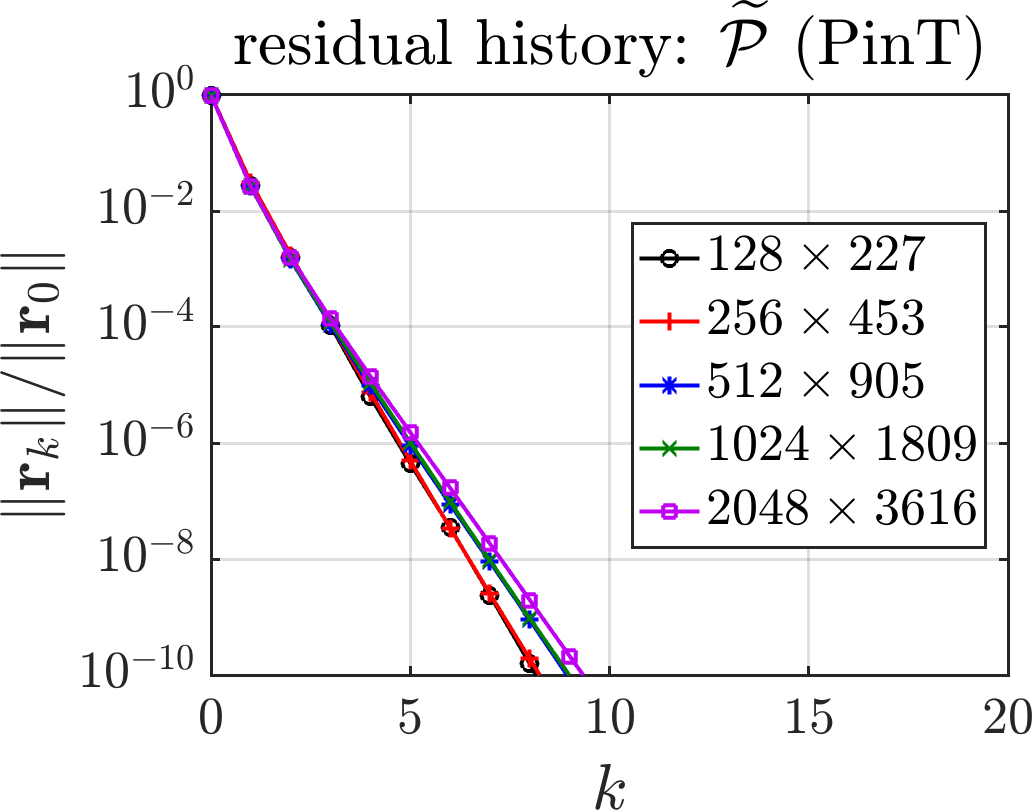}
}
\vspace{1ex}
\centerline{
\includegraphics[width=0.35\textwidth]{figures/acoustics-ex2-mat-pa}
\hspace{-5ex}
\includegraphics[width=0.35\textwidth]{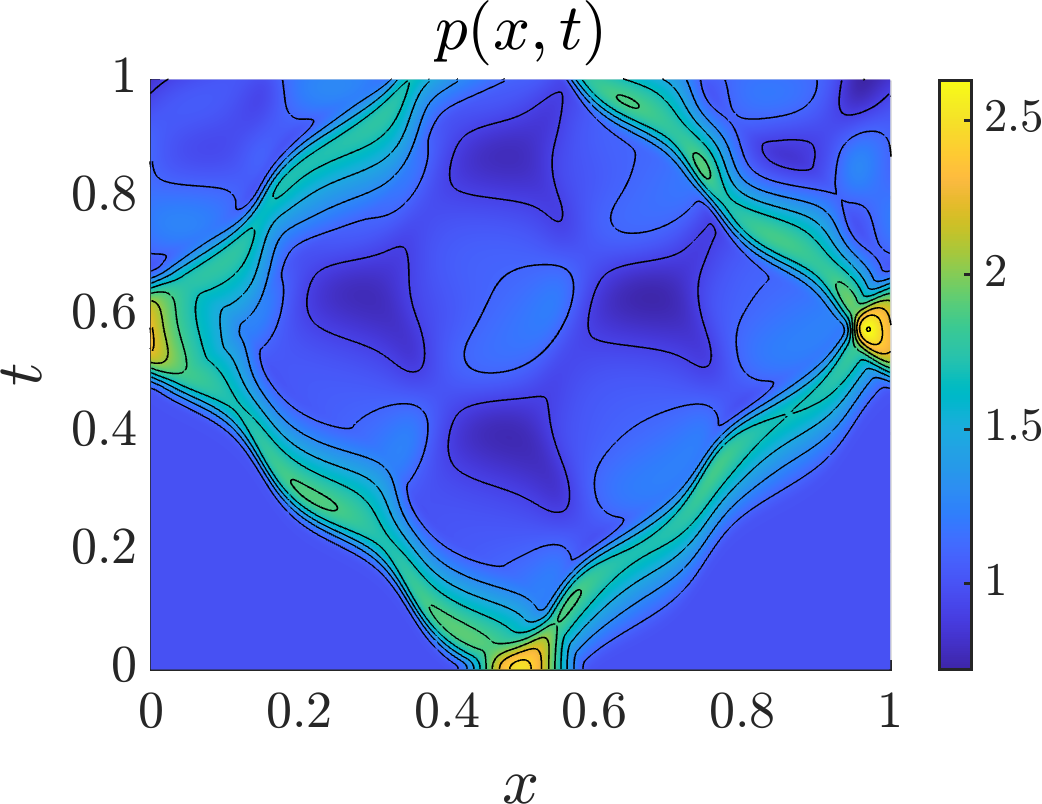}
\hspace{-1ex}
\includegraphics[width=0.35\textwidth]{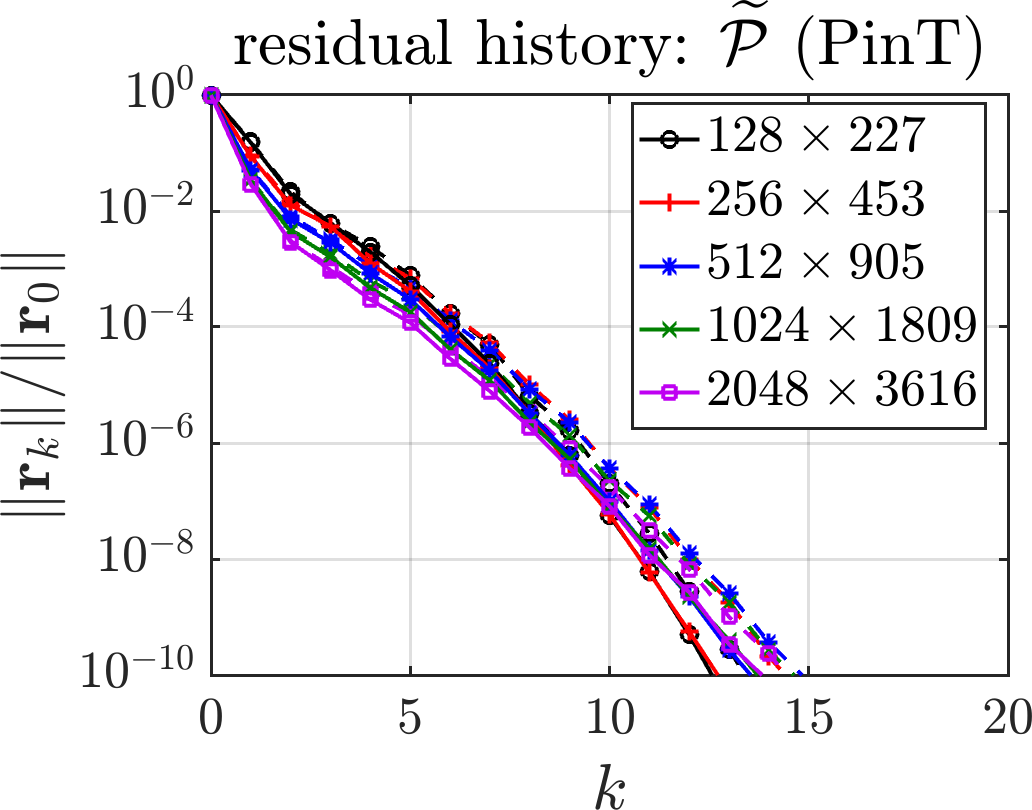}
}
\vspace{1ex}
\centerline{
\includegraphics[width=0.35\textwidth]{figures/acoustics-ex3-mat-pa}
\hspace{-5ex}
\includegraphics[width=0.35\textwidth]{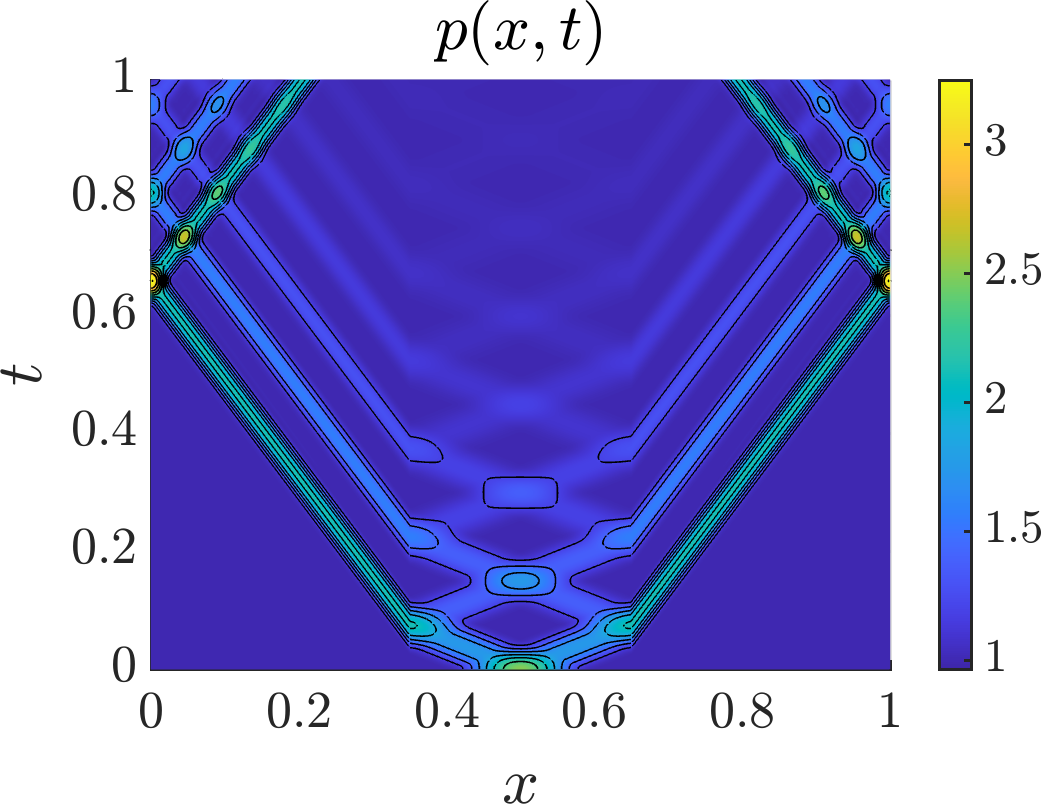}
\hspace{-1ex}
\includegraphics[width=0.35\textwidth]{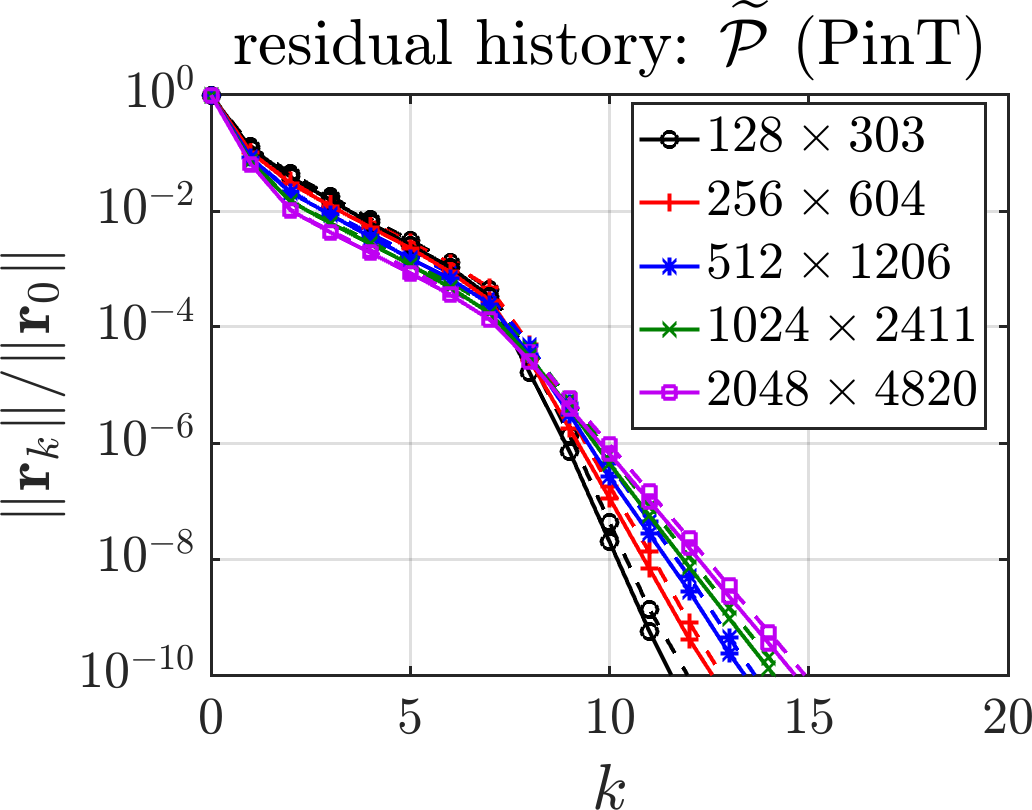}
}
\vspace{1ex}
\centerline{
\includegraphics[width=0.35\textwidth]{figures/acoustics-ex5-mat-pa}
\hspace{-5ex}
\includegraphics[width=0.35\textwidth]{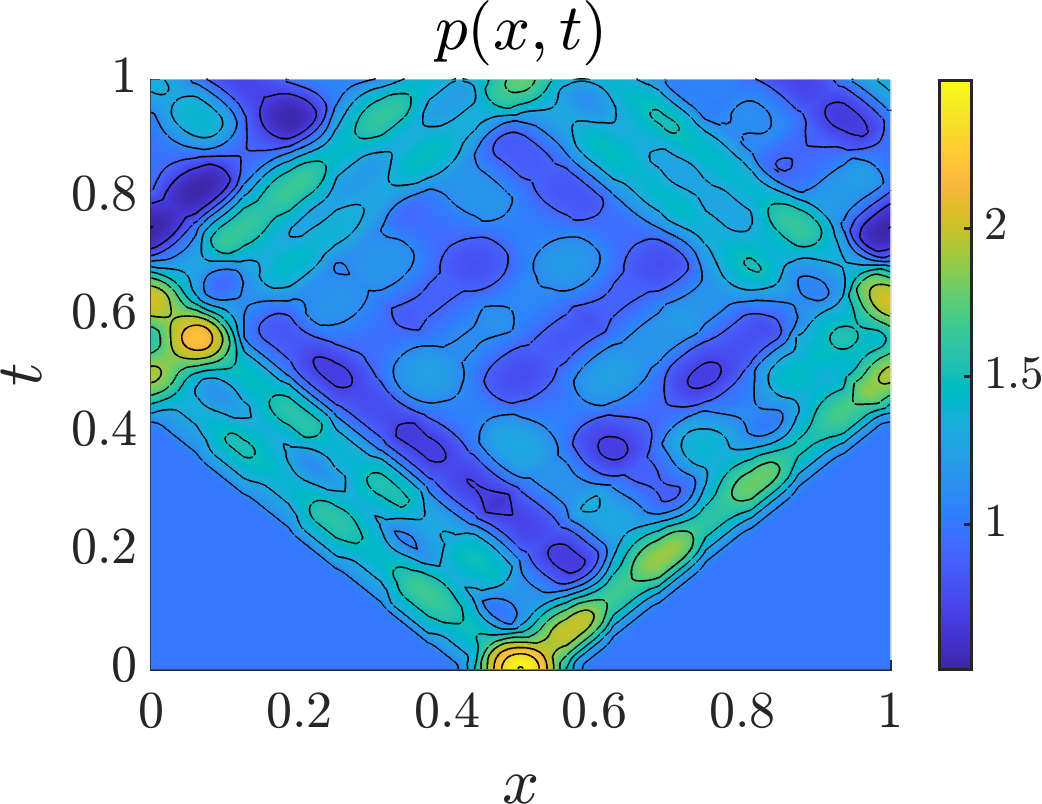}
\hspace{-1ex}
\includegraphics[width=0.35\textwidth]{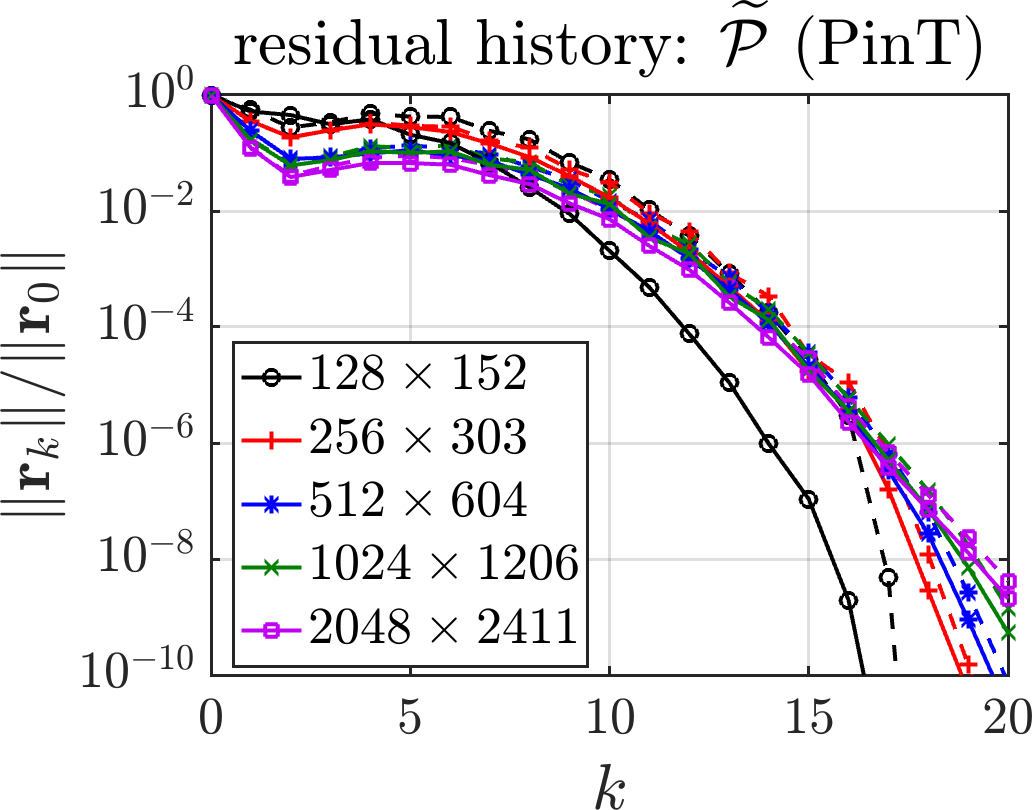}
}
\vspace{1ex}
\centerline{
\includegraphics[width=0.35\textwidth]{figures/acoustics-ex6-mat-pa}
\hspace{-5ex}
\includegraphics[width=0.35\textwidth]{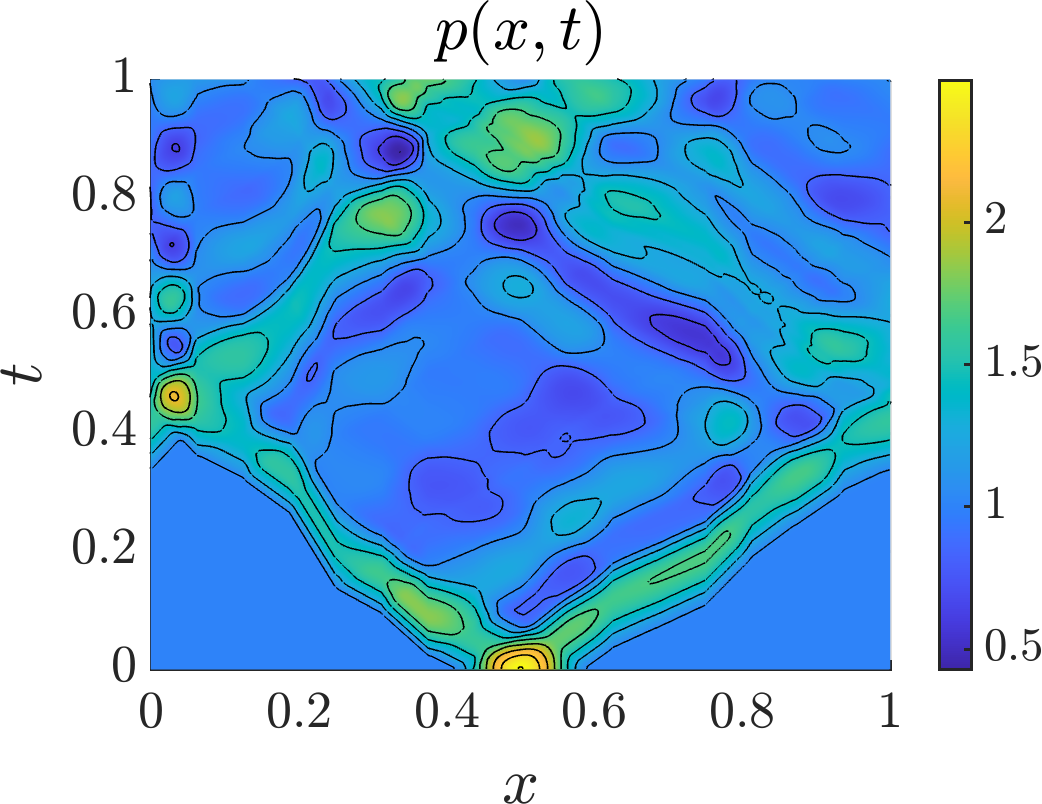}
\hspace{-1ex}
\includegraphics[width=0.35\textwidth]{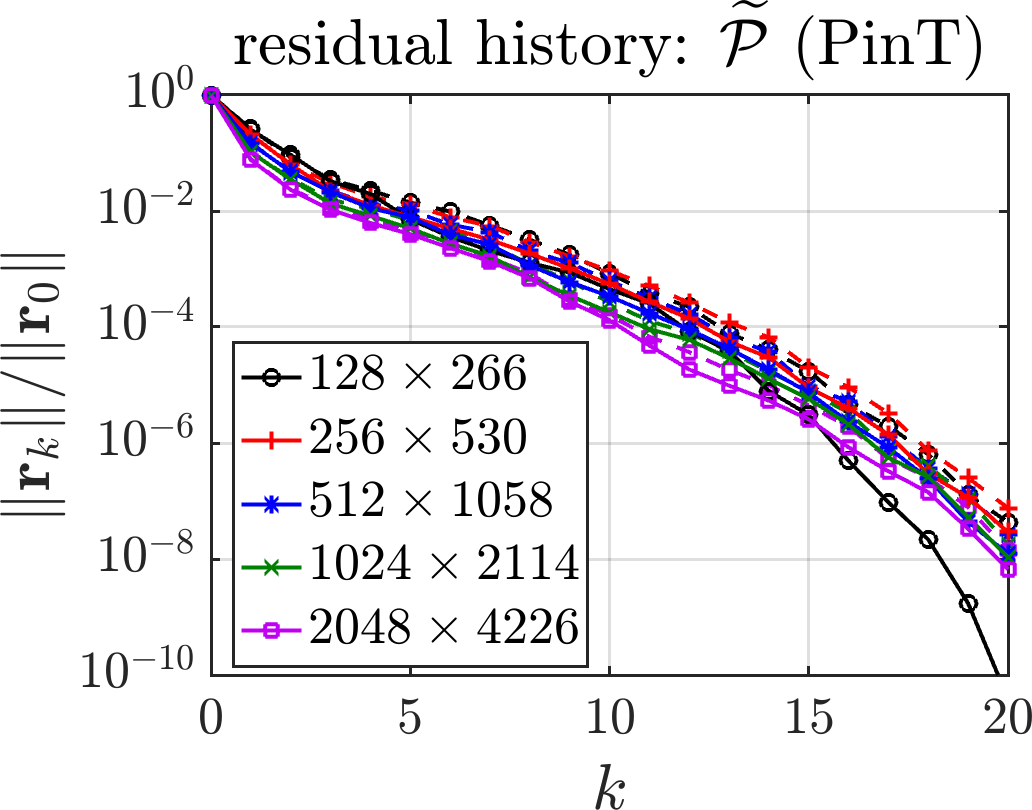}
}
\caption{Test problems for the acoustics equations \eqref{eq:acoustic}.
Each of the five rows in the figure corresponds to each of the five material parameters in \eqref{eq:acoustics-mat-params-examples}.
\textbf{Left}: Material parameters \eqref{eq:acoustics-mat-params-examples}. 
\textbf{Middle}: Space-time contours of the resulting pressure component of the solution vector; note these contours correspond to solves using $n_x = 2048$.
\textbf{Right}: Residual histories when using preconditioners $\wt{{\cal P}}$ in \eqref{eq:prec-approx-def} with diagonal blocks approximately inverted with a single MGRIT iteration (PinT). 
Dashed lines correspond to $\wt{{\cal P}}_D$ and solid lines correspond to $\wt{{\cal P}}_L$.
Legend entries correspond to space-time mesh resolutions of $n_x \times n_t$.}
\label{fig:acoustic-prec-inexact}
\end{figure}

Now let us turn our attention to \cref{fig:acoustic-prec-inexact}, corresponding to the inexact, parallel-in-time application of the approximate preconditioners $\wt{{\cal P}}$ in \eqref{eq:prec-approx-def} that employ approximate diagonal blocks $\wt{{\cal A}}_{ii}$.
In these tests, we apply a single MGRIT V-cycle to approximately invert the Godunov discretizations $\wt{{\cal A}}_{11}$ and $\wt{{\cal A}}_{22}$.
MGRIT coarsens by a factor of $m = 8$ on each level, and continues coarsening until fewer than two points in time would result. 
Each V-cycle uses F-relaxation on the finest level, and FCF-relaxation on all coarser levels.
Each V-cycle is initialized with the right-hand side vector.
Let $l \in \mathbb{N}_0$ denote the MGRIT level index, and $\wt{\Phi}_{ii}^{m^{l} \delta t}$ the time-stepping operator on level $l$.
Then on coarse levels $l > 0$ we use a modified semi-Lagrangian discretization \cite{DeSterck_etal_2023_MOL,DeSterck_etal_2023_SL} 
\begin{align} \label{eq:Phi-kk-multilevel}
\wt{\Phi}_{ii}^{m^{l} \delta t}
=
\big(
I - \diag( \bm{\gamma}_{ii} ) {\cal D}_2
\big)^{-1}
{\cal S}_{ii}^{m^{l} \delta t}
\approx
\prod \limits_{n = 0}^{m-1}
\wt{\Phi}_{ii}^{m^{l-1} \delta t}
\end{align}
where ${\cal S}_{ii}^{m^{l \delta t}}$ is a first-order accurate semi-Lagrangian discretization on grid $l$ of linear advection equations with wave-speeds $-c_0$ and $+c_0$ for $i = 1$ and $i = 2$, respectively.
The vector $\bm{\gamma}_{ii} \in \mathbb{R}^{n_x}$ is chosen as in \cite{DeSterck_etal_2023_SL} to minimize differences in truncation error of ${\cal S}_{ii}^{m^l \delta t}$ and $m$ applications of $\wt{\Phi}_{ii}^{m^{l-1} \delta t}$; the specific coefficients are determined via a simple analysis of the discretizations involved (see our code for specific formulae).
The matrix ${\cal D}_2 \in \mathbb{R}^{n_x}$ is the standard second-order accurate Laplacian, $ \big[ {\cal D}_2 \big]_{i} 
=
\left[1, \, -2, \, 1 \right] /h^2$, and the action of the inverse $\big(
I - \diag( \bm{\gamma}_{ii} ) {\cal D}_2
\big)^{-1}$ in \eqref{eq:Phi-kk-multilevel} is approximated with GMRES initialized with a zero vector and iterated until the relative residual is 0.01, or a maximum of 10 iteration is performed.

Considering \cref{fig:acoustic-prec-inexact}, we see that the convergence rate of the solver is independent of the mesh resolution.
We note that convergence is remarkably fast given the complicated nature of these problems, especially since only a single MGRIT V-cycle is used per inner iteration.
Comparing the right column of \cref{fig:acoustic-prec-inexact} with that in \cref{fig:acoustic-prec-exact}, generally speaking, there has been some degradation in convergence when moving from the exact to inexact application of $\wt{{\cal P}}^{-1}$. In particular, the lower triangular preconditioner $\wt{{\cal P}}_{L}$ loses its advantage over the diagonal preconditioner $\wt{{\cal P}}_{D}$ that is present in \cref{fig:acoustic-prec-exact} where exact inversions were used. 
Of course, the advantage of $\wt{{\cal P}}_{L}$ over $\wt{{\cal P}}_{D}$ could be regained through more accurate inner solves, e.g., using two MGRIT V-cycles instead of one; however, more accurate inner solves come at an increased computational cost, so it is unclear whether more accurate inner solves would be beneficial overall.
A more definitive answer to these questions would arise from running these solvers on a parallel machine and comparing their time-to-solution. 
We leave such parallel tests to future work.
It is also worth pointing out that in these tests we are significantly oversolving with respect to the discretization error.
%

\section{Conservative nonlinear systems: Introduction and background}
\label{sec:cons}

We now consider one-dimensional nonlinear conservation law systems of the form
\begin{align} \tag{cons-law}
\label{eq:cons}
\frac{\partial \bm{q}}{\partial t} + \frac{\partial \bm{f}(\bm{q})}{\partial x} = \bm{0},
\quad
(x, t) \in \Omega \times (0, T],
\quad
\bm{q}(x, 0) = \bm{q}_{\textrm{init}}(x),
\end{align}
with solution $\bm{q} = \bm{q}(x, t) = \big( q^1, \ldots, q^{\ell} \big)^\top \in \mathbb{R}^{\ell}$, and flux function $\bm{f} = \big( f^1, \ldots, f^{\ell} \big)^\top \in \mathbb{R}^{{\ell}}$.
The flux Jacobian $A(\bm{q}) := \bm{f}'(\bm{q}) \in \mathbb{R}^{\ell \times \ell}$ is assumed to have a full set of eigenvectors and real eigenvalues throughout the domain such that \eqref{eq:cons} is hyperbolic.
The flux Jacobian satisfies the decomposition $A = R \Lambda R^{-1}$, where $R = [\bm{r}^1 \, \ldots \, \bm{r}^{\ell}]$ with right eigenvectors $\bm{r}^{k} \in \mathbb{R}^{\ell}$, and $\Lambda = \diag(\lambda^1, \ldots, \lambda^{\ell})$, with eigenvalues ordered from smallest to largest, $\lambda^1 \leq \lambda^2 \ldots \leq \lambda^{\ell}$.

Numerical results later in \cref{sec:cons-num-res} consider two equation systems of the form \eqref{eq:cons}.
The first such system is the 1D shallow water equations \cite[Chapter 13]{LeVeque_2004}:
\begin{align} \tag{shallow-water}
\label{eq:SWE}
\frac{\partial }{\partial t}
\begin{bmatrix}
h \\
h u 
\end{bmatrix}
+
\frac{\partial }{\partial x}
\begin{bmatrix}
h u \\
h u^2 + \tfrac{1}{2} g h^2
\end{bmatrix}
=
\bm{0},
\quad
g \equiv \textrm{constant}.
\end{align}
Here, $h$ is the water height, and $u$ is its velocity, and we set $g = 1$.
The second model problem is the 1D Euler equations of gas dynamics \cite[Chapter 14]{LeVeque_2004}:
\begin{align} \tag{Euler}
\label{eq:euler}
\frac{\partial }{\partial t}
\begin{bmatrix}
\rho \\
\rho u \\
E
\end{bmatrix}
+
\frac{\partial }{\partial x}
\begin{bmatrix}
\rho u \\
\rho u^2 + p \\
(E + p) u
\end{bmatrix}
=
\bm{0},
\end{align}
closed with the equation of state for an ideal polytropic gas,
$
E = \frac{p}{\gamma - 1} + \frac{1}{2} \rho u^2, \, \gamma \equiv \textrm{constant}.
$
Here $\rho$ is the gas density, $u$ its velocity, $E$ its total energy, $p$ its pressure, and we set $\gamma = 7/5$.
For completeness, the the decompositions $A = R \Lambda R^{-1}$ for \eqref{eq:SWE} and \eqref{eq:euler} can be found in Supplementary Materials Section \ref{SMsec:flux-jacobians}.

\subsection{Finite-volume discretization}

The system \eqref{eq:cons} is discretized in space on a mesh of $n_x$ FV cells of equal width $h$, with the $i$th such cell being $x \in [x_{i - 1/2}, x_{i+1/2}]$, $i \in \{ 1, \ldots, n_x \}$.
The temporal domain $t \in [0, T]$ is discretized with a grid of $n_t$ points, $\{ t_n \}_{n = 0}^{n_t -1 }$, $t_n = n \delta t$, with constant spacing $\delta t$.
We consider a first-order accurate FV method paired with forward Euler integration in time.
Ultimately, this discretization results in the one-step scheme  
\begin{align} \label{eq:cons-Phi-def}
\bm{q}_i^{n+1}
 =
 \bm{q}_i^{n}
 -
 \frac{\delta t }{h} 
 \Big[ \wh{\bm{f}}_{i+1/2} (\bm{q}^n) - \wh{\bm{f}}_{i-1/2} (\bm{q}^n) \Big]
 =: 
\big( \Phi( \bm{q}^n ) \big)_i
\in
\mathbb{R}^{\ell},
\end{align}
with $ \bm{q}_i^{n}$ approximating the average of the solution of \eqref{eq:cons} in cell $i$ at time $t_n$.
More generally, we use the notation that if $\bm{\zeta} \in \mathbb{R}^{\ell n_x}$ is a stacked vector associated with the spatial mesh, then $\bm{\zeta}_i \in \mathbb{R}^{\ell}$ denotes a vector with the components of $\bm{\zeta}$ associated with the $i$th FV cell containing $\ell$ variables.
In \eqref{eq:cons-Phi-def}, $\wh{\bm{f}}_{i + 1/2} (\bm{q})$ is the numerical flux function, and in this work is the so-called Roe flux:
\begin{align} \label{eq:Roe-simple}
\wh{\bm{f}}_{i+1/2} (\bm{q}) = 
\frac{1}{2}
\Big[
\big(
\bm{f}(\bm{q}_{i+1}) + \bm{f}(\bm{q}_{i})
\big)
-
\big| A^{*}_{i+1/2}( \bm{q} ) \big|
\big(
\bm{q}_{i+1} - \bm{q}_{i}
\big)
\Big] 
\in 
\mathbb{R}^{\ell}.
\end{align}
Here $A^*_{i+1/2} = A^*_{i+1/2} ( \bm{q}_i, \bm{q}_{i+1}  )$ is a linearization of the flux Jacobian $A( \bm{q} ) := \bm{f}'( \bm{q} ) $ about a specially chosen average between $\bm{q}_{i}$ and $\bm{q}_{i+1}$;
this flux originated in \cite{Roe1981}, and we point the reader to LeVeque \cite[Chapter 15.3]{LeVeque_2004} for details on constructing it for our model problems, see \eqref{eq:SWE} and \eqref{eq:euler}.
This matrix has the eigen-decomposition $A^*_{i+1/2} = R^*_{i+1/2} \Lambda^*_{i+1/2} ( R^*_{i+1/2} )^{-1}$, where $R^*_{i+1/2} = \big[ \bm{r}_{i+1/2}^{*,1} \ldots \bm{r}_{i+1/2}^{{*,\ell}} \big]$, and $\Lambda^*_{i+1/2} = \diag \big( \lambda_{i+1/2}^{*,1}, \ldots, \lambda_{i+1/2}^{*,\ell} \big)$.
Furthermore, $\big| A^*_{i+1/2} \big| = R^*_{i+1/2} \big| \Lambda^*_{i+1/2} \big| (R^*_{i+1/2})^{-1}$.
As such, \eqref{eq:Roe-simple} can also be written as:
\begin{align} \label{eq:Roe-2}
\wh{\bm{f}}_{i+1/2}
( \bm{q} )
= 
\frac{1}{2}
\Big[
\big(
\bm{f}(\bm{q}_{i+1}) + \bm{f}(\bm{q}_{i})
\big)
-
\sum \limits_{j = 1}^{\ell} \alpha_{i+1/2}^{*,j} \big| \lambda_{i+1/2}^{*,j} \big| \bm{r}_{i+1/2}^{*,j}
\Big],
\end{align}
with coefficients $\{ \alpha_{i+1/2}^{*,j} \}_{j = 1}^{\ell}$ satisfying $\bm{q}_{i+1/2}^{+} - \bm{q}_{i+1/2}^{-} = \sum_{j = 1}^{\ell} \alpha_{i+1/2}^{*,j} \bm{r}_{i+1/2}^{*,j}$. 
Note the Roe flux adds wave-specific dissipation.
We also employ ``Harten's Entropy fix'' \cite[p. 326]{LeVeque_2004}, whereby the absolute value $|\lambda^{*,j}_{i+1/2}|$ is replaced with the smoothed absolute value in \cite[(15.53)]{LeVeque_2004} using smoothing parameter $\delta = 10^{-6}$. To maintain a simple notation, we still write the smoothed absolute eigenvalues as $|\lambda^{*,j}_{i+1/2}|$.
%

\section{Conservative nonlinear systems: Parallel-in-time solution}
\label{sec:cons-pint}

In this section, we consider the parallel-in-time solution of the fully discretized problem 
\begin{align}
\bm{q}^{n+1} \label{eq:cons-ts}
=
\Phi 
(\bm{q}^n),
\quad
n = 0, 1, \ldots, n_t-2,
\end{align}
corresponding to the FV discretization \eqref{eq:cons-Phi-def} of the PDE system \eqref{eq:cons}.
To this end, let us write \eqref{eq:cons-ts} as a global, nonlinear residual equation:
\begin{align} \label{eq:cons-all-at-once}
\bm{r}( \bm{q} )
:=
\bm{b}
-
{\cal N} (\bm{q})
\equiv
\begin{bmatrix}
\bm{q}^0 \\
\bm{0} \\
\vdots \\
\bm{0}
\end{bmatrix}
-
\begin{bmatrix}
I & \\
-\Phi( \, \cdot \, ) & I \\
& \ddots & \ddots \\
& & -\Phi( \, \cdot \, ) & I
\end{bmatrix}
\begin{bmatrix}
\bm{q}^0 \\
\bm{q}^1 \\
\vdots \\
\bm{q}^{n_t-1} 
\end{bmatrix}
=
\bm{0}.
\end{align}
Here $\bm{q} := ( \bm{q}^{0, \top}, \bm{q}^{1, \top}, \ldots, \bm{q}^{n_t-1,\top} )^\top \in \mathbb{R}^{{\ell} n_x n_t}$, ${\cal N} \colon \mathbb{R}^{{\ell} n_x n_t} \to \mathbb{R}^{{\ell} n_x n_t}$ is the nonlinear space-time discretization operator, and $\bm{b} \in \mathbb{R}^{{\ell} n_x n_t}$ contains the initial data.

Here we solve \eqref{eq:cons-all-at-once} using a residual correction scheme based on that developed in \cite{DeSterck-etal-2023-nonlin-scalar} for scalar nonlinear hyperbolic PDEs.
A summary of the procedure is given in \cref{alg:richardson}, wherein $\bm{q}_k$ is an iterative approximation of $\bm{q}$ with iteration index $k$.
This scheme combines an outer iteration based on a global linearization with an inner iteration for the associated linearized systems.
Note that this algorithm requires a CF-splitting defined by some integer $m > 1$, and that the F-relaxation in Line \ref{ln:nonlin-F-relax} is equivalent to the regular F-relaxation that MGRIT uses for nonlinear problems \cite{Howse_etal_2019}.
The primary step in the algorithm is the solution of a linearized problem ${\cal A}_k \bm{e}_k^{\rm lin} = \bm{r}_k$ with space-time matrix ${\cal A}_k$ given by
\begin{align} \label{eq:Pk-def}
{\cal A}_k
=
\begin{bmatrix}
& I \\
& - \Phi^{\rm lin}( \bm{q}_k^{0}) & I \\
& & \ddots & \ddots  \\
& & & - \Phi^{\rm lin}( \bm{q}_k^{n_t - 2}) & I
\end{bmatrix}
\in
\mathbb{R}^{{\ell} n_x n_t \times {\ell} n_x n_t}.
\end{align}
In \cref{alg:richardson} there is an option to invert ${\cal A}_k$ directly with sequential time-stepping or approximately with the characteristic-based block preconditioner from \cref{alg:char-prec}. 
Ultimately we are interested in the latter option since it allows for a parallel-in-time implementation, but, importantly, the direct inversion option allows for us to understand best-case convergence of the outer preconditioned nonlinear residual correction iteration.

\begin{algorithm}[t!]
\setcounter{AlgoLine}{0}
\KwIn{Initial iterate $\bm{q}_0 \approx \bm{q}$; number of iterations $\texttt{maxit}$; number of inner iterations $\texttt{inner-it}$}
$k \gets 0$\tcp*{Iteration counter}
\KwOut{Approximate solution $\bm{q}_k \approx \bm{q}$}
\While{$k < \mathtt{maxit}$}{
  	${\bm{q}}_{k} \gets \textrm{relax on }{\cal N}(\bm{q}_k) \approx \bm{b}$\tcp*{Nonlinear F-relaxation}\label{ln:nonlin-F-relax}
	${\bm{r}}_{k} \gets \bm{b} - {\cal N} ({\bm{q}}_{k})$\tcp*{Compute residual}
	\uIf{exact linear solve}{ 
	\tcc{Directly compute linearized error via time-stepping; sequential-in-time}
      ${\bm{e}}_k^{\rm lin} \gets {\cal A}_k^{-1} {\bm{r}}_k$\label{ln:nonlin-exact-solve}
          }	 
          \Else{
          \tcc{Approximate error using \texttt{inner-it} iterations of char.-based block preconditioning in \cref{alg:char-prec}; may be solved parallel-in-time}
      ${\bm{e}}_k^{\rm lin} \gets \textrm{char-based-block-prec} ({\cal A}_k, {\bm{e}}_k^{\rm lin}, {\bm{r}}_k,$\texttt{inner-it})\tcp*{Need to initialize ${\bm{e}}_k^{\rm lin}$}\label{ln:nonlin-char-prec}
    }
	$\bm{q}_{k+1} \gets {\bm{q}}_k + {\bm{e}}_k^{\rm lin}$\tcp*{Compute new iterate}
	$k \gets k + 1$\tcp*{Increment iteration counter}
	} 
	  \caption{Preconditioned residual correction scheme for \eqref{eq:cons-all-at-once}.
  \label{alg:richardson}}
\end{algorithm}

Next in \cref{sec:cons-outer-iter}, we develop an expression for $\Phi^{\rm lin}( \bm{q}_k^{n})$ in \eqref{eq:Pk-def}, and then in \cref{sec:cons-inner-iter,sec:cons-P} we discuss characteristic-variable block preconditioning for the linearized problems. We conclude with numerical results in \cref{sec:cons-num-res}.

\subsection{The linearized time-stepping operator}
\label{sec:cons-outer-iter}

Considering the form of $\Phi$ in \eqref{eq:cons-Phi-def}, its Jacobian exists everywhere because the dissipation matrix (the only potentially problematic term) $|A^*_{i+1/2}|$ is based on a smoothed absolute value function.
In \cite{DeSterck-etal-2023-nonlin-scalar}, in the context of scalar PDEs, non-differentiability of $\Phi$ arose due to non-smoothness of the dissipation coefficient in a local Lax--Friedrichs numerical flux, which uses a non-smoothed absolute value.
Yet, we were able to develop effective $\Phi^{\rm lin}$ in \cite{DeSterck-etal-2023-nonlin-scalar} by freezing the dissipation coefficient in the numerical flux while differentiating $\Phi$.
This also led to tractable methods with manageable computational costs. 
To develop a tractable expression for $\Phi^{\rm lin}$ in the present setting, we apply this same idea leading to the following approximate Jacobian:
\begin{align} \label{eq:Phi-lin-def}
\Phi^{\rm lin}( \bm{q}_k^n )
=
\Big[ \nabla_{ \bm{q}_k^n } \Phi( \bm{q}_k^n ) \Big]_{| A^*(\bm{q}_k^n)|},
\end{align}
where this notation denotes holding $|A^*(\bm{q}_k^n)|$ constant when differentiating $\Phi$ in \eqref{eq:cons-Phi-def}.

To simplify notation, let us drop the iteration index $k$ henceforth.
Consider that in \cref{alg:richardson} we need to solve ${\cal A} \bm{e} = \bm{r}$. Given the sparsity structure of ${\cal A}$ in \eqref{eq:Pk-def} it becomes apparent that we need to know how to compute the action of $\Phi^{\rm lin}( \bm{q}^n)$ on some (error) vector $\bm{e}^n \in \mathbb{R}^{\ell n_x}$.
In particular, it suffices for us to understand how to compute this action in cell $i$: 
\begin{align}
\big( \Phi^{\rm lin}( \bm{q}^n) \bm{e}^n \big)_i
=
\Big(
\Big[ \nabla_{ \bm{q}^n } \Phi( \bm{q}^n ) \Big]_{| A^* (\bm{q}^n)|} \bm{e}^n 
\Big)_i
=
\Big[ \nabla_{ \bm{q}^n } \big( \Phi( \bm{q}^n ) \big)_i \Big]_{| A^* (\bm{q}^n)|} \bm{e}^n 
\in
\mathbb{R}^{\ell},
\end{align}
with $\big( \Phi( \bm{q}^n ) \big)_i$ defined in \eqref{eq:cons-Phi-def}, and $\Big[ \nabla_{ \bm{q}^n } \big( \Phi( \bm{q}^n ) \big)_i \Big]_{| A^* (\bm{q}^n)|} \in \mathbb{R}^{\ell \times \ell n_x}$ being its approximate Jacobian. 
Differentiating \eqref{eq:cons-Phi-def} in the direction of $\bm{e}^n$ we find
\begin{align} \label{eq:linearized-Phi}
\big( \Phi^{\rm lin}( \bm{q}^n ) \bm{e}^n \big)_i 
=
\Big[ \nabla_{ \bm{q}^n } \big( \Phi( \bm{q}^n ) \big)_i \Big]_{| A^* (\bm{q}^n)|} 
\bm{e}^n
=
\bm{e}^n_i
-
\frac{\delta t}{h}
\Big[
\wh{\bm{f}}_{i+1/2}^{\textrm{lin},n} - \wh{\bm{f}}_{i-1/2}^{\textrm{lin},n}
\Big],
\end{align}
where, by the chain rule, we have $
\wh{\bm{f}}_{i+1/2}^{\textrm{lin},n} 
:=
\Big[ \nabla_{ \bm{q}^n } \wh{\bm{f}}_{i+1/2}( \bm{q}^n ) \Big]_{| A^* (\bm{q}^n)|} \bm{e}^n
$.
Using \eqref{eq:Roe-simple} to compute this directional derivative gives
\begin{align} \label{eq:linearized-flux}
\wh{\bm{f}}_{i+1/2}^{\textrm{lin},n} 
=
\frac{1}{2}
\Big[
\Big(
A(\bm{q}_{i+1}^n) \bm{e}^n_{i+1} + A(\bm{q}_{i}^n) \bm{e}^n_{i}
\Big)
-
\big| A^*_{i+1/2}( \bm{q}^n ) \big|
\big(
\bm{e}_{i+1}^n - \bm{e}_{i}^n
\big)
\Big] 
\in 
\mathbb{R}^{\ell},
\end{align}
recalling the shorthand $A(\bm{q}) := \bm{f}'(\bm{q})$ for the flux Jacobian.

Observe that the \textit{linearized} time-stepping operator \eqref{eq:linearized-Phi}, paired with the \textit{linearized} flux \eqref{eq:linearized-flux}, has a structure that is almost identical to the \textit{nonlinear} time-stepping operator \eqref{eq:cons-Phi-def}, paired with the \textit{nonlinear} flux \eqref{eq:Roe-simple}.
That is, \eqref{eq:linearized-Phi} is the time-stepping operator of a first-order FV discretization of the linear system of conservation laws
\begin{align} \label{eq:cons-linearized}
\frac{\partial \bm{e}_k}{\partial t} + \frac{\partial }{\partial x} \big( A(\bm{q}_k) \bm{e}_k \big) = \bm{0},
\quad
(x, t) \in \Omega \times (0, T],
\end{align}
with $\bm{q}_k$ the current approximation to the solution $\bm{q}$ of the nonlinear space-time system \eqref{eq:cons-all-at-once} (the nonlinear iteration index $k$ is included here for clarity).
Specifically, if we define the linearized flux function $\bm{f}^{\textrm{lin}} (\bm{e}) := A(\bm{q}_k) \bm{e}$, then the linearized discretization \eqref{eq:linearized-Phi} arises from applying a first-order FV discretization to \eqref{eq:cons-linearized} with a Roe-like flux based on $\bm{f}^{\textrm{lin}} (\bm{e})$.
This scenario is a direct systems generalization of the situation in \cite{DeSterck-etal-2023-nonlin-scalar} where the linearized problem is a scalar conservation law based on a linearization of the flux in the underlying nonlinear PDE.
This realization allows us to compute truncation error corrections for semi-Lagrangian coarse-grid schemes that are crucial for efficient MGRIT preconditioning when inexactly applying Line \ref{ln:nonlin-char-prec} of \cref{alg:richardson}.
%

\subsection{The inner linearized iteration}
\label{sec:cons-inner-iter}

We now discuss solving the linearized systems in \cref{alg:richardson} using characteristic-variable block preconditioning.
Throughout this section we drop the nonlinear iteration index $k$ for readability, but, where convenient to do so, we add ``0'' subscripts to quantities associated with the flux Jacobian to remind the reader that it is based on $\bm{q}_k \approx \bm{q}$; that is, $A_0 := A( \bm{q}_k )$.

Our approach for solving the linearized system ${\cal A}_0 \bm{e} = \bm{r}$ is based on that developed for the acoustics system in \cref{sec:acoustics-pint}. However, there are a few key differences:
(i) the PDE system \eqref{eq:cons-linearized} is ${\ell}$-dimensional, while \eqref{eq:acoustic} is two-dimensional;
(ii) the PDE system \eqref{eq:cons-linearized} is in conservative form, while \eqref{eq:acoustic} is in non-conservative form;
(iii) the coefficient matrix $A_0 := A( \bm{q}_k )$ in \eqref{eq:cons-linearized} varies in both space and time, while that in \eqref{eq:acoustic} varies in space only.
Despite these key differences, \cref{alg:char-prec} can be used without modification (modulo changes in notation) to approximately solve ${\cal A}_0 \bm{e} = \bm{r}$.
The remainder of this section discusses transforming to characteristic variables, and block preconditioners.

To gain insight into the characteristic variables at the discrete level, we first consider them at the PDE level.
To re-express the linearized PDE \eqref{eq:cons-linearized} in characteristic variables $\wh{\bm{e}} := R_0^{-1} \bm{e}$---recall that characteristic-space quantities are denoted with a hat---, multiply the PDE from the left by $R_0^{-1}$ to get
\begin{align} \label{eq:linearized-char-PDE}
\bm{0} 
=
R_0^{-1}
\left[
\frac{\partial }{\partial t} 
\big( R_0 \wh{\bm{e}} \big)
+
\frac{\partial }{\partial x} \big( R_0 \Lambda_0 \wh{\bm{e}} \big)
\right]
\approx
\frac{\partial \wh{\bm{e}}}{\partial t} 
+
\frac{\partial }{\partial x} \big( \Lambda_0 \wh{\bm{e}} \big)
\quad
\textrm{when $R_0$ is slowly varying},
\end{align}
with $A_0 = R_0 \Lambda_0 R_0^{-1}$.
A pertinent question is to what extent the eigenvectors $R_0$ of the flux Jacobian may be slowly varying.
In general, we cannot assume they vary slowly; however, in practice, this may be a reasonable approximation for many problems.
That is, even though \eqref{eq:cons} admits solutions with shocks---across which $R_0$ may change discontinuously---, in practice they are typically limited in number, contained in isolated regions of space-time, and away from these shocked regions, the solution may be smoothly and slowly varying. 
For example, in the context of motivating linearized Riemann solvers, LeVeque \cite[p. 316]{LeVeque_2004} writes ``\textit{... The solution to a conservation law typically consists of at most a few isolated shock waves or contact discontinuities separated by regions where the solution is smooth. In these regions, the variation in $\bm{q}$ from one grid cell to the next has $\Vert \bm{q}_{i} - \bm{q}_{i-1} \Vert = {\cal O}(h)$ and the Jacobian matrix is nearly constant, $\bm{f}'(\bm{q}_{i - 1}) \approx \bm{f}'(\bm{q}_{i})$.}''
%

\subsection{Constructing $\wh{{\cal P}}$ and $\wt{{\cal P}}$}
\label{sec:cons-P}

We now consider characteristic variables at the discrete level. Recall from \cref{sec:cons-outer-iter} that the global linearized system is 
\begin{align} \label{eq:linearized-all-at-once}
{\cal A}_0 
\bm{e}
=
\bm{r},
\quad
\textrm{or}
\quad
\bm{e}^{n+1} = \Phi_0^{n} \bm{e}^n + \bm{r}^{n+1},
\quad
n = 0, \ldots, n_t - 2,
\end{align}
where we now use the shorthand $\Phi^{n}_0 := \Phi^{\textrm{lin}}( \bm{q}^n )$ for readability.
To transform this discretized system to characteristic variables we need to introduce the discretized left and right eigenvector matrices $({\cal R}_0^n )^{-1}$ and ${\cal R}_0^n \in \mathbb{R}^{\ell n_x \times \ell n_x}$, respectively. These are $\ell \times \ell$ block matrices where the $ij$th block is a diagonal matrix corresponding to the values of the $ij$th element of $(R_0^n)^{-1}$ and $R_0^n$, respectively, over the grid points, i.e., these matrices generalize those in \eqref{eq:acoustic-calR} that discretize the eigenvectors of the flux Jacobian in \eqref{eq:acoustic}.
Multiplying the $n$th equation in \eqref{eq:linearized-all-at-once} on the left by $({\cal R}_0^{n+1})^{-1}$ we get the characteristic-space residual equation
\begin{align}
\begin{split}
\wh{\bm{e}}^{n+1} &= \wh{\Phi}_0^{n} \wh{\bm{e}}^n + \wh{\bm{r}}^{n+1},
\quad
n = 0, \ldots, n_t - 2,
\quad
\textrm{where},
\quad
\quad
\wh{\bm{e}}^n 
:=
({\cal R}_0^{n})^{-1}
{\bm{e}}^n,
\\
\wh{\Phi}_0^{n}
&:=
({\cal R}_0^{n+1})^{-1}
{\Phi}_0^{n}
{\cal R}_0^{n}
=
\begin{bmatrix}
\wh{\Phi}_{0,11}^{n} & \ldots & \wh{\Phi}_{0,1 \ell}^{n} \\
\vdots & \ddots & \vdots \\
\wh{\Phi}_{0,\ell 1}^{n} & \ldots & \wh{\Phi}_{0,\ell \ell}^{n}
\end{bmatrix},
\quad
\wh{\bm{r}}^n 
:=
({\cal R}_0^{n})^{-1}
{\bm{r}}^n.
\end{split}
\end{align}
Re-ordering this system to be blocked by characteristic variable gives: 
\begin{align}
\wh{{\cal A}}_0 
\wh{\bm{e}}
=
\wh{\bm{r}},
\quad
\textrm{or}
\quad
\begin{bmatrix}
\wh{{\cal A}}_{0,11} & \ldots & \wh{{\cal A}}_{0,1 \ell} \\
\vdots & \ddots & \vdots \\
\wh{{\cal A}}_{0,\ell 1} & \ldots & \wh{{\cal A}}_{0,\ell \ell}
\end{bmatrix}
\begin{bmatrix}
\wh{\bm{e}}^{1} \\
\vdots \\
\wh{\bm{e}}^{\ell}
\end{bmatrix}
=
\begin{bmatrix}
\wh{\bm{r}}^{1} \\
\vdots \\
\wh{\bm{r}}^{\ell}
\end{bmatrix},
\end{align}
where the blocks $\wh{{\cal A}}_{0,ii}, \wh{{\cal A}}_{0,ij} \in \mathbb{R}^{n_x n_t \times n_x n_t}$, $j \neq i$, are given by
\begin{align} 
\wh{{\cal A}}_{0,ii}
=
\begin{bmatrix}
I \\
-\wh{\Phi}^0_{0,ii} & I \\
& \ddots & \ddots \\
& & -\wh{\Phi}^{n_t-1}_{0,ii} & I
\end{bmatrix},
\quad
\wh{{\cal A}}_{0,ij}
=
\begin{bmatrix}
0 \\
-\wh{\Phi}^0_{0,ij} & 0 \\
& \ddots & \ddots \\
& & -\wh{\Phi}^{n_t-1}_{0,ij} & 0
\end{bmatrix}.
\end{align}
We now propose an $\ell \times \ell$ block preconditioner $\wh{{\cal P}}$ by truncating $\wh{\cal A}_0$ to be block diagonal: \footnote{An alternative way to think about this preconditioner is that it replaces the $\ell \times \ell$ block time-stepping operator $\wh{\Phi}_0^{n}$ in \eqref{eq:linearized-all-at-once} with its block diagonal.}
\begin{align}
\wh{{\cal P}}_D^{-1}
=
\begin{bmatrix}
\wh{{\cal A}}_{0,11} & \\
& \ddots  \\
&  & \wh{{\cal A}}_{0,\ell \ell}
\end{bmatrix}^{-1}.
\end{align}
Another preconditioning possibility concerns block triangular approximations to $\wh{{\cal A}}_{0}$; however, we omit discussion of these, since results (not shown here) indicate little or no difference in convergence behavior compared to $\wh{{\cal P}}_D^{-1}$.

Analogous to the acoustics preconditioners in \cref{sec:acoustics-prec}, we consider also an alternative preconditioner $\wt{{\cal P}}_D$ as in \eqref{eq:prec-approx-def} based on replacing $\wh{A}_{0,ii}$ in $\wh{{\cal P}}_D$ with approximate diagonal blocks $\wt{A}_{0,ii}$.
A straightforward way of developing a potentially suitable approximate time-stepping operator $\wt{\Phi}_{ii}^n$ is to discretize, with a Roe-like flux, the $i$th characteristic variable in the decoupled characteristic system from the right-hand side of \eqref{eq:linearized-char-PDE}.
Doing so yields a $\wt{\Phi}^{n}_{ii}$ with action given by
\begin{align}
\big( \wt{\Phi}^{n}_{ii} \wh{\bm{e}}^{n} \big)_i
=
\wh{e}^{n}_i
-
\frac{\delta t}{h}
\big(
\wh{g}_{i+1/2}^n
-
\wh{g}_{i-1/2}^n
\big),
\end{align}
with flux 
$
\wh{g}_{i+1/2}^n
=
\frac{1}{2}
\big[
\big(
\lambda_{i+1}^{n} \wh{e}^{n}_{i+1}
+
\lambda_{i}^{n} \wh{e}^{n}_{i}
\big)
-
\big| \lambda_{i}^{*,n} \big| 
\big(
\wh{e}^{n}_{i+1}
-
\wh{e}^{n}_{i}
\big)
\big]
$.

\subsection{Numerical experiments}
\label{sec:cons-num-res}

We now present numerical results for our solver of the all-at-once system \eqref{eq:cons-all-at-once}.
First we discuss parameter settings.
\Cref{alg:richardson} is iterated until the $\ell^2$-norm of the space-time residual $\bm{r}_k$ is reduced by at least 10 orders of magnitude from its initial value or for a maximum of 15 iterations.
Each PDE problem is solved on a sequence of increasingly finer space-time meshes so as to understand the asymptotic performance of the solver; the initial nonlinear iterate on a given space-time mesh is computed by interpolating the solution of the same PDE problem on the preceding coarser mesh.
The nonlinear F-relaxation in \cref{alg:richardson} is employed with a CF-splitting factor of $m = 8$.
Where the characteristic-based solver \cref{alg:char-prec} is applied to linearized problems, i.e., as in Line \ref{ln:nonlin-char-prec}, it uses F-relaxation with a CF-splitting factor of $m = 8$, and the linearized error is initialized as the right-hand side vector.
When applying preconditioners $\wt{{\cal P}}$ based on approximate diagonal blocks $\wt{{\cal A}}_{ii}$, we consider inverting $\wt{{\cal A}}_{ii}$ directly with time-stepping and inexactly using a single MGRIT V-cycle. Here MGRIT uses a modified conservative semi-Lagrangian discretization as in \cite{DeSterck-etal-2023-nonlin-scalar}; otherwise, all MGRIT parameters are chosen as they were in \cref{sec:acoustics-num-res} for the acoustics problems.

For both \eqref{eq:SWE} and \eqref{eq:euler} we consider initial conditions below with amplitude dependent on a constant $\varepsilon$, and thus the strength of the nonlinearity of the resulting PDE systems is parametrized by $\varepsilon$ \cite[Chapter 13.5]{LeVeque_2004}.
Our numerical results show that the characteristic-based preconditioner struggles more for larger-amplitude problems than smaller-amplitude problems, which is to be expected from the discussion in \cref{sec:cons-inner-iter} since the eigenvectors of the flux Jacobian for larger amplitude problems can change more rapidly at discontinuities.
The time-step size $\delta t$ is constant across the time domain, and is chosen such that the constant factor $c_0 := \max_{s \in \{ 1, \ldots, \ell \}  } \max_{x \in \Omega} |\lambda_s( \bm{q}(x, 0) )| \tfrac{\delta t}{h}$ (to be specified) is smaller than unity.

The test problems we consider are as follows.
First, consider \eqref{eq:SWE} with an initial depth perturbation:
\begin{align}
\tag{IDP}
\label{eq:SWE-idp}
h(x, 0) &= 1 + \varepsilon \exp( -5 (x - 5/2)^2 ),
\quad u(x, 0) = 0,
\end{align}
with $(x, t) \in (-5, 5 ) \times (0, 10)$, $c_0 = 0.8$, and subject to periodic boundary conditions in space.
Small- and larger-amplitude results are shown in \cref{fig:SWE-IDP-weakly,fig:SWE-IDP-fully}, respectively.
Second, consider \eqref{eq:euler} with an initial density and pressure perturbation:
\begin{align}
\tag{IDPP}
\label{eq:euler-idp}
\rho(x, 0) = 1 + \varepsilon \exp( -5 (x - 5/2)^2 ), 
\quad
u(x, 0) = 0, 
\quad
p(x, 0) = \rho(x, 0),
\end{align}
with $(x, t) \in (-5, 5) \times (0, 10)$, $c_0 = 0.7$, and with periodic boundary conditions. 
Small- and larger-amplitude results are shown in \cref{fig:euler-IDPP-weakly,fig:euler-IDPP-fully}, respectively.
The initial pulses in \eqref{eq:SWE-idp} and \eqref{eq:euler-idp} split into left- and right-travelling pulses which, due to the periodic boundaries, collide twice over the space-time domain. 
The leading edge of each travelling pulse develops into a shock, and the trailing edge into a rarefaction. 
Additional qualitatively similar numerical results can be found in Supplementary Materials Section \ref{SMsec:num-res-nonlin} wherein standard Riemann problems are considered for both \eqref{eq:SWE} and \eqref{eq:euler} in the form of dam break and Sod problems, respectively.

\begin{figure}[t!]
\centerline{
\includegraphics[width=0.345\textwidth]{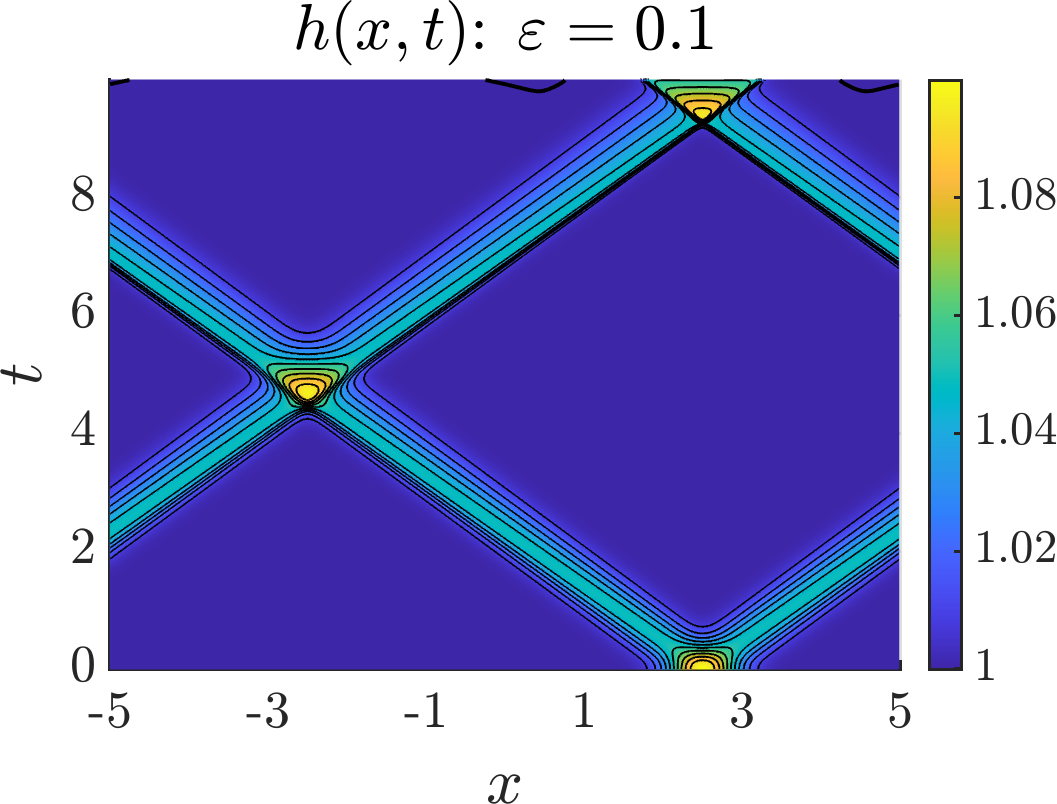}
\hspace{0ex}
\includegraphics[width=0.325\textwidth]{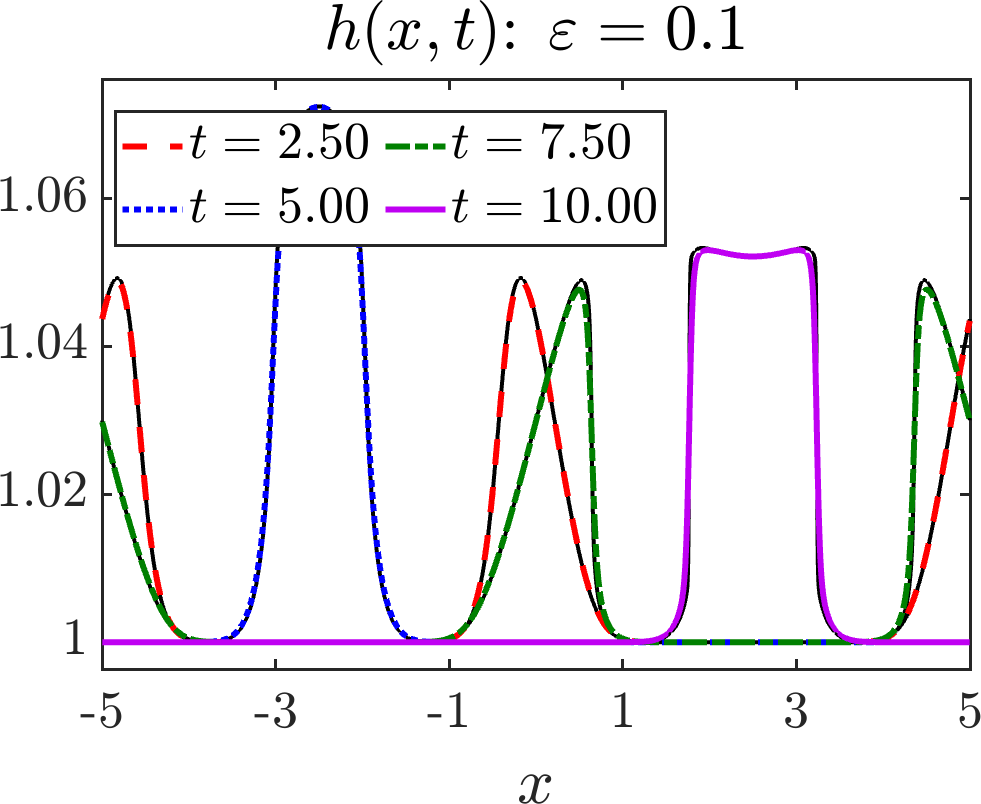}
\hspace{0ex}
\includegraphics[width=0.325\textwidth]{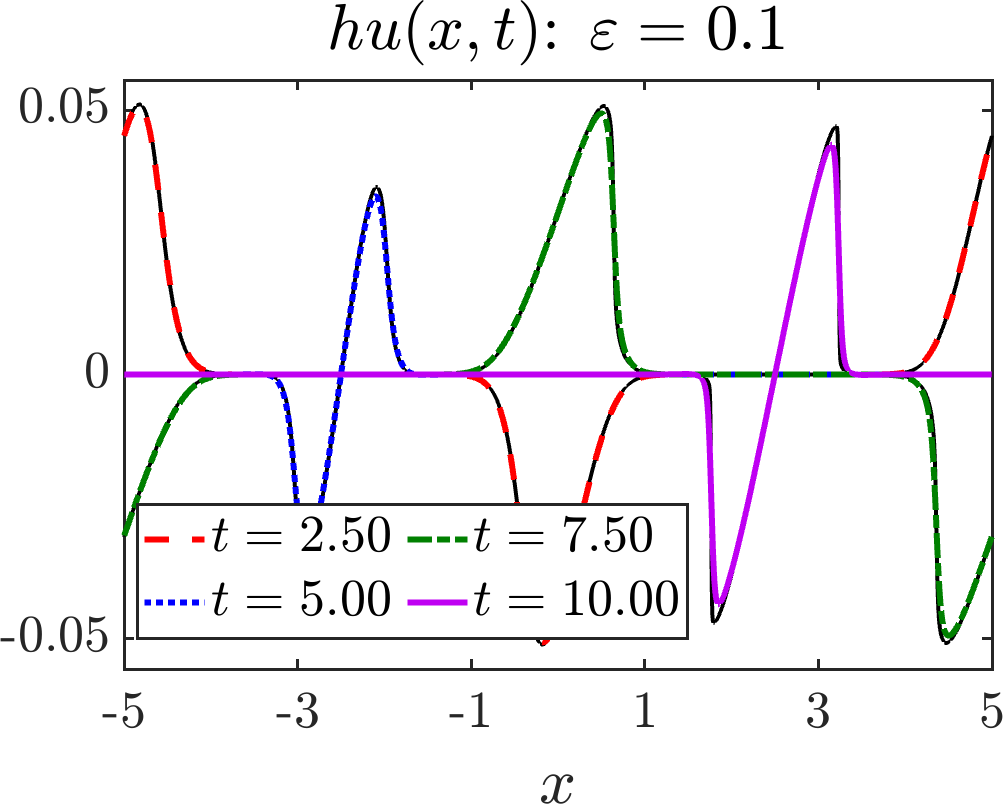}
}
\vspace{1ex}
\centerline{
\includegraphics[width=0.335\textwidth]{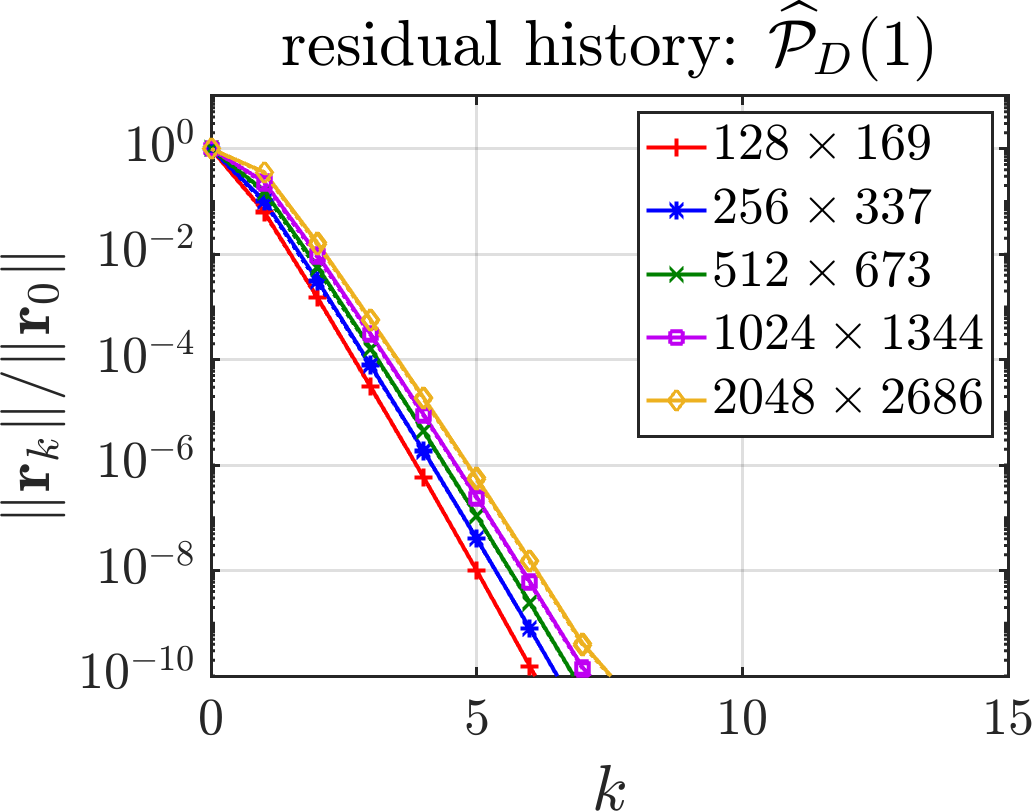}
\hspace{1ex}
\includegraphics[width=0.335\textwidth]{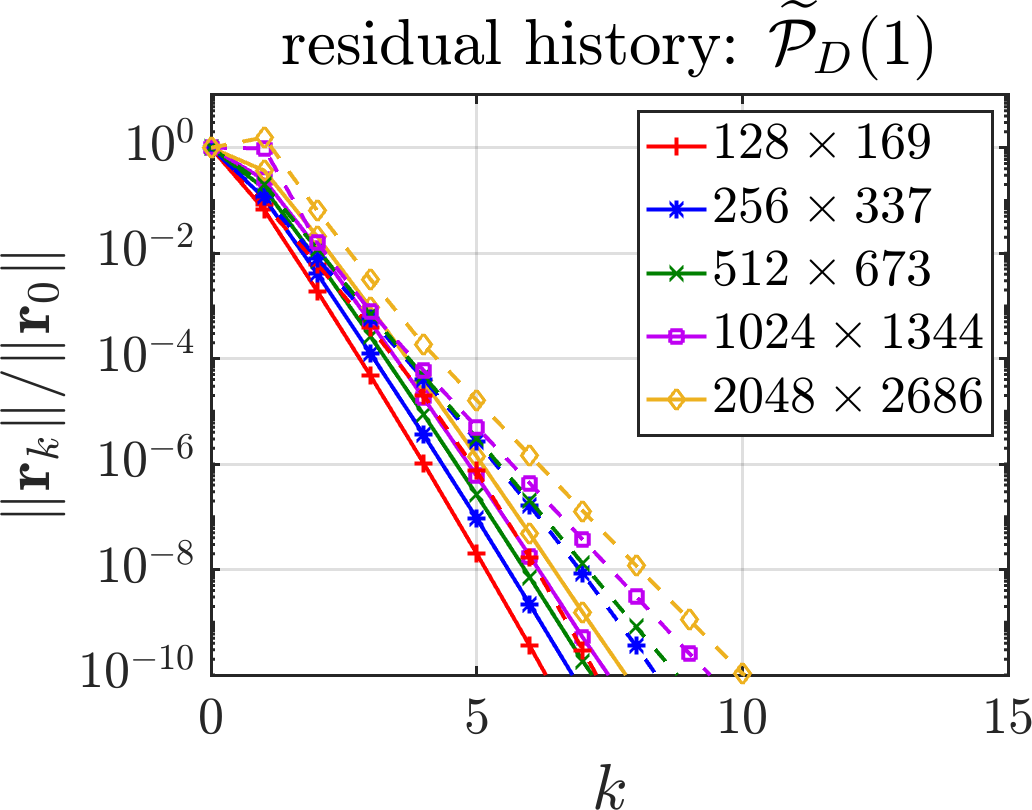}
}
\caption{
Small-amplitude \eqref{eq:SWE-idp} problem for \eqref{eq:SWE}, with $\varepsilon = 0.1$.
\textbf{Bottom left:} Nonlinear residual history with dotted lines corresponding to exact solves of the linearized problems ${\cal A}_k \bm{e}_k^{\rm lin} = \bm{r}_k$, and solid lines to approximate solves with a single iteration of \cref{alg:char-prec} using $\wh{{\cal P}}$.
\textbf{Bottom right:} Nonlinear residual history where linearized problems ${\cal A}_k \bm{e}_k^{\rm lin} = \bm{r}_k$ are approximately solved with a single iteration of \cref{alg:char-prec} using $\wt{{\cal P}}$ with diagonal blocks $\wt{{\cal A}}_{ii}$ either inverted exactly (solid lines) or approximately with one MGRIT V-cycle (dashed lines).
Legends in the bottom row correspond to space-time mesh resolutions of $n_x \times n_t$.
\label{fig:SWE-IDP-weakly}
}
\end{figure}

The results of the above problems are shown in \cref{fig:SWE-IDP-weakly,fig:SWE-IDP-fully,fig:euler-IDPP-weakly,fig:euler-IDPP-fully} and the plots in each figure can be interpreted in the following way.
Top row: Space-time contour of the first component of $\bm{q}$, and cross-sections at the times indicated for the first two components of $\bm{q}$ corresponding to the highest spatial resolution used in our tests of $n_x = 2048$; underlying the cross-sections are thin black curves showing reference solutions obtained at spatial resolutions of $n_x = 8192$.
Bottom row: Nonlinear residual history using \cref{alg:richardson} when linearized systems are solved directly with time-stepping, as in Line \ref{ln:nonlin-exact-solve} of \cref{alg:richardson}, or approximately using \texttt{inner-it} applications of the characteristic-based preconditioned iteration \cref{alg:char-prec}, as in Line \ref{ln:nonlin-char-prec} of \cref{alg:richardson}. 
Ideally \texttt{inner-it} $=1$, but results show more than one inner iteration is sometimes beneficial and even necessary for the outer iteration to converge.
Recall that $\wh{{\cal P}}$ uses diagonal blocks from $\wh{\Phi}_0^n$ and $\wt{\cal P}$ uses a certain approximation to them.\footnote{Implementation note: The proposed preconditioner $\wh{{\cal P}}_D$ requires the action of the diagonal blocks in $\wh{\Phi}_0^n$. For the acoustics problem we computed the stencils of the analogous blocks analytically (see \cref{lem:god-char-stencils}). For the current nonlinear problems, these calculations are significantly more involved, so an alternative is needed.
Recall that $\wh{\Phi}_0^{n}
:=
({\cal R}_0^{n+1})^{-1}
{\Phi}_0^{n}
{\cal R}_0^{n}$; thus, since we know how to compute the action of $({\cal R}_0^{n+1})^{-1}$, ${\Phi}_0^{n}$, and ${\cal R}_0^{n}$ we can use these to compute the action of $\wh{\Phi}_0^{n}$.
While the resulting computations are inefficient, since computing the action of the block diagonal component of $\wh{\Phi}_0^{n}$ requires the action of $({\cal R}_0^{n+1})^{-1}$, ${\Phi}_0^{n}$, and ${\cal R}_0^{n}$ $\ell$ times each, this allows us to examine the convergence behavior of the proposed preconditioner.} 
Plots in the bottom row of each figure include either $\wh{{\cal P}}$(\texttt{inner-it}) or $\wt{{\cal P}}$(\texttt{inner-it}) in their titles. Plots titled $\wh{{\cal P}}$(\texttt{inner-it}) show results with linearized systems in \cref{alg:richardson} either solved approximately with block preconditioner $\wh{{\cal P}}$(\texttt{inner-it}) (solid lines), or solved exactly via time-stepping (dotted lines); note that all results using $\wh{\cal P}$ apply the preconditioner exactly, i.e., using exact inverses for diagonal blocks $\wh{{\cal A}}_{ii}$.
Plots titled $\wt{\cal P}$(\texttt{inner-it}) show results for block preconditioning with $\wt{\cal P}$, with solid lines corresponding to exact inverses of diagonal blocks $\wt{{\cal A}}_{ii}$ and dashed lines to inexact inverses via one MGRIT V-cycle.

\begin{figure}[t!]
\centerline{
\includegraphics[width=0.345\textwidth]{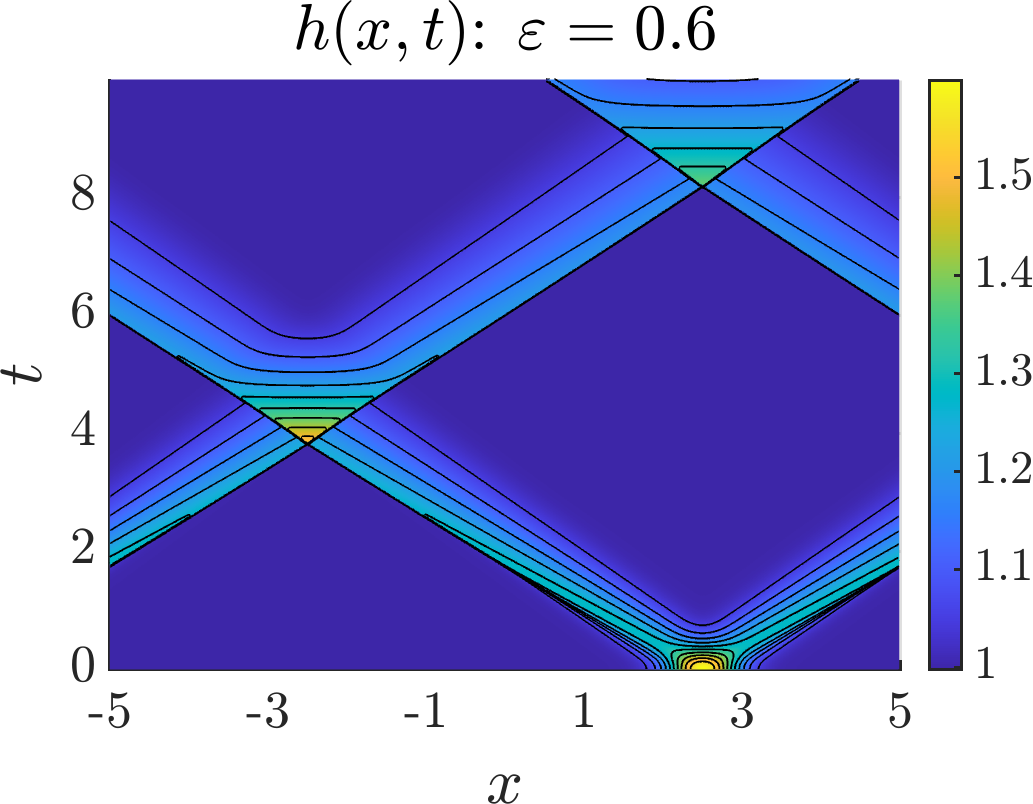}
\hspace{0ex}
\includegraphics[width=0.325\textwidth]{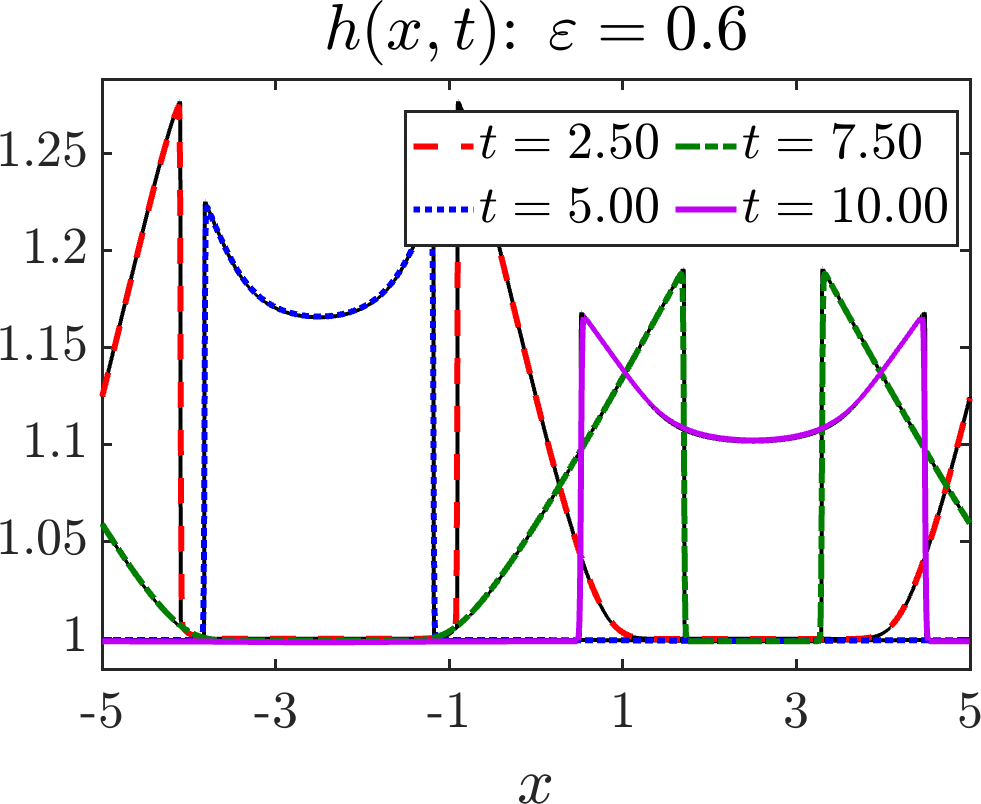}
\hspace{0ex}
\includegraphics[width=0.325\textwidth]{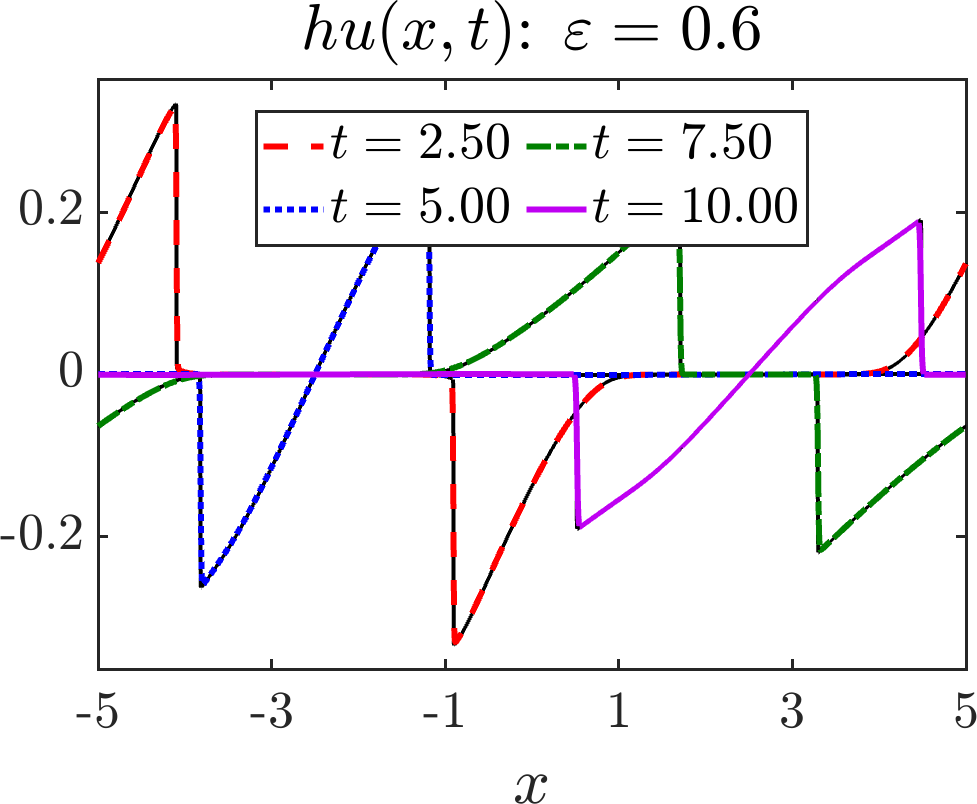}
}
\vspace{1ex}
\centerline{
\includegraphics[width=0.335\textwidth]{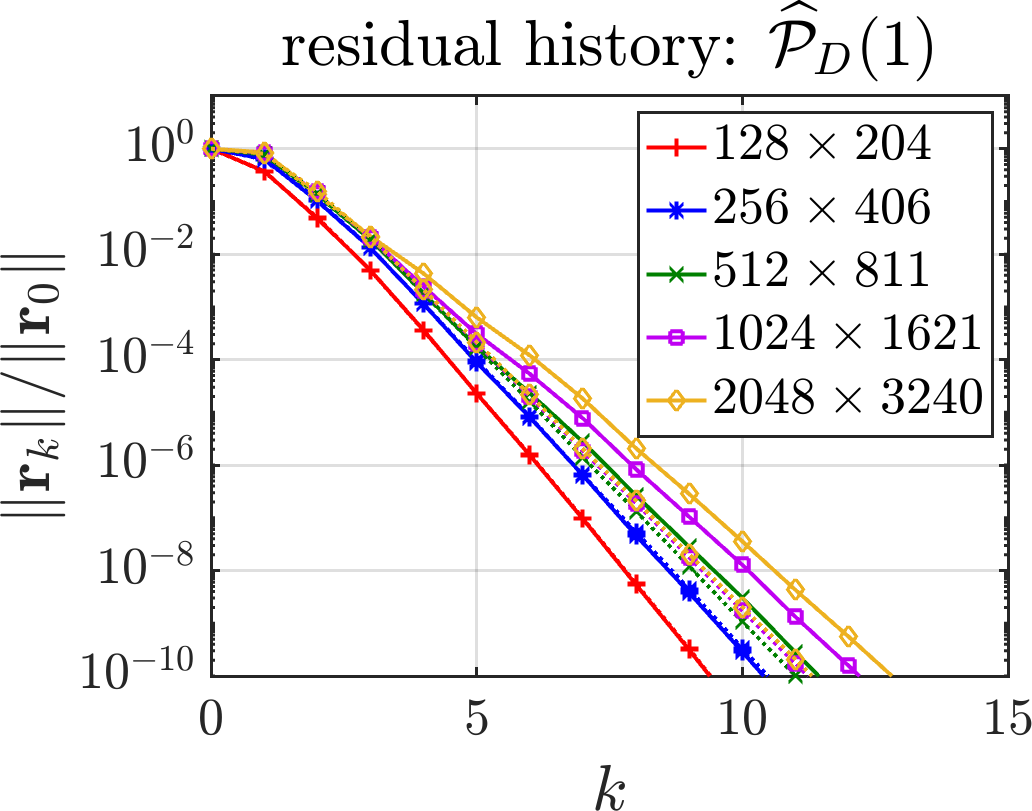}
\includegraphics[width=0.335\textwidth]{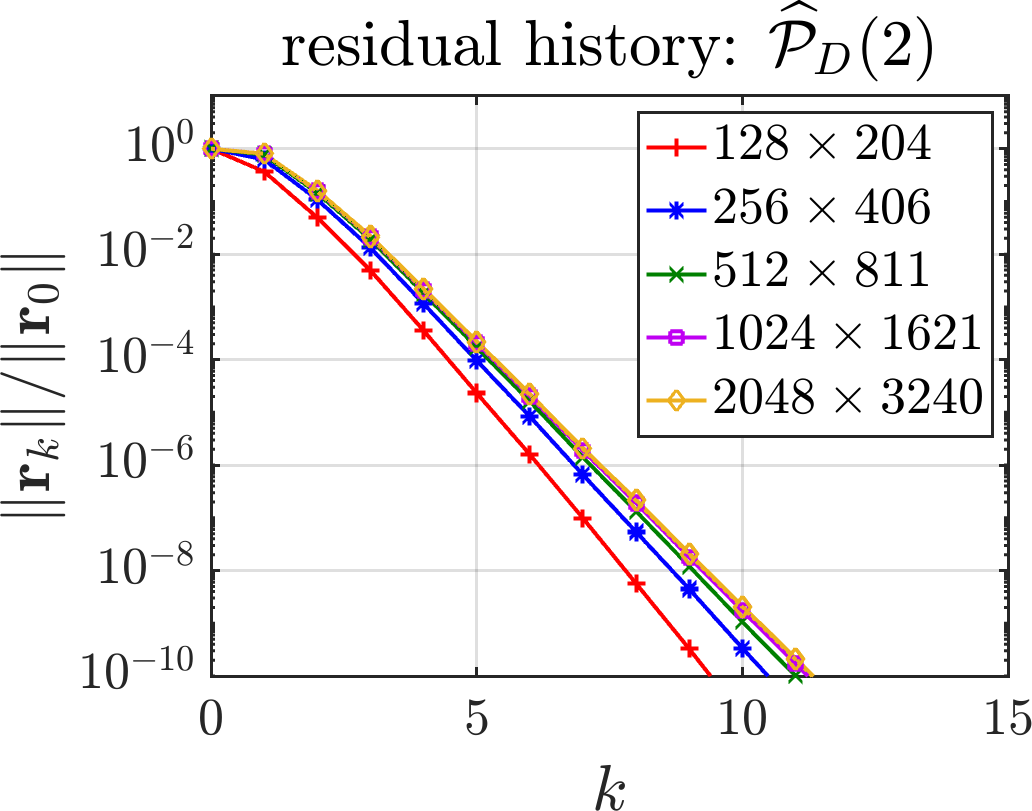}
\includegraphics[width=0.335\textwidth]{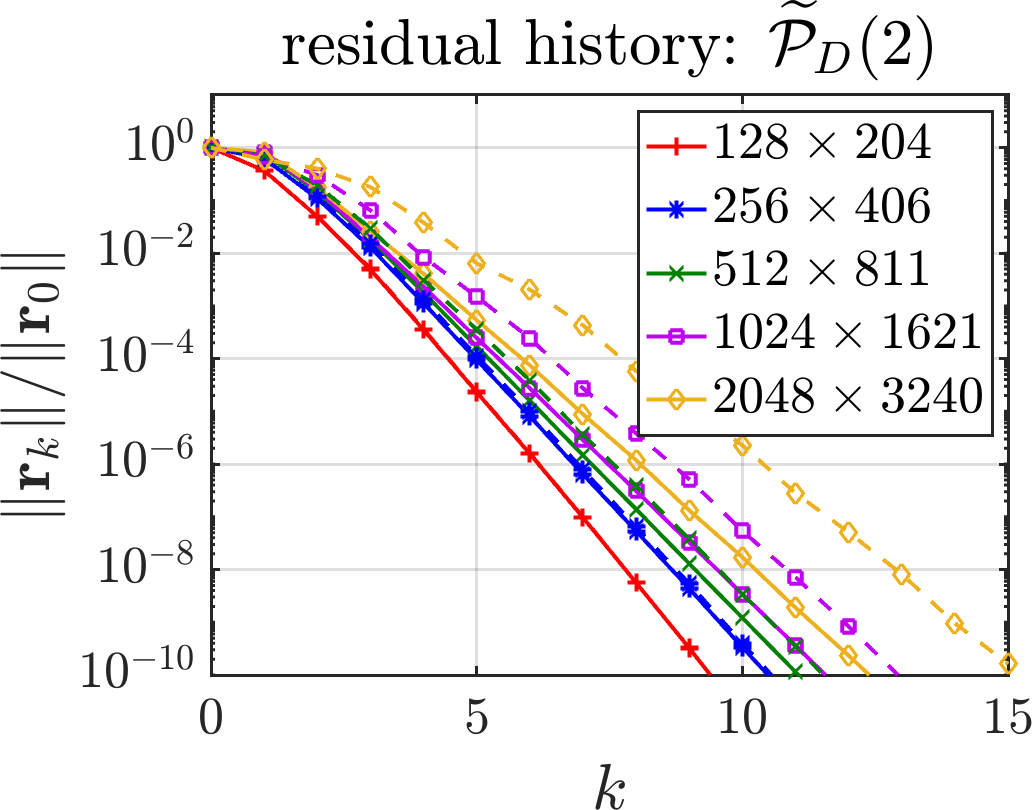}
}
\caption{
Larger-amplitude \eqref{eq:SWE-idp} problem for \eqref{eq:SWE}, with $\varepsilon = 0.6$.
\label{fig:SWE-IDP-fully}
}
\end{figure}

Some commentary is now in order. For all problems, when the linearized problems are solved directly, the solver quickly converges (i.e., the dotted lines in plots titled $\wh{{\cal P}}$(\texttt{inner-it}), noting that these dotted lines often sit directly underneath the solid lines in these plots). 
This indicates that the idea of freezing the dissipation matrix in the linearization \eqref{eq:Phi-lin-def} is reasonable. 
This result also suggests that if one wishes to solve nonlinear hyperbolic systems parallel-in-time, then it suffices to know how to solve linear(ized) hyperbolic systems parallel-in-time.
In all of the small-amplitude problems, the solver converges quickly, and there is essentially no deterioration in convergence when direct solves of linearized problems are replaced with a single iteration of either block preconditioner $\wh{{\cal P}}$ or $\wt{{\cal P}}$. 
We also see little, if any, deterioration in convergence when diagonal blocks of $\wt{{\cal P}}$ are inverted inexactly via MGRIT---the dashed lines in plots titled $\wt{{\cal P}}$(\texttt{inner-it}).
\begin{figure}[t!]
\centerline{
\includegraphics[width=0.345\textwidth]{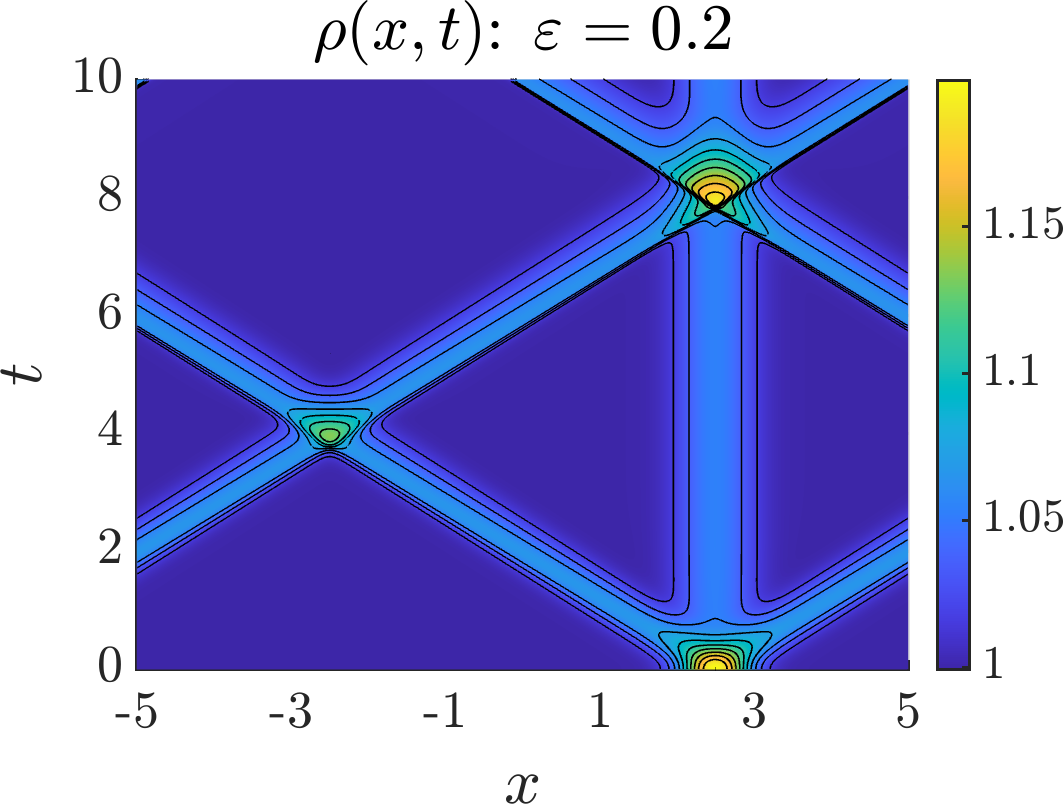}
\hspace{0ex}
\includegraphics[width=0.325\textwidth]{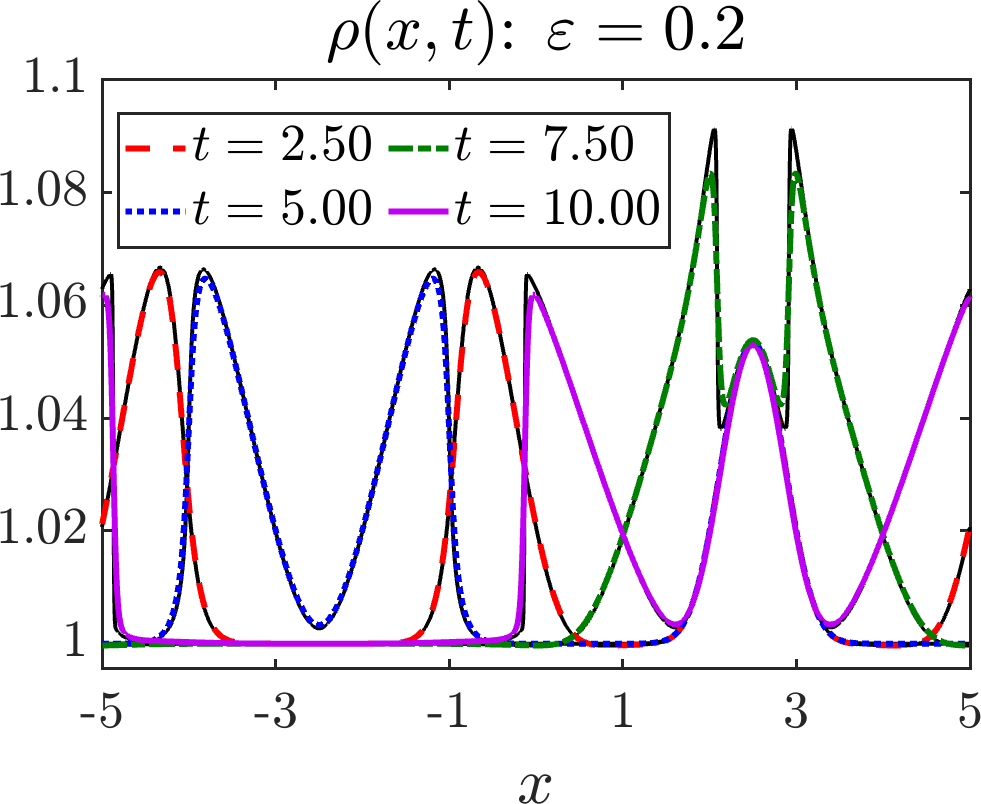}
\hspace{0ex}
\includegraphics[width=0.325\textwidth]{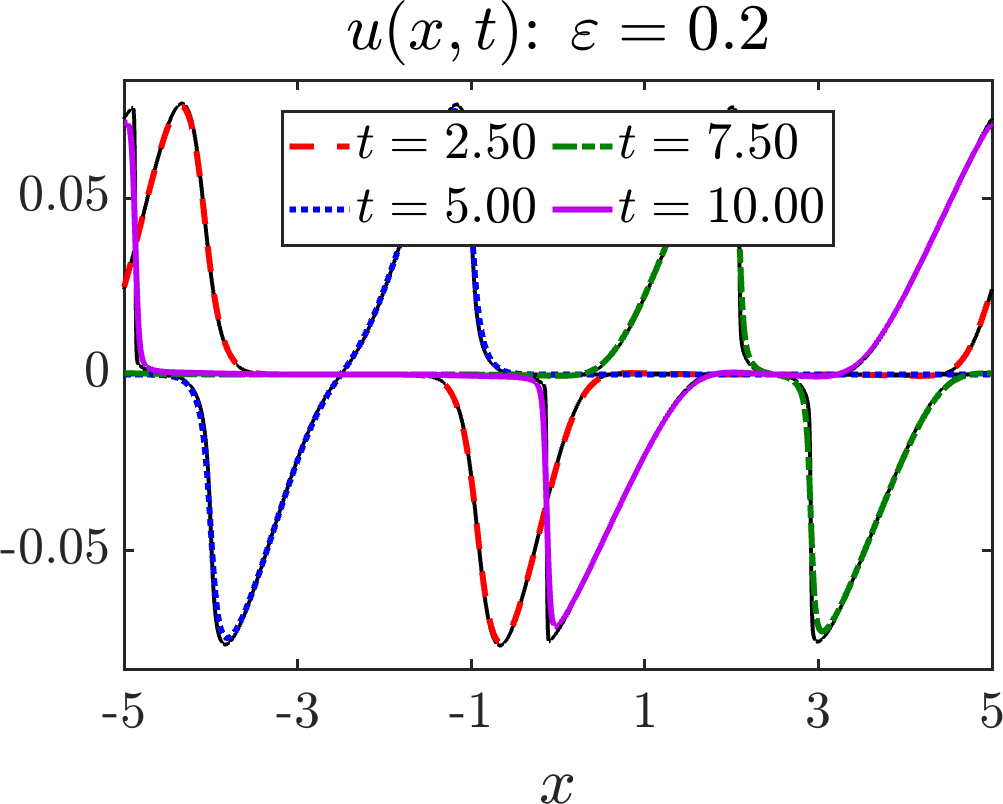}
}
\centerline{
\includegraphics[width=0.335\textwidth]{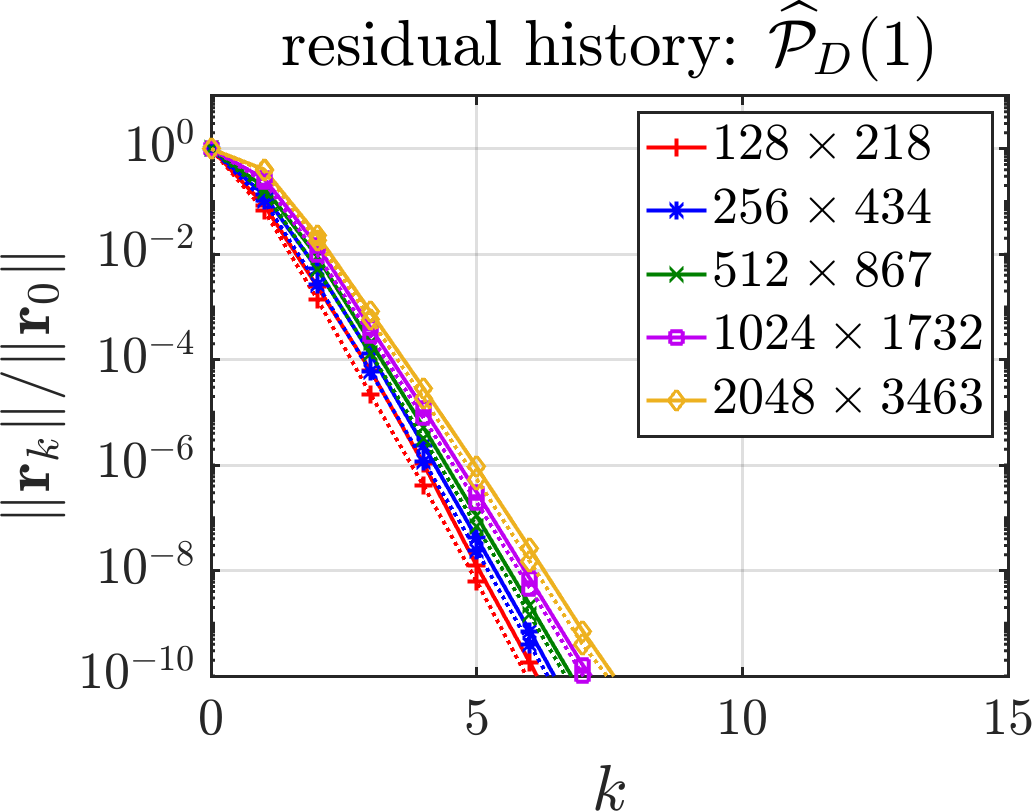}
\includegraphics[width=0.335\textwidth]{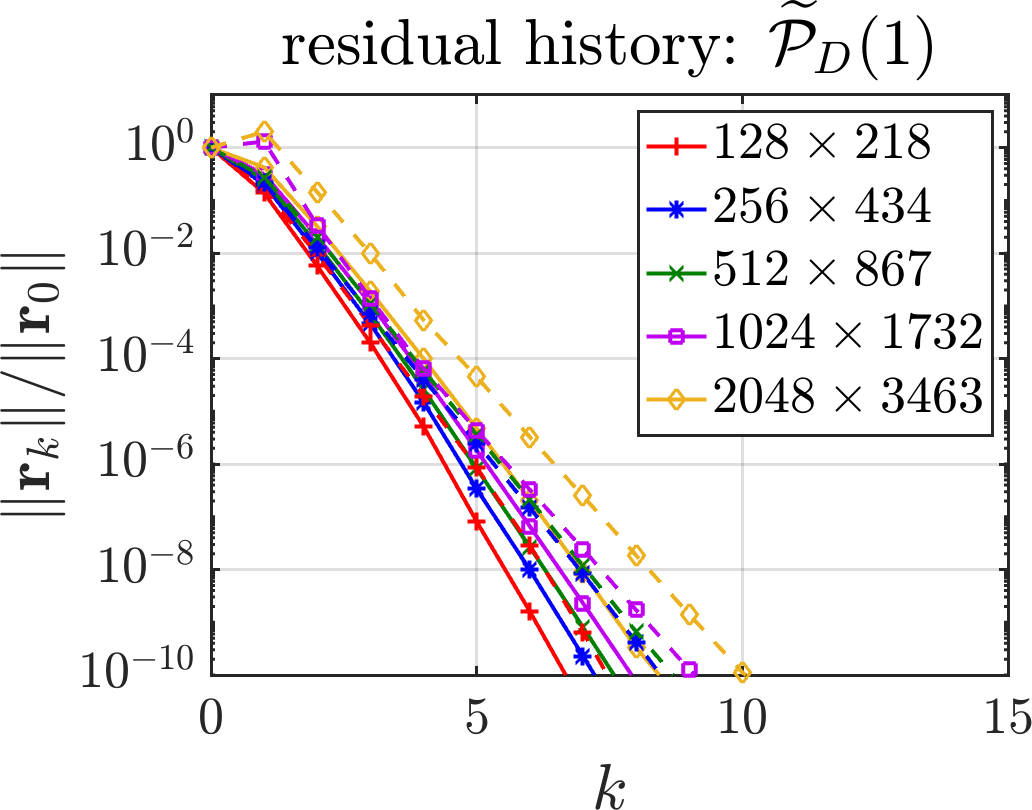}
}
\caption{
Small-amplitude \eqref{eq:euler-idp} problem for \eqref{eq:euler} with $\varepsilon = 0.2$.
\label{fig:euler-IDPP-weakly}
}
\end{figure}
\begin{figure}[t!]
\centerline{
\includegraphics[width=0.345\textwidth]{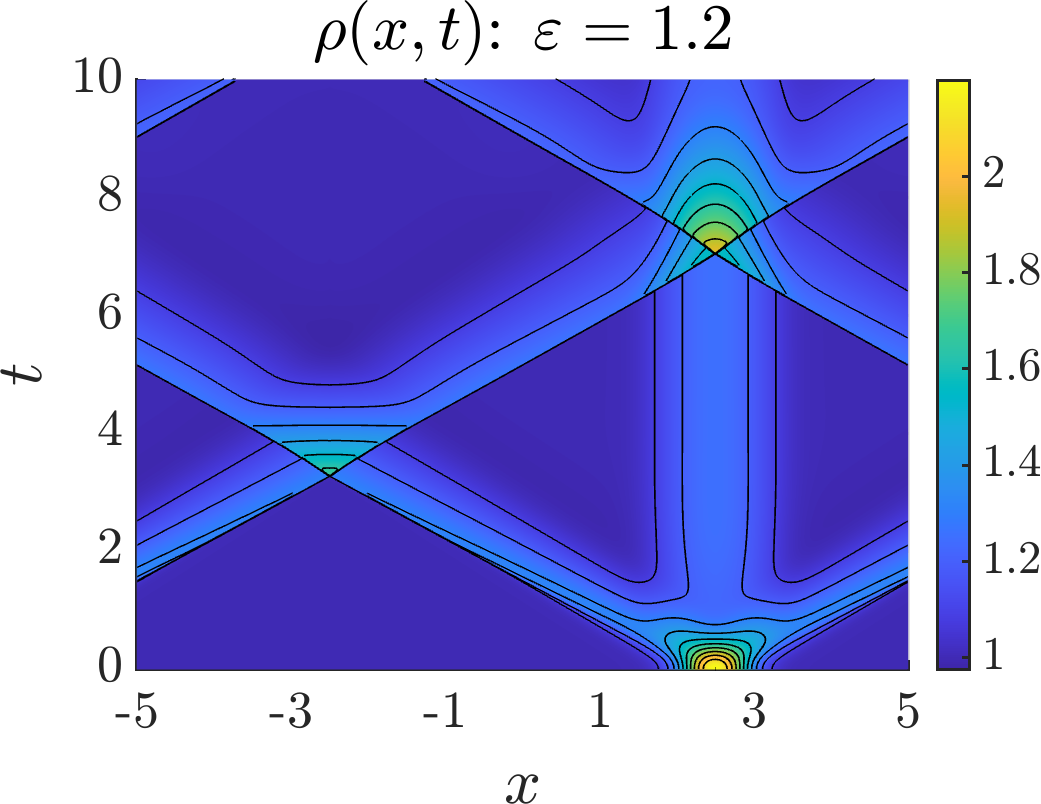}
\hspace{0ex}
\includegraphics[width=0.325\textwidth]{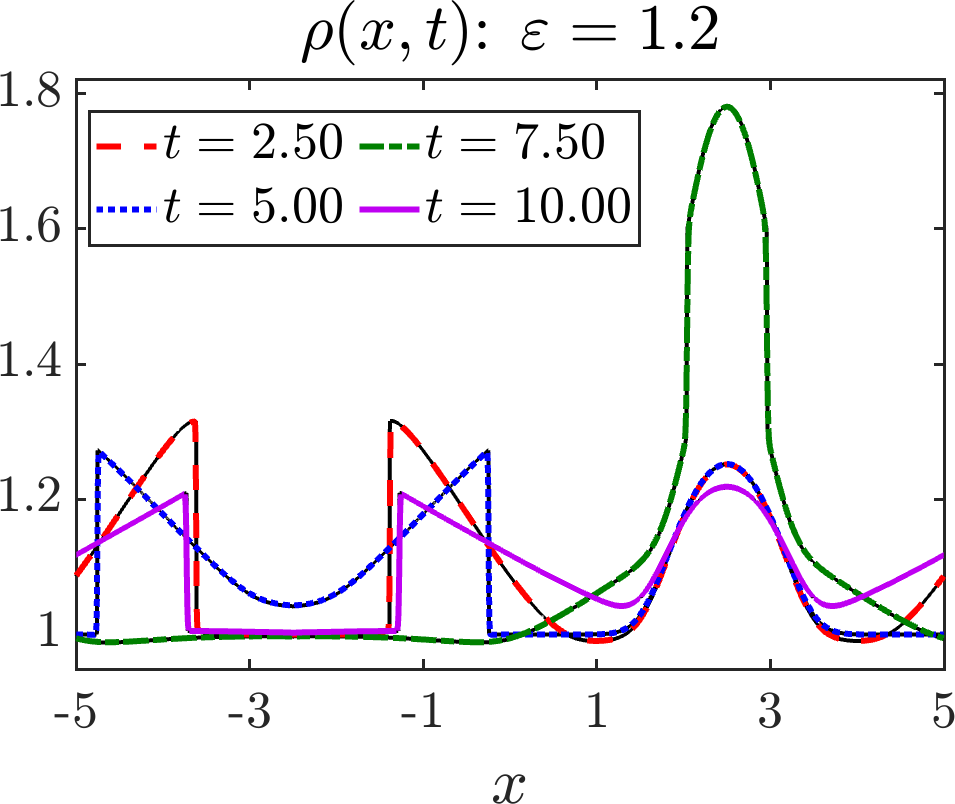}
\hspace{0ex}
\includegraphics[width=0.325\textwidth]{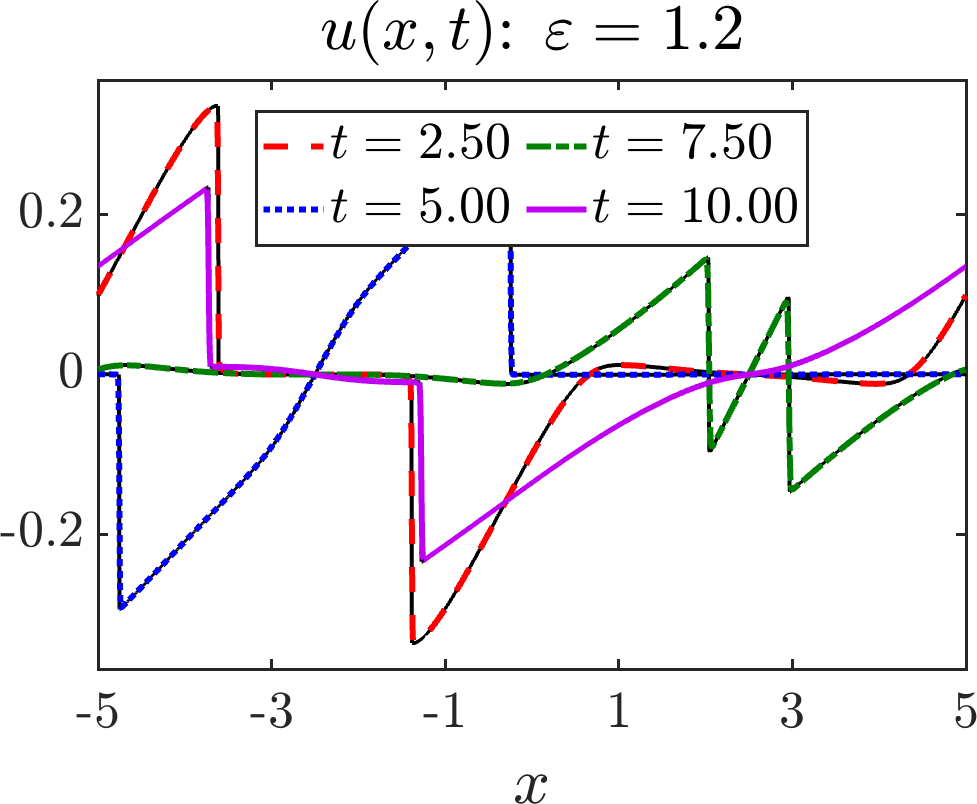}
}
\centerline{
\includegraphics[width=0.335\textwidth]{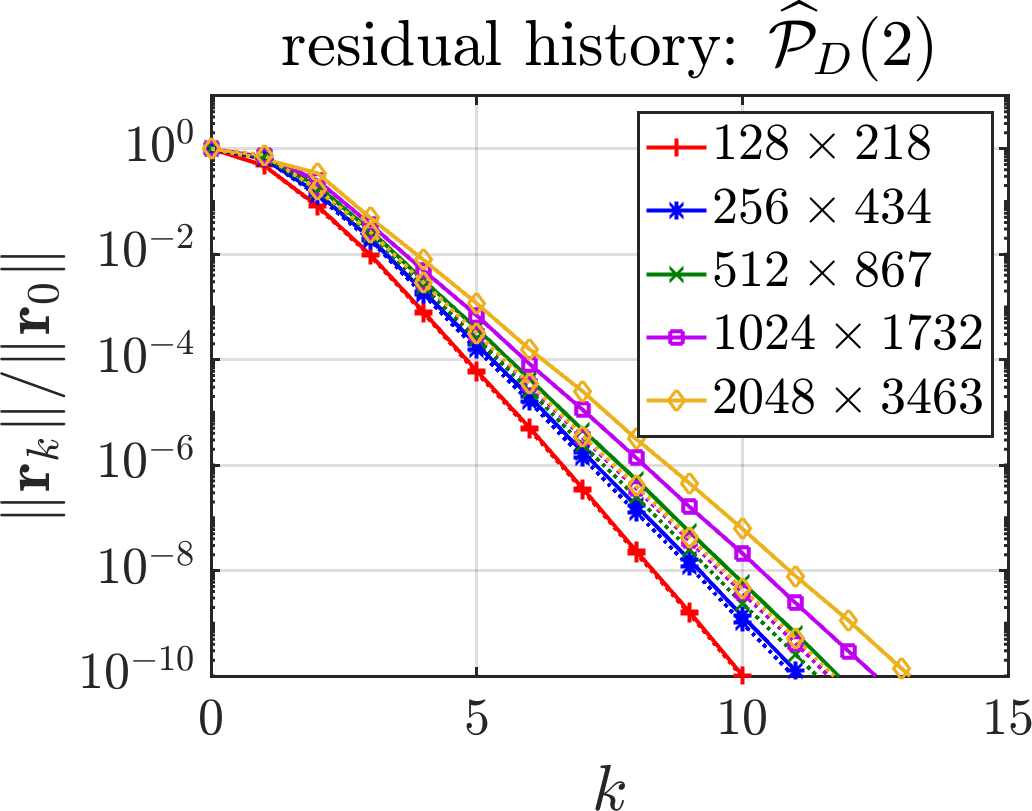}
\includegraphics[width=0.335\textwidth]{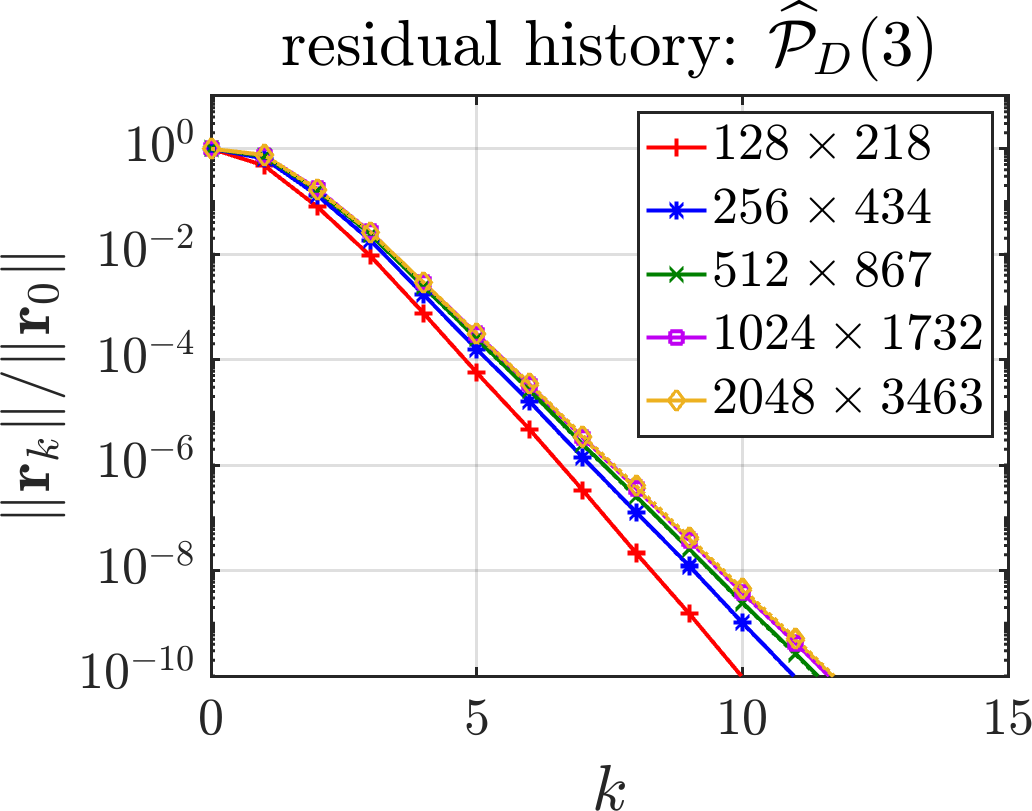}
\includegraphics[width=0.335\textwidth]{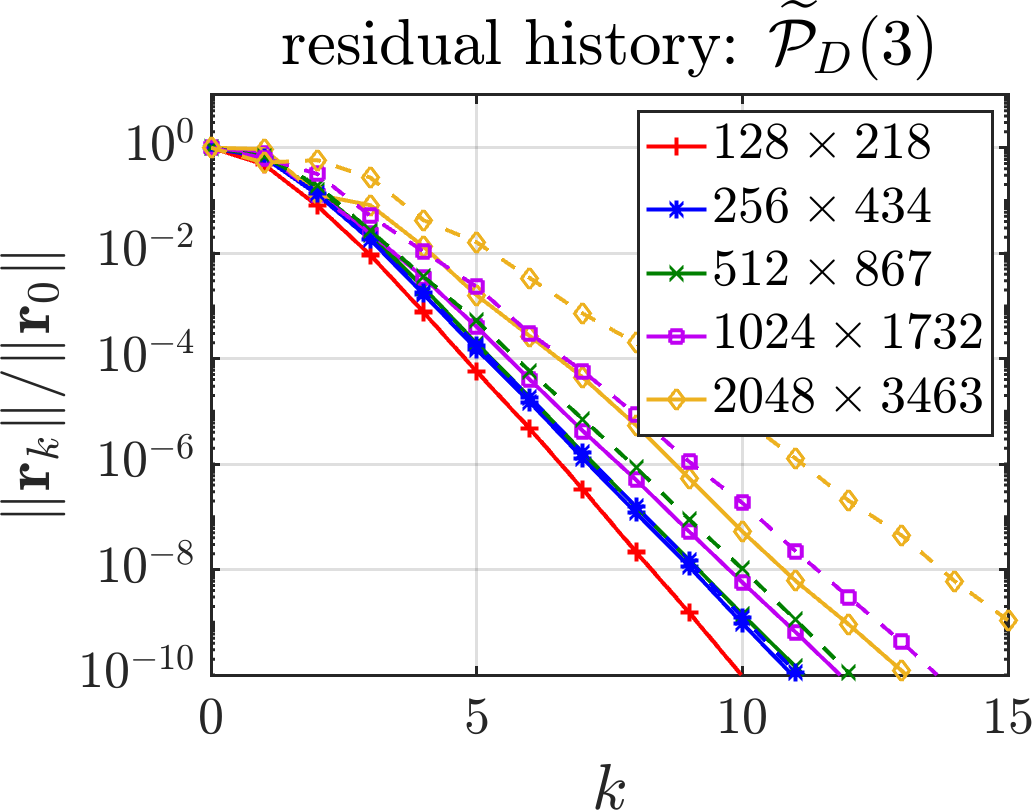}
}
\caption{
Larger-amplitude \eqref{eq:euler-idp} problem for \eqref{eq:euler} with $\varepsilon = 1.2$.
\label{fig:euler-IDPP-fully}
}
\end{figure}

For the larger-amplitude problems, we see that, unsurprisingly, the number of iterations increases relative to the small-amplitude problems, but we emphasize that the solver still converges very quickly, at least when direct solves of linearized problems are used.
When direct solves are replaced with block preconditioning, there is a slow degradation in convergence speed of the solver, with multiple inner iterations being required to restore scalability with respect to mesh size. 
In fact, in some cases the solver fails on sufficiently fine meshes if enough inner iterations are not used; for example, when $\wt{{\cal P}}$(1) is applied with exact inverses for $\wt{{\cal A}}_{ii}$ to the $(n_x, n_t) = (2048, 3463)$ problem in \cref{fig:euler-IDPP-fully}, the solver produces a negative density. 
As a second example, even $\wt{{\cal P}}$(5) produced a negative density for the larger-amplitude Sod problem in Supplementary Materials Figure \ref{SMfig:euler-sod-fully} with $(n_x, n_t) = (2048, 1348)$.
Furthermore, in the larger-amplitude problems, there is often a large deterioration in performance when moving $\wh{{\cal P}}$ to $\wt{{\cal P}}$.
Thus, while these results represent an advancement in the area of parallel-in-time methods for nonlinear hyperbolic systems, and the number of outer iterations is still quite small for these difficult nonlinear problems, they indicate that further research is required to develop methods capable of solving nonlinear problems with strong shocks with iteration numbers that are fully mesh independent.
%

\section{Conclusions and future outlook}
\label{sec:conclusion}

We have presented a framework for the all-at-once solution of systems of hyperbolic PDEs in one spatial dimension, both linear and nonlinear, that can admit parallelism in time.
The crux of the approach is a block preconditioner applied in the characteristic-variables associated with the linear(ized) PDE.
Transforming to characteristic variables is motivated by the observation that inter-variable coupling is generally weaker than in the original variables.
For an $\ell$-dimensional PDE system, application of the preconditioner requires solving $\ell$ scalar linear-advection-like PDEs, which can, at least in principle, be approximated parallel-in-time.
While we considered only serial implementations, the small number of MGRIT iterations required in our acoustics and small-amplitude nonlinear tests is consistent with obtaining parallel speed-ups over sequential time-stepping \cite{DeSterck_etal_2023_SL,DeSterck_etal_2023_MOL}.

In essence, the effectiveness of our approach boils down to the approximation of a block lower triangular space-time Schur complement.
In a certain sense, the complexity of this Schur complement relates to the coupling strength between characteristic variables in the linear(ized) PDE.
Our current approach deteriorates with increasingly strong spatio-temporal variation in the eigenvectors of the flux Jacobian, e.g., when transitioning from a problem with weak shocks to one with stronger shocks.
In future work we hope to improve the robustness of our solver against such variations.
Higher-order discretizations could also be considered, which for acoustics in heterogeneous media are often nonlinear due to limiters \cite{Fogarty-LeVeque-1999}, and for systems of nonlinear conservation laws there is potential based on our previous success with high-order WENO discretizations of scalar conservation laws \cite{DeSterck-etal-2023-nonlin-scalar}.
Other directions include parallel implementations, the application of these ideas to advection-diffusion PDE systems like the Navier--Stokes equations, and problems in multiple spatial dimensions, which may be handled by dimensionally splitting the problem into a sequence of one-dimensional space-time problems (with each then being solved with the method proposed herein).
%

\section*{Acknowledgments}

Critical feedback from anonymous referees is gratefully acknowledged.

\bibliographystyle{siamplain}
\bibliography{nonlinear-systems-1d}

\appendix

\section{Proof of \cref{lem:god-char-stencils}}
\label{app:proofs}

\begin{proof}
We begin by computing the triple-matrix product $\wh{\Phi} := {\cal R}_0^{-1} \Phi {\cal R}_0$,
where
\begin{align}
{\cal R}_0^{-1}
=
\frac{1}{2}
 \begin{bmatrix}
-{\cal Z}_0^{-1} & I \\
{\cal Z}_0^{-1} & I
\end{bmatrix}
,
\quad
\Phi 
=
\begin{bmatrix}
\Phi_{pp} & \Phi_{pu} \\
\Phi_{up} & \Phi_{uu}
\end{bmatrix},
\quad
{\cal R}_0
=
 \begin{bmatrix}
-{\cal Z}_0 & {\cal Z}_0 \\
I & I
\end{bmatrix}.
\end{align}
Multiplying this out gives
\begin{subequations} 
\begin{align}
\wh{\Phi}_{22}
&=
\tfrac{1}{2}
\Big[
\big(
\Phi_{uu} + {\cal Z}_0^{-1} \Phi_{pp} {\cal Z}_0
\big)
+
\big(
\Phi_{up} {\cal Z}_0
+
{\cal Z}_0^{-1} \Phi_{pu} 
\big)
\Big],
\\
\wh{\Phi}_{11}
&=
\tfrac{1}{2}
\Big[
\big(
\Phi_{uu} + {\cal Z}_0^{-1} \Phi_{pp} {\cal Z}_0
\big)
-
\big(
\Phi_{up} {\cal Z}_0
+
{\cal Z}_0^{-1} \Phi_{pu} 
\big)
\Big],
\\
\wh{\Phi}_{12}
&=
\tfrac{1}{2}
\Big[
\big(
\Phi_{uu} - {\cal Z}_0^{-1} \Phi_{pp} {\cal Z}_0
\big)
+
\big(
\Phi_{up} {\cal Z}_0
-
{\cal Z}_0^{-1} \Phi_{pu} 
\big)
\Big],
\\
\wh{\Phi}_{21}
&=
\tfrac{1}{2}
\Big[
\big(
\Phi_{uu} - {\cal Z}_0^{-1} \Phi_{pp} {\cal Z}_0
\big)
-
\big(
\Phi_{up} {\cal Z}_0
-
{\cal Z}_0^{-1} \Phi_{pu} 
\big)
\Big].
\end{align}
\end{subequations} 
Notice that the same four matrices appear in each of the four blocks here, i.e., those inside the open parentheses.
Plugging in the form of the blocks of $\Phi$ from \eqref{eq:acoustics-god-Phi} and using the fact that ${\cal C}_0$ and ${\cal Z}_0^{-1}$ commute due to their both being diagonal gives
\begin{subequations} \label{eq:acoustics-whPhi-simplified}
\begin{alignat}{2}
\wh{\Phi}_{22}
&=
I 
&&- \frac{\delta t}{2 h} {\cal C}_0
\Big[
\big( {\cal L}_{uu} + {\cal Z}_0^{-1} {\cal L}_{pp} {\cal Z}_0 \big)
+
\big( {\cal L}_{up} {\cal Z}_0 + {\cal Z}_0^{-1} {\cal L}_{pu} \big)
\Big],
\\
\wh{\Phi}_{11}
&=
I 
&&- \frac{\delta t}{2 h} {\cal C}_0
\Big[
\big( {\cal L}_{uu} + {\cal Z}_0^{-1} {\cal L}_{pp} {\cal Z}_0 \big)
-
\big( {\cal L}_{up} {\cal Z}_0 + {\cal Z}_0^{-1} {\cal L}_{pu} \big)
\Big],
\\
\wh{\Phi}_{12}
&=
&&- \frac{\delta t}{2 h} {\cal C}_0
\Big[
\big( {\cal L}_{uu} - {\cal Z}_0^{-1} {\cal L}_{pp} {\cal Z}_0 \big)
+
\big( {\cal L}_{up} {\cal Z}_0 - {\cal Z}_0^{-1} {\cal L}_{pu} \big)
\Big],
\\
\wh{\Phi}_{21}
&=
&&- \frac{\delta t}{2 h} {\cal C}_0
\Big[
\big( {\cal L}_{uu} - {\cal Z}_0^{-1} {\cal L}_{pp} {\cal Z}_0 \big)
-
\big( {\cal L}_{up} {\cal Z}_0 - {\cal Z}_0^{-1} {\cal L}_{pu} \big)
\Big].
\end{alignat}
\end{subequations} 
Now we compute ${\cal Z}_0^{-1} {\cal L}_{pp}  {\cal Z}_0$, ${\cal Z}_0^{-1} {\cal L}_{pu}$, ${\cal L}_{up} {\cal Z}_0$, and ${\cal L}_{uu}$, with stencils for ${\cal L}_{ij}$ in \eqref{eq:acoustics-L-stencils}.
Since ${\cal Z}_0^{-1}$ is a diagonal matrix, left-multiplying a matrix by it scales the rows of that matrix by its diagonal entries, while right-multiplying a matrix by ${\cal Z}_0$ scales the columns of that matrix by the diagonal entries of ${\cal Z}_0$.
Thus,
\begin{subequations} \label{eq:acoustics-L-stencils-scaled}
\begin{align}
\big[  {\cal Z}_0^{-1} {\cal L}_{pp}  {\cal Z}_0 \big]_i
&=
\left[
- \frac{Z_{i-1}}{Z_{i-1} + Z_{i}}, \,
   \frac{Z_i}{Z_{i-1} + Z_{i}} + \frac{Z_{i}}{Z_{i} + Z_{i+1}}, \,
- \frac{Z_{i+1}}{Z_{i} + Z_{i+1}}
\right],
\\
\big[ {\cal Z}_0^{-1} {\cal L}_{pu} \big]_i
&=
\left[
- \frac{Z_{i-1}}{Z_{i-1} + Z_{i}}, \,
   \frac{Z_{i-1}}{Z_{i-1} + Z_{i}} - \frac{Z_{i+1}}{Z_{i} + Z_{i+1}}, \,
   \frac{Z_{i+1}}{Z_{i} + Z_{i+1}}
\right],
\\
\big[ {\cal L}_{up} {\cal Z}_0 \big]_i
&=
\left[
- \frac{Z_{i-1}}{Z_{i-1} + Z_{i}}, \,
   \frac{Z_{i}}{Z_{i-1} + Z_{i}} - \frac{Z_{i}}{Z_{i} + Z_{i+1}}, \,
   \frac{Z_{i+1}}{Z_{i} + Z_{i+1}}
\right],
\\
\big[ {\cal L}_{uu} \big]_i
&=
\left[
- \frac{Z_{i-1}}{Z_{i-1} + Z_{i}}, \,
   \frac{Z_{i-1}}{Z_{i-1} + Z_{i}} + \frac{Z_{i+1}}{Z_{i} + Z_{i+1}}, \,
- \frac{Z_{i+1}}{Z_{i} + Z_{i+1}}
\right].
\end{align}
\end{subequations}
Using the above expressions, we now compute the stencils of the terms in \eqref{eq:acoustics-whPhi-simplified}:
\begin{subequations}
\begin{align}
\big[ {\cal L}_{uu} + {\cal Z}_0^{-1} {\cal L}_{pp}  {\cal Z}_0 \big]_i
&=
\left[
- \frac{2 Z_{i-1}}{Z_{i-1} + Z_{i}}, \,
   2, \,
- \frac{2 Z_{i+1}}{Z_{i} + Z_{i+1}}
\right],
\\
\big[ {\cal L}_{uu} - {\cal Z}_0^{-1} {\cal L}_{pp}  {\cal Z}_0 \big]_i
&=
\left[
	0, \,
   \frac{Z_{i-1} - Z_{i}}{Z_{i-1} + Z_{i}} - \frac{Z_{i} - Z_{i+1}}{Z_{i} + Z_{i+1}}, \,
    0
\right],
\\
\big[ {\cal L}_{up} {\cal Z}_0 + {\cal Z}_0^{-1} {\cal L}_{pu}   \big]_i
&=
\left[
- \frac{2 Z_{i-1}}{Z_{i-1} + Z_{i}}, \,
   0, \,
   \frac{2 Z_{i+1}}{Z_{i} + Z_{i+1}}
\right],
\\
\big[ {\cal L}_{up} {\cal Z}_0 - {\cal Z}_0^{-1} {\cal L}_{pu}   \big]_i
&=
\left[
	0, \,
   -\frac{Z_{i-1} - Z_{i}}{Z_{i-1} + Z_{i}} - \frac{Z_{i} - Z_{i+1}}{Z_{i} + Z_{i+1}}, \,
   0
\right].
\end{align}
\end{subequations}
Plugging these expressions into \eqref{eq:acoustics-whPhi-simplified} and simplifying concludes the proof. 
\end{proof}


\newpage

\setcounter{section}{0}
\setcounter{equation}{0}
\setcounter{figure}{0}
\setcounter{table}{0}
\setcounter{page}{1}
\makeatletter
\renewcommand{\thesection}{SM\arabic{section}}
\renewcommand{\theequation}{SM\arabic{equation}}
\renewcommand{\thefigure}{SM\arabic{figure}}
\renewcommand{\thetable}{SM\arabic{table}}
\renewcommand{\thepage}{SM\arabic{page}}

\thispagestyle{plain} 

\headers{SUPPLEMENTARY MATERIALS: PinT solution of hyperbolic systems}{H. De Sterck, R. D. Falgout, O. A. Krzysik, J. B. Schroder}

\begin{center}
    \textbf{\normalsize\MakeUppercase{
    Supplementary Materials:
    Parallel-in-time solution of hyperbolic PDE systems via characteristic-variable block preconditioning}} 
    \vspace{6ex}
\end{center}

\section{Block preconditioning in primitive variables}
\label{SMsec:block-prec-prim}

In the main paper, to solve space-time systems arising from the discretization of systems of hyperbolic PDEs we propose a block-preconditioning strategy that first transforms to characteristic variables before applying block preconditioning techniques.
We argue that this transformation to characteristic variables is justified because the inter-variable coupling strength between characteristic variables is weak relative to that of the underlying primitive variables.
In fact, inter-variable coupling between primitive variables is \textit{not} weak, so that we cannot reasonably expect block preconditioning techniques to be effective on them.
In this section we further justify this line of reasoning by numerical experiments wherein we apply block preconditioning to the primitive variables of the acoustics problem (see Sections \ref{sec:acoustics} and \ref{sec:acoustics-pint}).
Recall from \eqref{eq:acoustics-god-Phi} that the fully discretized acoustics problem takes the form
\begin{align} \label{SMeq:acoustics-god-Phi}
\bm{q}^{n+1}
=
\Phi
\bm{q}^{n},
\quad
\Phi 
=
\begin{bmatrix}
\Phi_{pp} & \Phi_{pu} \\
\Phi_{up} & \Phi_{uu}
\end{bmatrix},
\quad
n = 0, 1, \ldots, n_t - 1,
\end{align}
with stacked spatial vector $\bm{q}^n := (\bm{p}^n, \bm{u}^n)^{\top} \in \mathbb{R}^{2 n_x}$, with $\bm{p}^n = (p_1^n, \ldots, p_{n_x}^n)^\top, \bm{u}^n = (u_1^n, \ldots, u_{n_x}^n)^\top$.
Further recall from \eqref{eq:godunov-all-at-once} that the global space-time system for the acoustic equations takes the form
\begin{align} \label{SMeq:godunov-all-at-once}
{\cal A} \bm{q} 
\equiv
\begin{bmatrix}
I & \\
-\Phi & I \\
& \ddots & \ddots \\
& & -\Phi & I
\end{bmatrix}
\begin{bmatrix}
\bm{q}^0 \\
\bm{q}^1 \\
\vdots \\
\bm{q}^{n_t-1} 
\end{bmatrix}
=
\begin{bmatrix}
\bm{q}^0 \\
\bm{0} \\
\vdots \\
\bm{0}
\end{bmatrix}
=:
\bm{b},
\end{align}
where ${\cal A} \in \mathbb{R}^{2 n_x n_t \times 2 n_x n_t}$, $\bm{q} \in \mathbb{R}^{2 n_x n_t}$, and $\bm{b} \in \mathbb{R}^{2 n_x n_t}$ are the space-time discretization matrix, space-time solution vector, and right-hand side vector, respectively.
Given an iterate $\bm{q}_k \approx \bm{q}$, the algebraic residual of \eqref{SMeq:godunov-all-at-once} is
\begin{align} \label{SMeq:rk}
\bm{r}_k = \bm{b} - {\cal A} \bm{q}_k.
\end{align}

A preconditioned residual correction scheme computes iterates according to \\${\bm{q}_{k + 1} = \bm{q}_k + {\cal P}^{-1} \bm{r}_k}$, where ${\cal P}^{-1} \bm{r}_k \approx {\cal A}^{-1} \bm{r}_k = \bm{e}_k = \bm{q} - \bm{q}_k$ is the algebraic error. 
Specifically in this section we consider preconditoners based on block preconditioning in the primitive variables $p$ and $u$.
To this end, let ${\cal Q} \in \mathbb{R}^{2 n_x n_t \times 2  n_x n_t}$ be a permutation matrix such that ${\cal Q} \bm{q} = (\bm{q}_p, \bm{q}_u)^\top$, where 
\begin{align}
\bm{q}_p = ( \bm{p}^0, \ldots, \bm{p}^{n_t-1} )^\top,
\quad
\bm{q}_u = ( \bm{u}^0, \ldots, \bm{u}^{n_t-1} )^\top.
\end{align}
Then we can re-order the system \eqref{SMeq:godunov-all-at-once} so that all time points for each primitive variable are blocked together:
\begin{align} \label{SMeq:reordered-2x2}
\left( {\cal Q} {\cal A} {\cal Q}^{\top} \right)
\left( {\cal Q} \bm{q} \right)
=
\begin{bmatrix}
{{\cal A}}_{pp} & {{\cal A}}_{pu} \\
{{\cal A}}_{up} & {{\cal A}}_{uu}
\end{bmatrix}
\begin{bmatrix}
\bm{q}_p \\
\bm{q}_u
\end{bmatrix}
=
\begin{bmatrix}
\bm{b}_p \\
\bm{b}_u
\end{bmatrix}
=
\left({\cal Q} \bm{b} \right),
\end{align}
where the $n_x n_t \times n_x n_t$ blocks in the block $2 \times 2$ matrix shown here are
\begin{align} 
{\cal A}_{pp}
=
\begin{bmatrix}
I \\
-{\Phi}_{pp} & I \\
& \ddots & \ddots 
\end{bmatrix},
\quad
{{\cal A}}_{pu}
=
\begin{bmatrix}
0 \\
-{\Phi}_{pu} & 0 \\
& \ddots & \ddots 
\end{bmatrix},
\\
{\cal A}_{uu}
=
\begin{bmatrix}
I \\
-{\Phi}_{uu} & I \\
& \ddots & \ddots 
\end{bmatrix},
\quad
{\cal A}_{up}
=
\begin{bmatrix}
0 \\
-{\Phi}_{up} & 0 \\
& \ddots & \ddots 
\end{bmatrix}.
\end{align}

\begin{figure}[b!]
\centerline{
\includegraphics[width=0.475\textwidth]{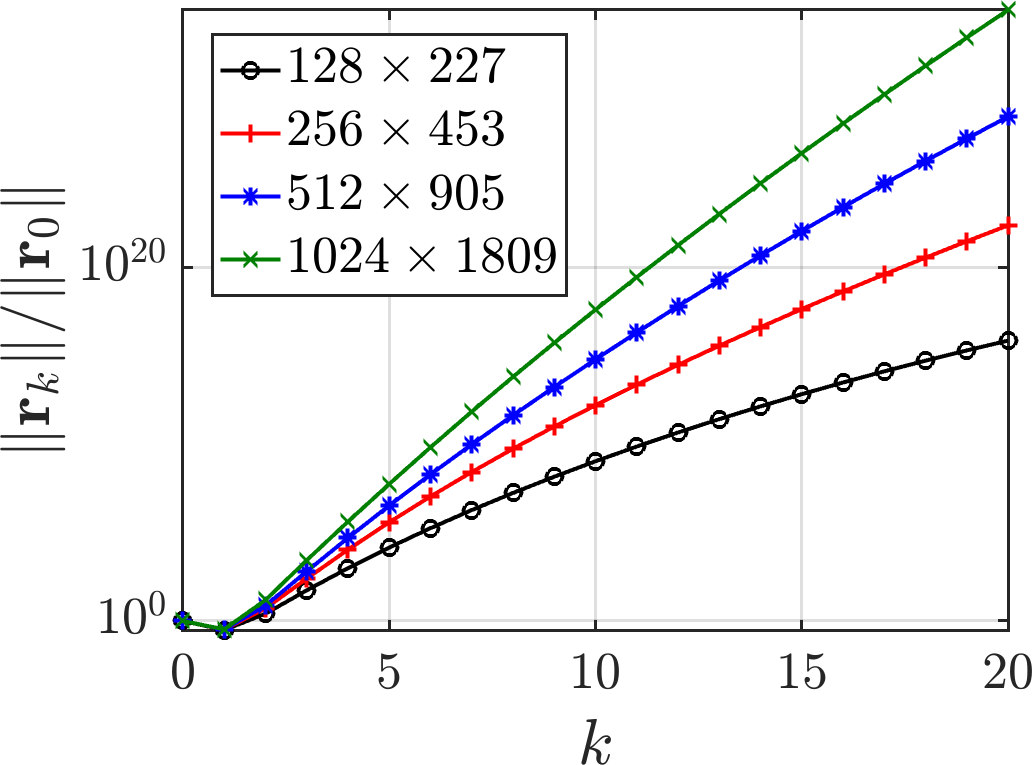}
\hspace{1ex}
\includegraphics[width=0.475\textwidth]{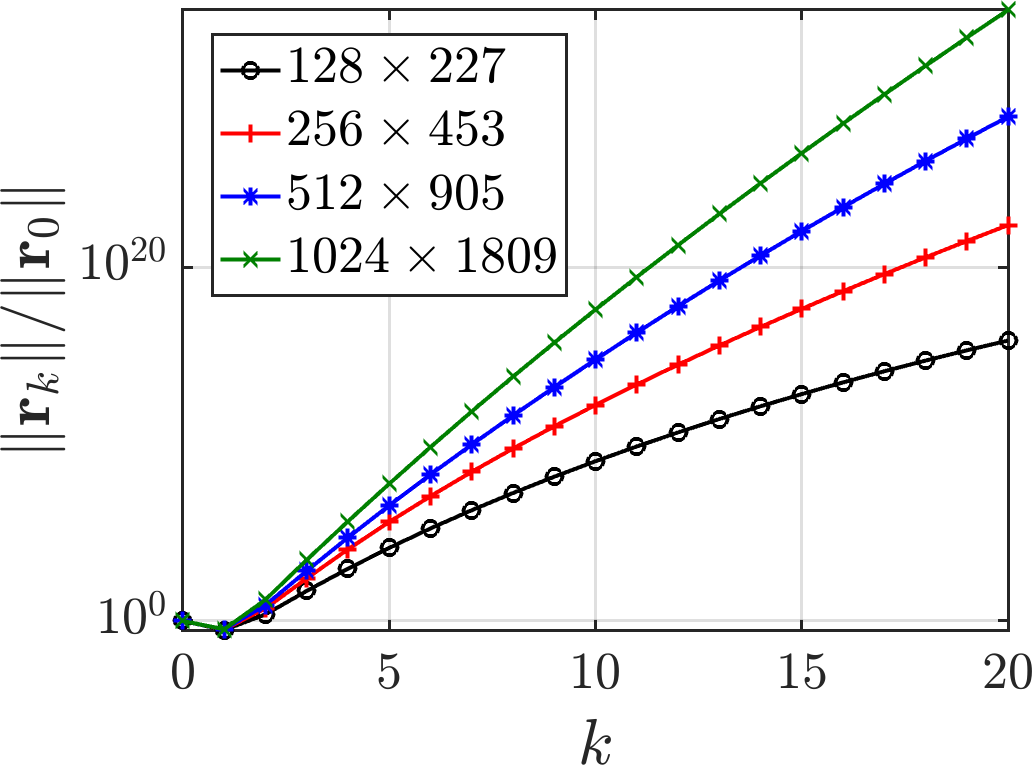}
}
\caption{
Block preconditioning in primitive variables for \eqref{eq:acoustic}. Shown is the residual history for the preconditioned residual scheme \eqref{SMeq:res-iter} when using the block Gauss--Seidel preconditioner \eqref{SMeq:P-GS}.
\textbf{Left:} Acoustics material parameters given by \eqref{eq:mat-param-1}. \textbf{Right:} Acoustics material parameters given by 
\eqref{eq:mat-param-2}.
Legends entries correspond to space-time mesh resolutions of $n_x \times n_t$.
\label{SMfig:acoustic-prec-prim}
}
\end{figure}

We can now create a block preconditioner ${\cal P}^{-1} \approx
 \begin{bmatrix}
{{\cal A}}_{pp} & {{\cal A}}_{pu} \\
{{\cal A}}_{up} & {{\cal A}}_{uu}
\end{bmatrix}^{-1}$, and consider the associated preconditioned residual scheme:
\begin{align} \label{SMeq:res-iter}
\bm{q}_{k+1}
=
\bm{q}_{k}
+
\big( {\cal Q}^\top {\cal P}^{-1} {\cal Q} \big) \bm{r}_k,
\quad
k = 0, 1, \ldots,
\end{align}
where $\bm{r}_k$ is the algebraic residual from \eqref{SMeq:rk}.
For simplicity, in our tests we just consider the block Gauss--Seidel preconditioner given by
\begin{align} \label{SMeq:P-GS}
{\cal P}^{-1}
=
\begin{bmatrix}
{\cal A}_{pp} & 0 \\
{\cal A}_{up} & {\cal A}_{uu}
\end{bmatrix}^{-1}.
\end{align}
We consider two sets of material parameters for the acoustics system: Those given in \eqref{eq:mat-param-1} and \eqref{eq:mat-param-2}.
Two-norm residual histories for the residual iteration \eqref{SMeq:res-iter} are shown in \cref{SMfig:acoustic-prec-prim}.
In all tests the initial iterate $\bm{q}_0$ is chosen with entries drawn randomly from a standard normal distribution.
In all tests the preconditoner ${\cal P}^{-1}$ is applied exactly, i.e., the diagonal blocks ${\cal A}_{pp}$ and ${\cal A}_{uu}$ are inverted exactly via time-stepping, and not approximately with MGRIT.
For both sets of material parameters the iteration diverges at all mesh resolutions.
We conclude that block preconditoning in the primitive variables is not effective, especially when compared to that in characteristic variables which is extremely effective on these problems (see Section \ref{sec:acoustics-num-res}).
This outcome is not surprising given that inter-variable coupling is not weak in the primitive variables, i.e., the ${\cal A}_{pu}$ block that is dropped in the preconditioner \eqref{SMeq:P-GS} is not small relative to the other blocks that are retained (see \eqref{eq:acoustics-god-Phi}).
%

\section{Flux Jacobians and their eigen decompositions}
\label{SMsec:flux-jacobians}
The flux Jacobian of \eqref{eq:SWE} is
\begin{align}
f'(\bm{q})
=
\begin{bmatrix}
0 & 1 \\
-u^2 + h & 2 u
\end{bmatrix}.
\end{align}
The eigenvalues of this matrix are $\lambda^1 = u - \sqrt{h}$ and $\lambda^2 = u + \sqrt{h}$, with corresponding right and left eigenvector matrices given by
\begin{align}
R
=
\begin{bmatrix}
1 & 1 \\
\lambda^1 & \lambda^2
\end{bmatrix},
\quad
R^{-1}
=
\frac{1}{\lambda^2 - \lambda^1}
\begin{bmatrix}
\lambda^2 & -1 \\
-\lambda^1 & 1
\end{bmatrix},
\end{align}
respectively.
The flux Jacobian of \eqref{eq:euler} is 
\begin{align}
f'(\bm{q})
=
\begin{bmatrix}
0 & 1 & 0 \\
\tfrac{1}{2} (\gamma - 3) u^2 & (3 - \gamma) u & \gamma - 1 \\
\tfrac{1}{2} (\gamma - 1) u^3 - u H & H - (\gamma - 1) u^2 & \gamma u
\end{bmatrix},
\end{align}
where $H = \frac{E + p}{\rho}$ is the total specific enthalpy, and $c = \sqrt{\frac{\gamma p}{\rho}}$ is the sound speed.
The eigenvalues are $\lambda^1 = u - c$, $\lambda^2 = u$, and $\lambda^3 = u + c$, with the associated right and left eigenvector matrices given by
\begin{align}
R
=
\begin{bmatrix}
1 & 1 & 1 \\
u - c & u &  u + c \\
H - u c & \tfrac{1}{2} u^2 & H + u c
\end{bmatrix},
\quad
R^{-1}
=
\tfrac{\gamma - 1}{2 c^2}
\begin{bmatrix}
H + \tfrac{c (u - c)}{\gamma - 1}  & -u - \tfrac{c}{\gamma - 1} & 1 \\
\tfrac{4 c^2}{\gamma - 1} - 2H  & 2u & -2 \\
H - \tfrac{c (u + c)}{\gamma - 1}  & -u + \tfrac{c}{\gamma - 1} & 1
\end{bmatrix},
\end{align}
respectively.

\newpage
\section{Additional numerical results for nonlinear problems}
\label{SMsec:num-res-nonlin}
Here we consider additional numerical results to complement those already in Section \ref{sec:cons-num-res}.
In particular we consider standard Riemann problems.
The layout of the plots in the figures below is the same as in Section \ref{sec:cons-num-res}, with the exception that residual histories with $\wt{\cal P}$(\texttt{inner-it}) in the title no longer include dashed lines corresponding to inexact inverses of $\wt{{\cal A}}_{ii}$ via MGRIT.

The two problems we consider are as follows. First, consider \eqref{eq:SWE} for the dam break problem:
\begin{align}
\tag{DB}
\label{SMeq:SWE-db}
h(x, 0) &= 
\begin{cases}
1 + \varepsilon, \quad &x < 0 \\
1 \quad &x \geq 0,
\end{cases} 
\quad
u(x, 0) = 0,
\end{align}
with $(x, t) \in (-10, 10) \times (0, 5)$, $c_0 = 0.7$, and with constant boundary conditions.
Small- and larger-amplitude results are shown in \cref{SMfig:SWE-DB-weakly,SMfig:SWE-DB-fully}, respectively.
Second, consider the following Sod problem for \eqref{eq:euler}:
\begin{align}
\tag{Sod}
\label{SMeq:euler-sod}
\rho(x, 0) = 
\begin{cases}
1, \quad &x < 0.5 \\
1 - \varepsilon, \quad &x \geq 0.5
\end{cases},
\quad
u(x, 0) = 0,
\quad
E(x, 0) = 
\frac{\rho(x, 0)}{\gamma - 1},
\end{align}
with $(x, t) \in (0, 1) \times (0, 0.25)$, $c_0 = 0.45$, and with constant boundary conditions. 
Small- and larger-amplitude results are shown in \cref{SMfig:euler-sod-weakly,SMfig:euler-sod-fully}, respectively.

\begin{figure}[h!]
\centerline{
\includegraphics[width=0.345\textwidth]{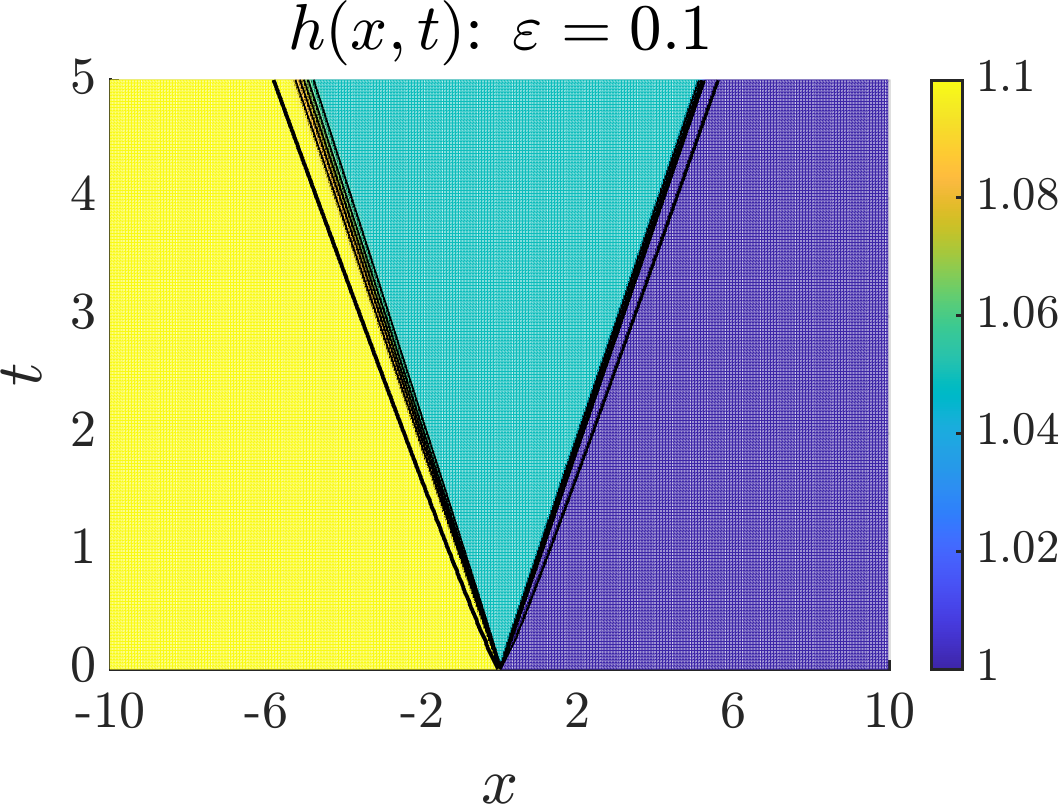}
\hspace{0ex}
\includegraphics[width=0.325\textwidth]{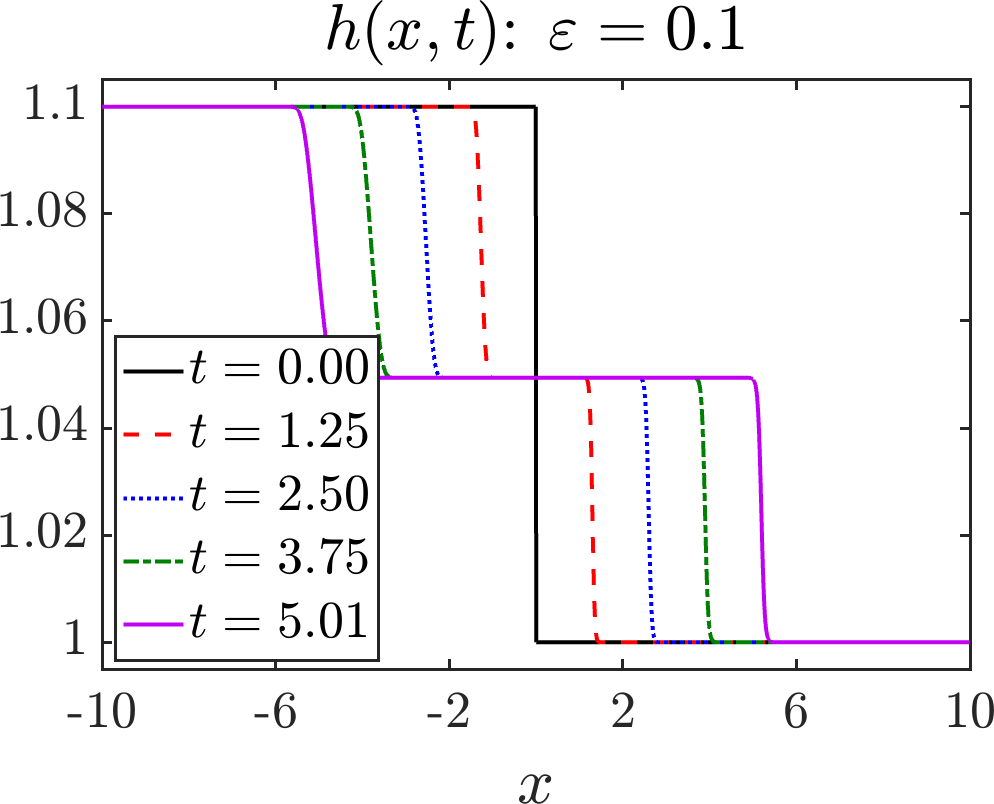}
\hspace{0ex}
\includegraphics[width=0.325\textwidth]{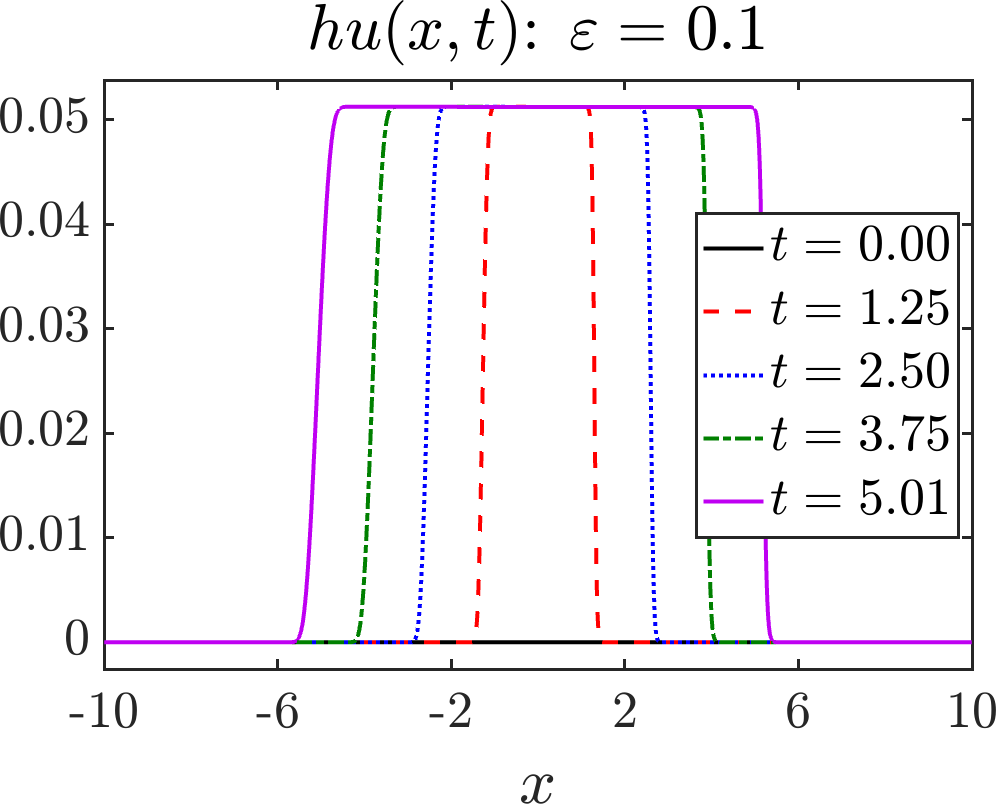}
}
\vspace{1ex}
\centerline{
\includegraphics[width=0.335\textwidth]{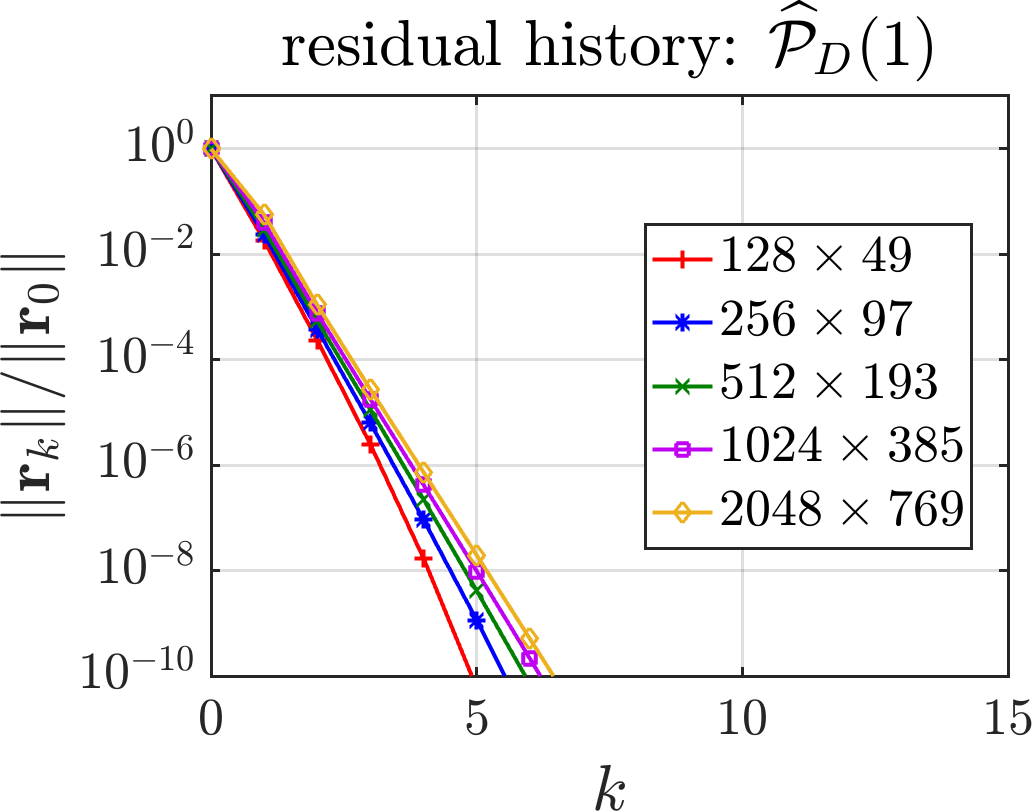}
\hspace{1ex}
\includegraphics[width=0.335\textwidth]{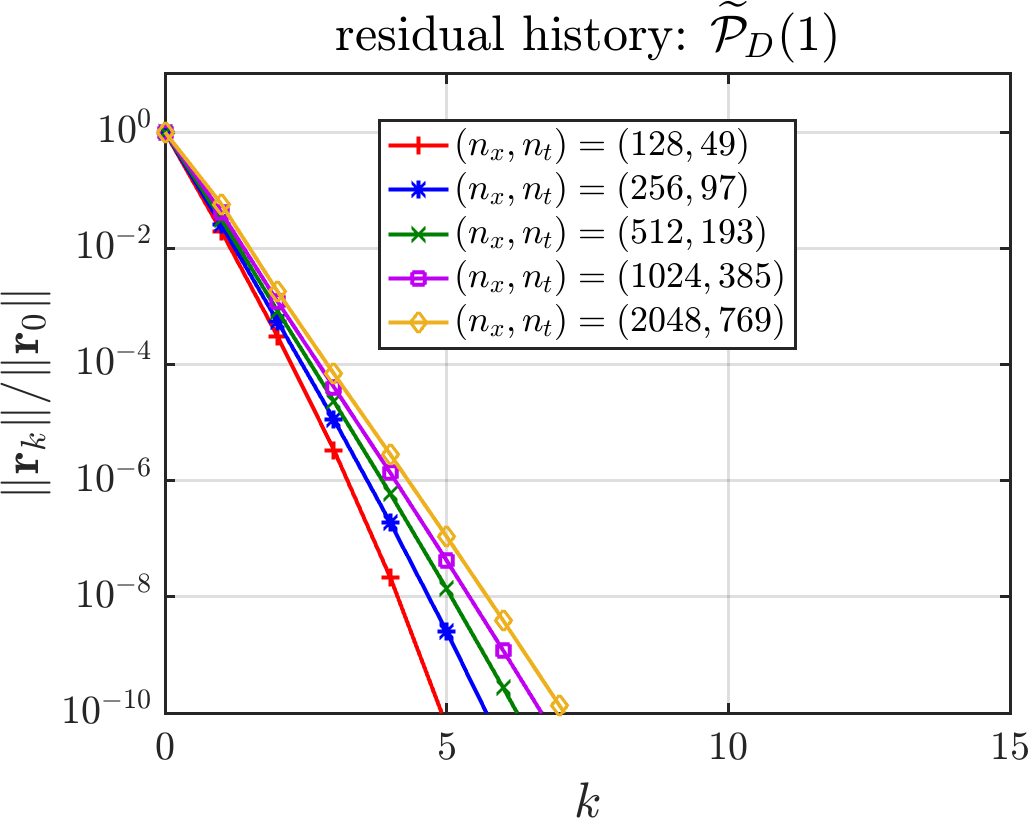}
}
\caption{
Small-amplitude dam break problem \eqref{SMeq:SWE-db} for \eqref{eq:SWE}, with $\varepsilon = 0.1$.
\label{SMfig:SWE-DB-weakly}
}
\end{figure}

\begin{figure}[t!]
\centerline{
\includegraphics[width=0.345\textwidth]{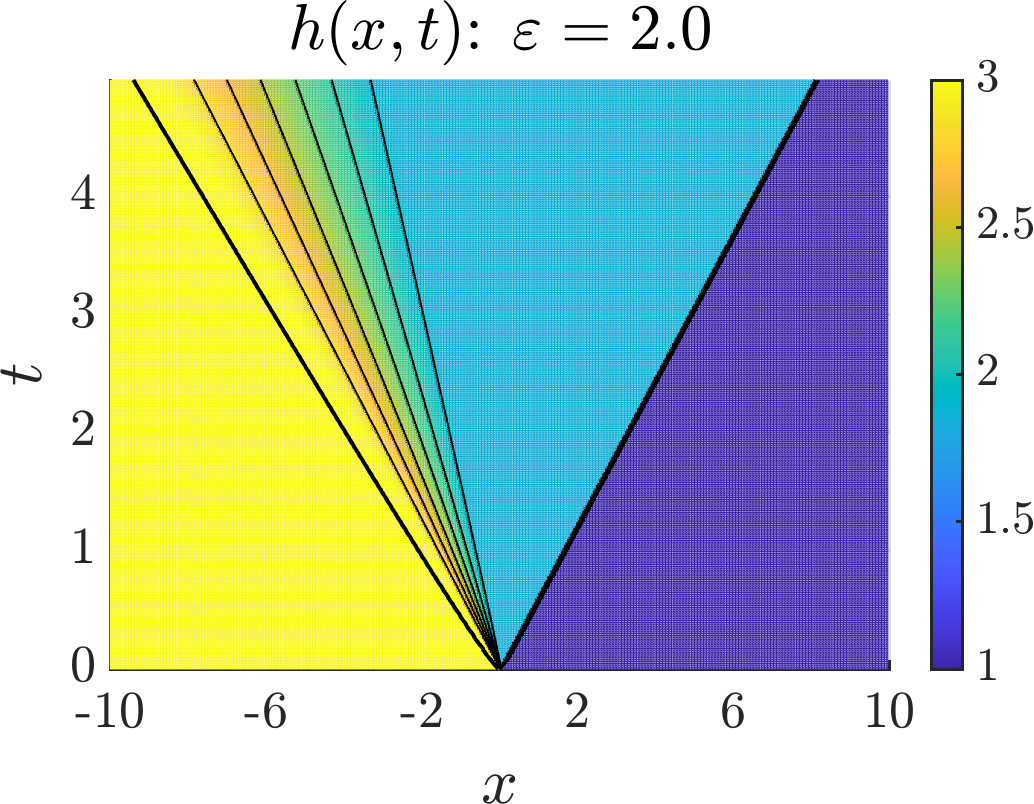}
\hspace{0ex}
\includegraphics[width=0.325\textwidth]{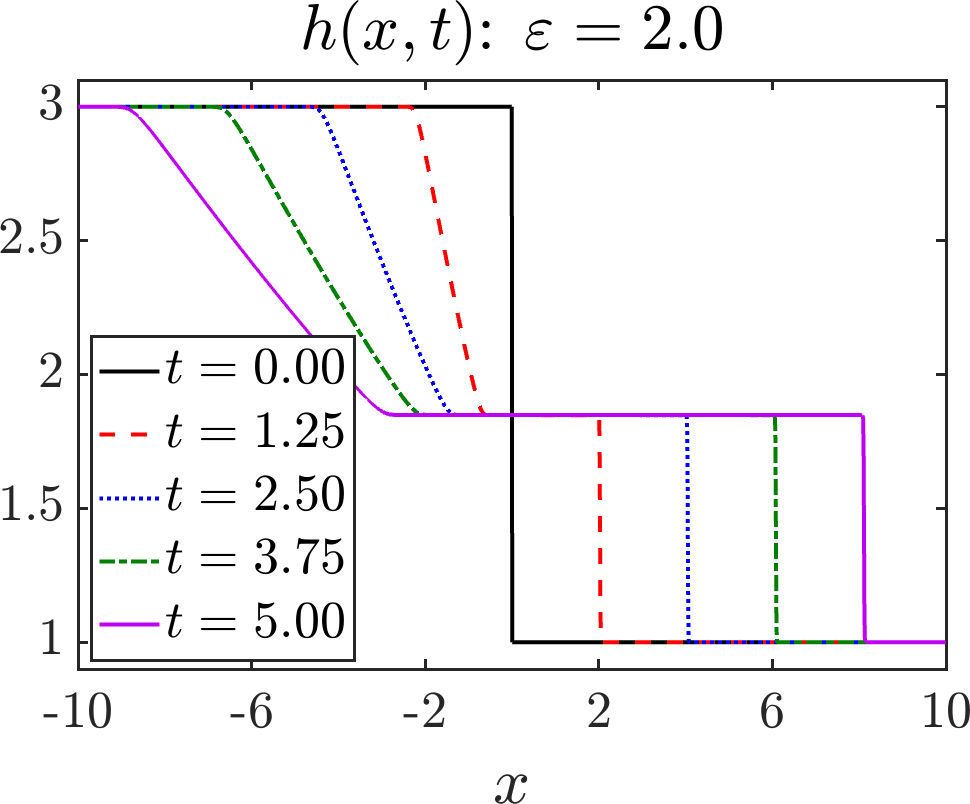}
\hspace{0ex}
\includegraphics[width=0.325\textwidth]{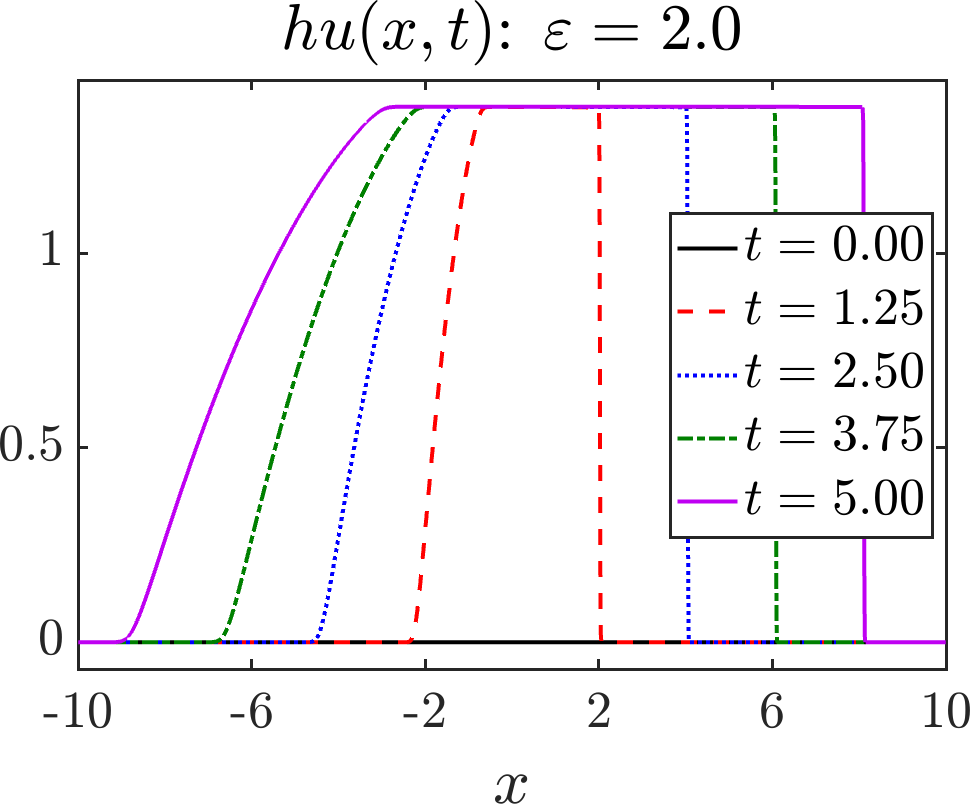}
}
\vspace{1ex}
\centerline{
\includegraphics[width=0.335\textwidth]{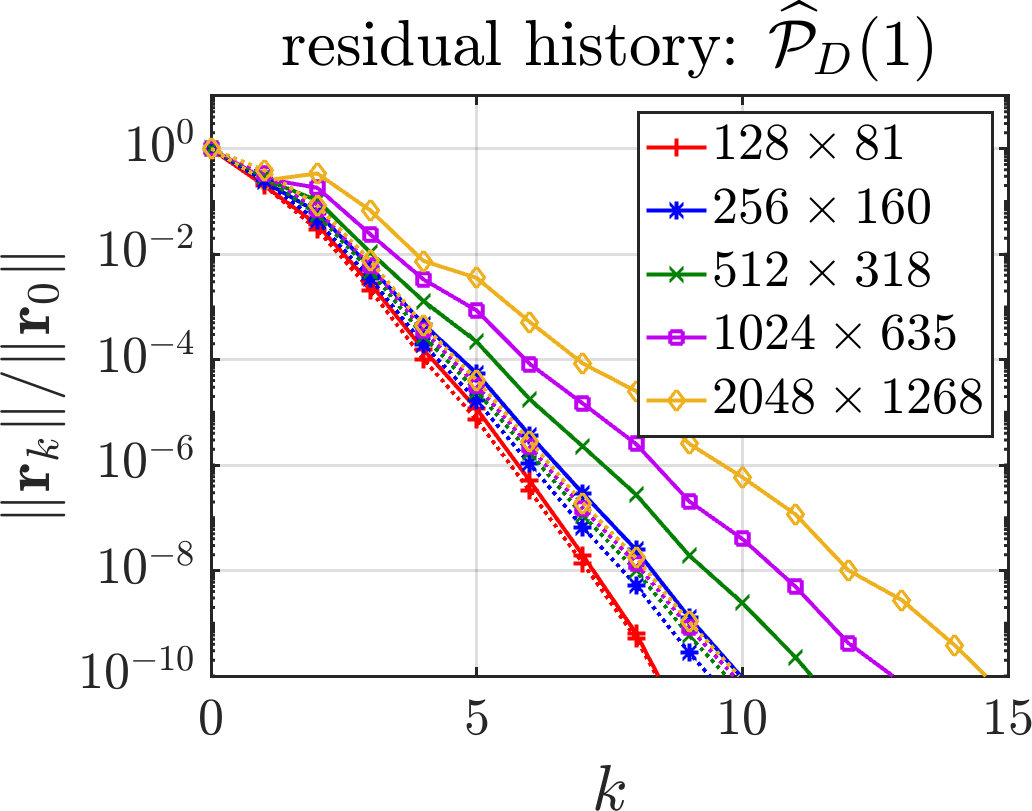}
\includegraphics[width=0.335\textwidth]{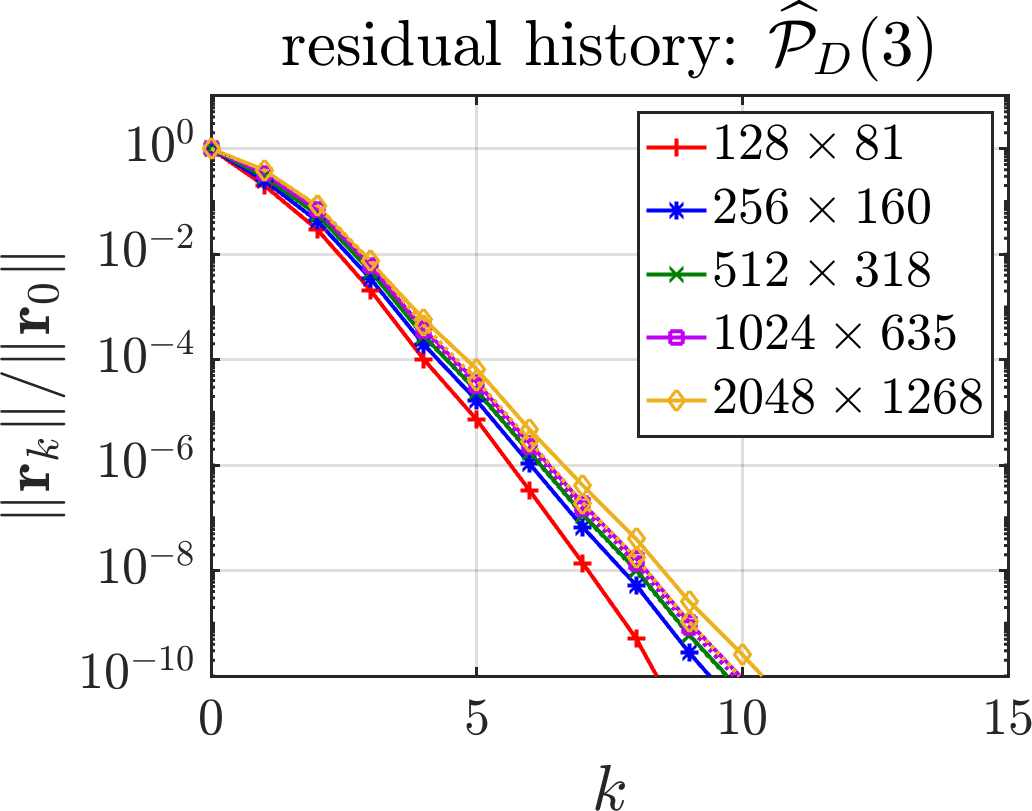}
\includegraphics[width=0.335\textwidth]{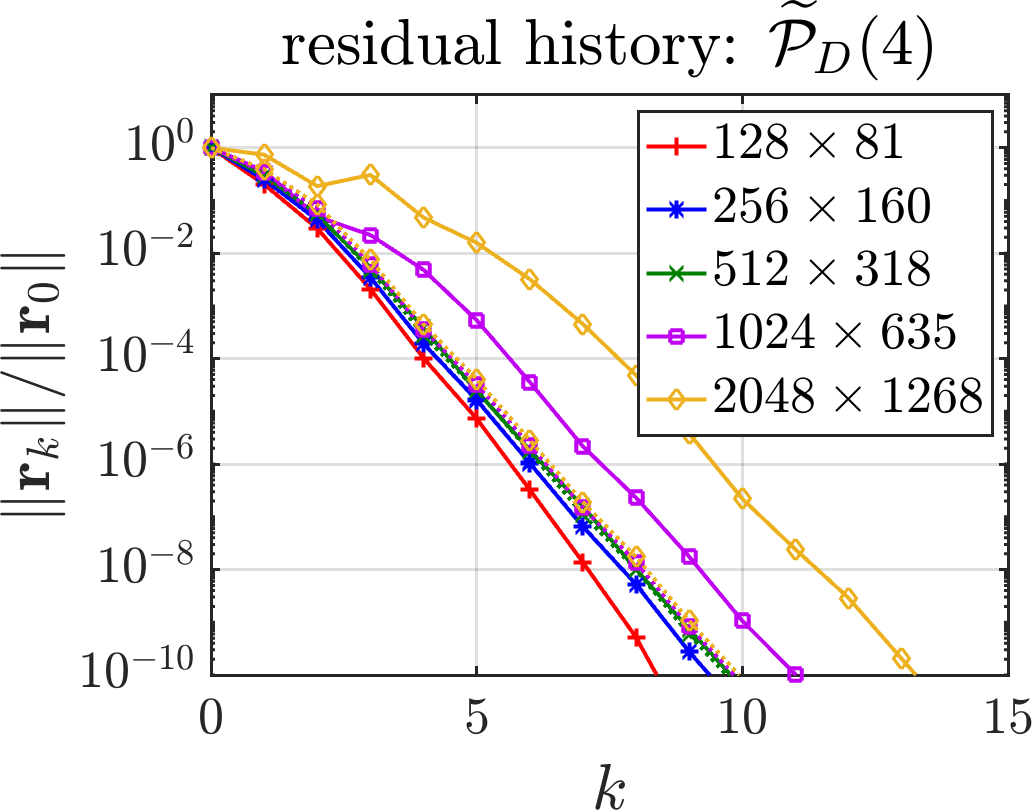}
}
\caption{
Larger-amplitude dam break problem \eqref{SMeq:SWE-db} for \eqref{eq:SWE}, with $\varepsilon = 2.0$.
Dotted lines in plots with $\wh{{\cal P}}$ in their titles correspond to exact solution of linearized systems. 
\label{SMfig:SWE-DB-fully}
}
\end{figure}

\begin{figure}[b!]
\centerline{
\includegraphics[width=0.345\textwidth]{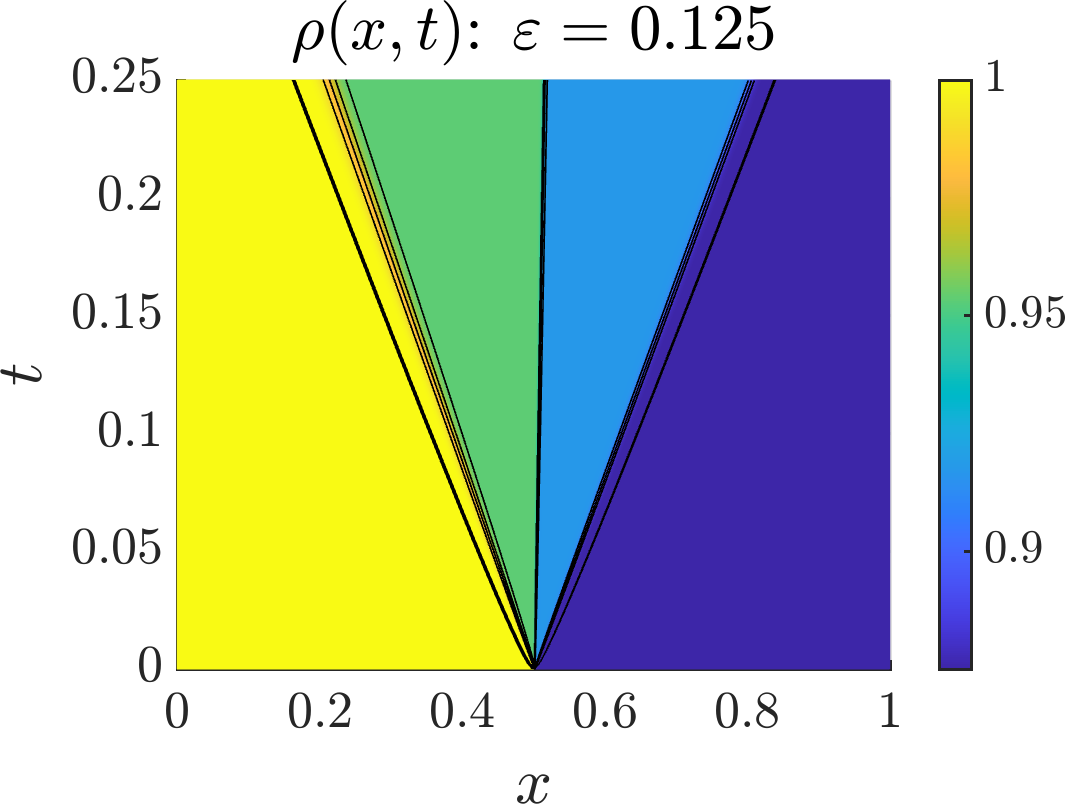}
\hspace{0ex}
\includegraphics[width=0.325\textwidth]{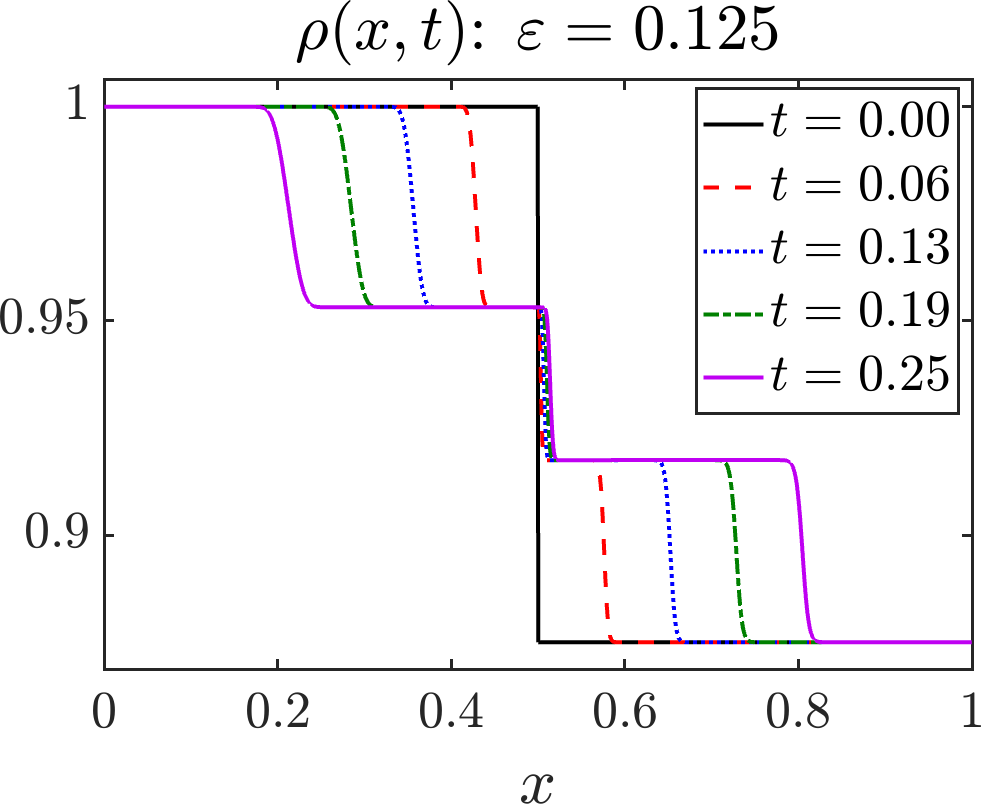}
\hspace{0ex}
\includegraphics[width=0.325\textwidth]{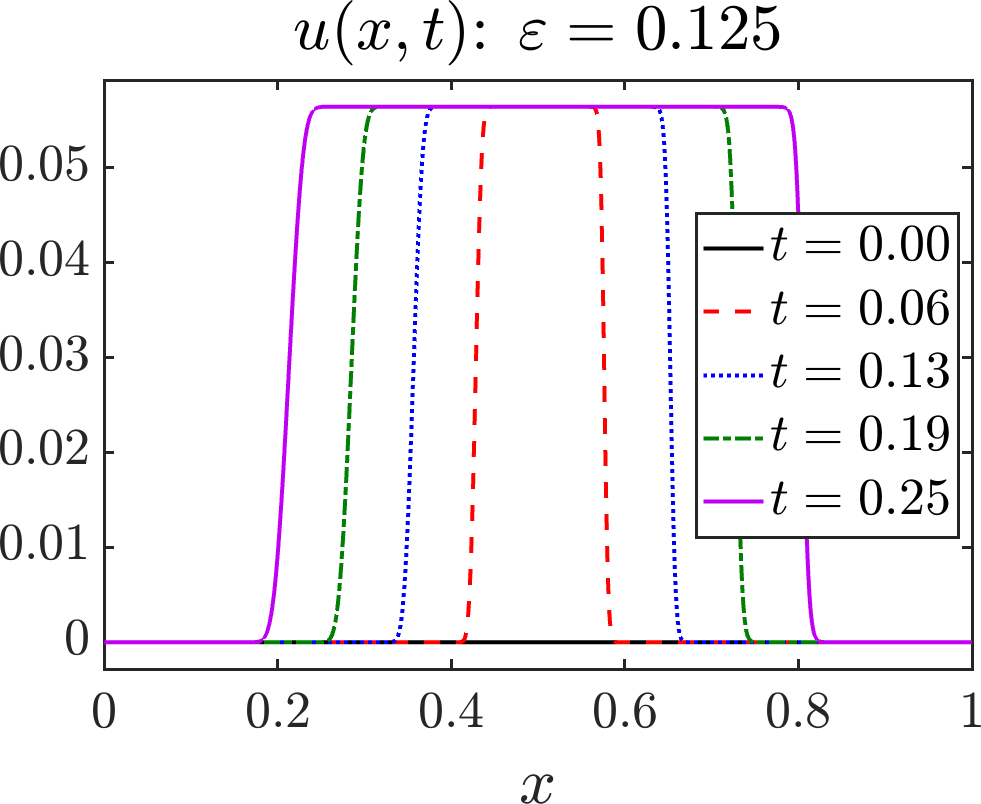}
}
\centerline{
\includegraphics[width=0.335\textwidth]{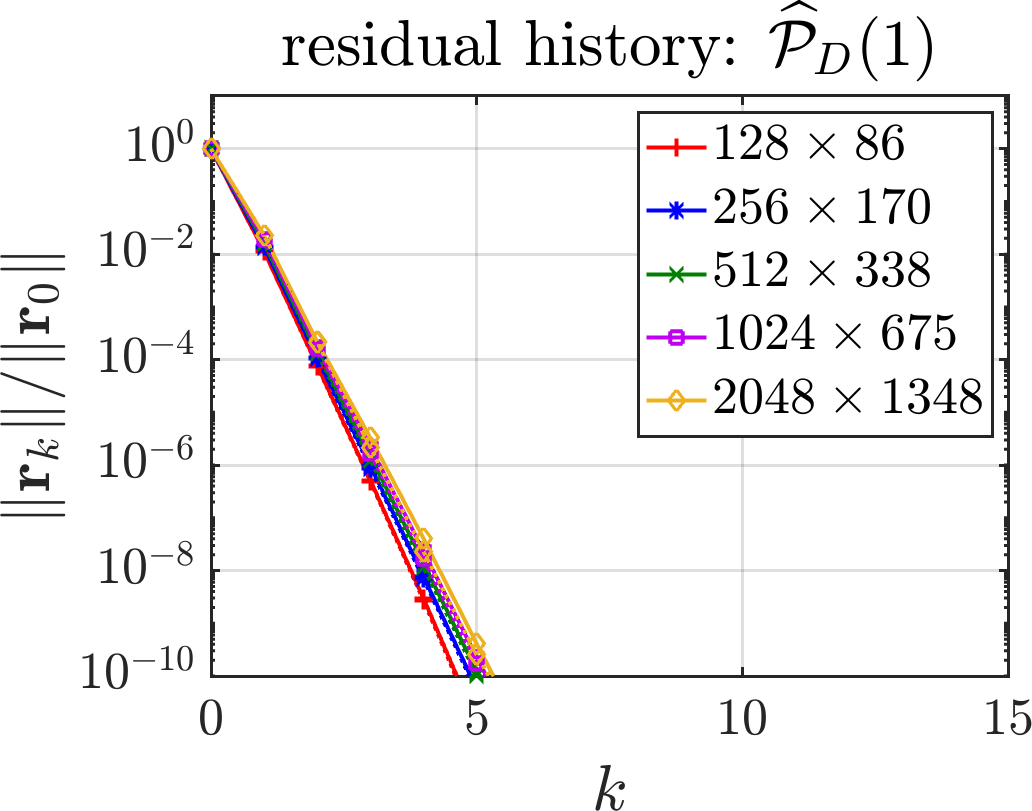}
\includegraphics[width=0.335\textwidth]{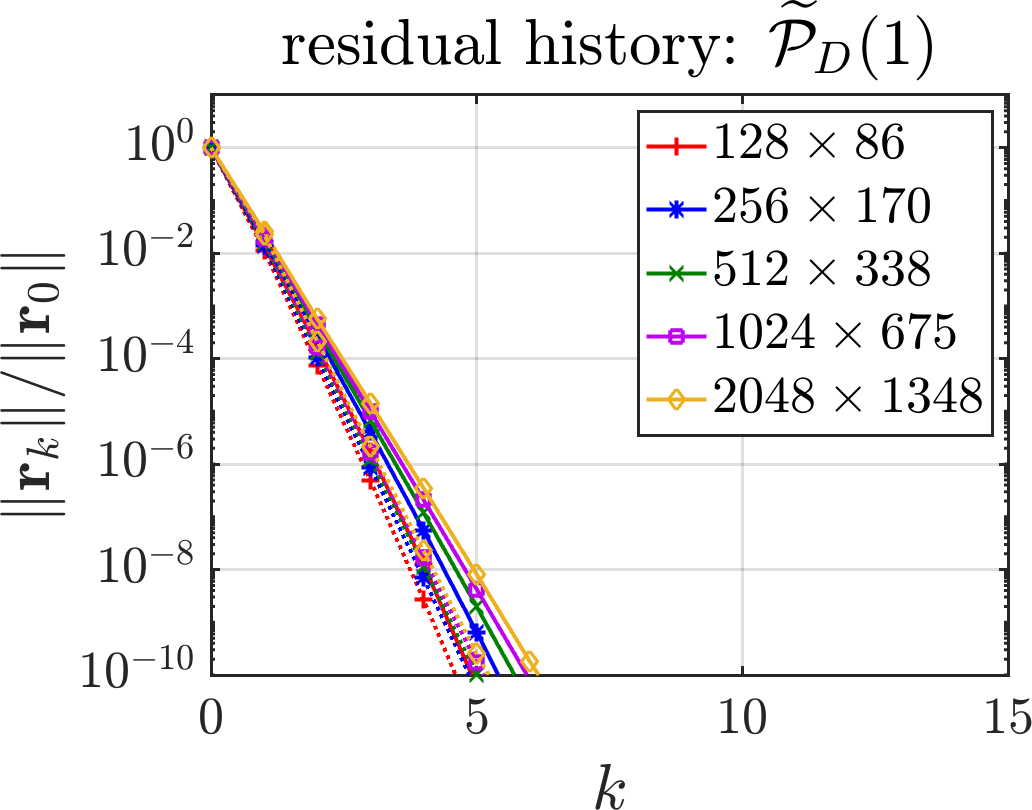}
}
\caption{
Small-amplitude Sod problem \eqref{SMeq:euler-sod} for \eqref{eq:euler} with  $\varepsilon = 0.125$.
Dotted lines in plots with $\wh{{\cal P}}$ in their titles correspond to exact solution of linearized systems. 
\label{SMfig:euler-sod-weakly}
}
\end{figure}
\begin{figure}[t!]
\centerline{
\includegraphics[width=0.345\textwidth]{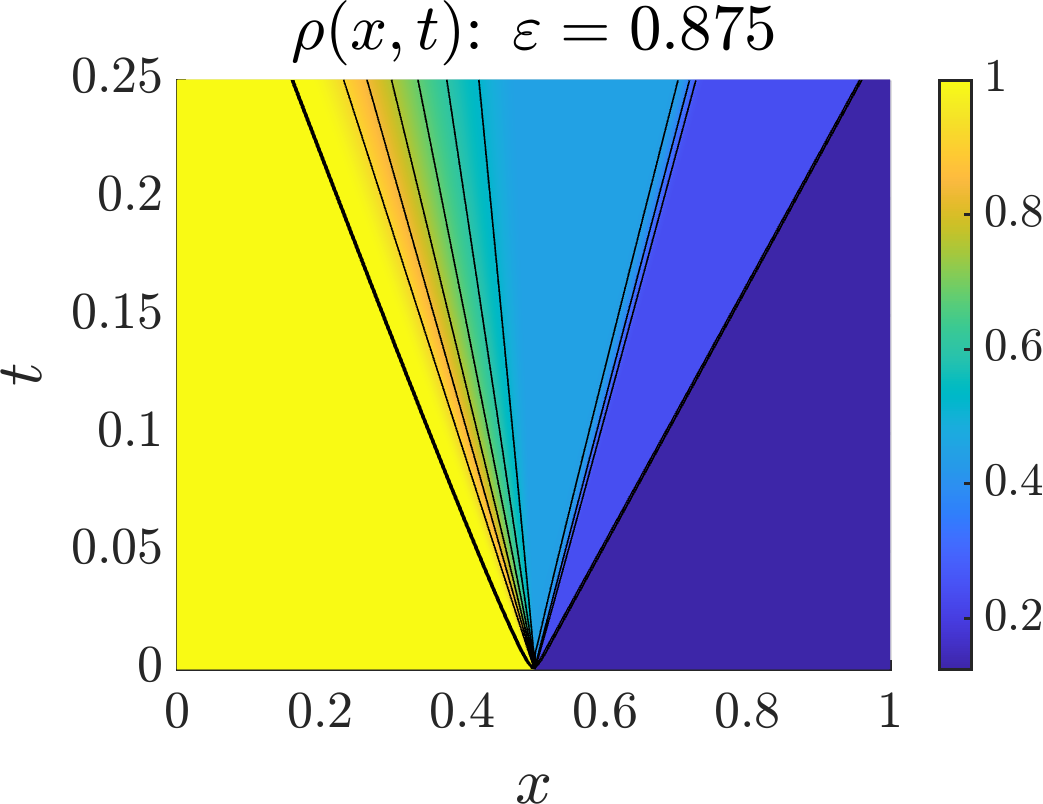}
\hspace{0ex}
\includegraphics[width=0.325\textwidth]{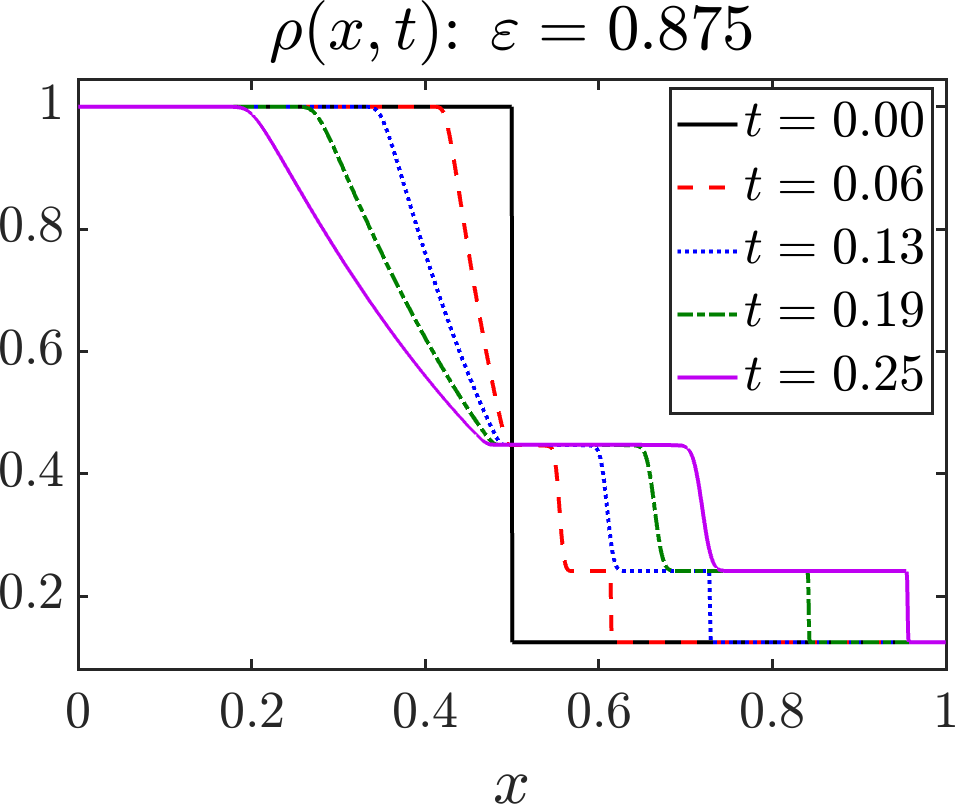}
\hspace{0ex}
\includegraphics[width=0.325\textwidth]{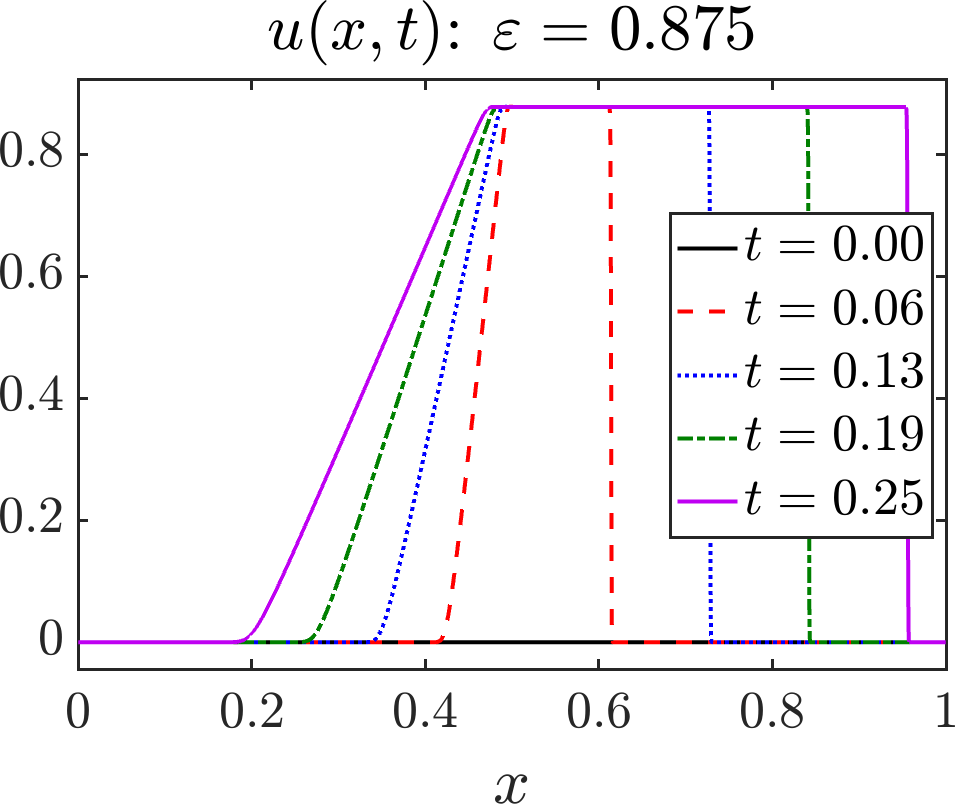}
}
\centerline{
\includegraphics[width=0.335\textwidth]{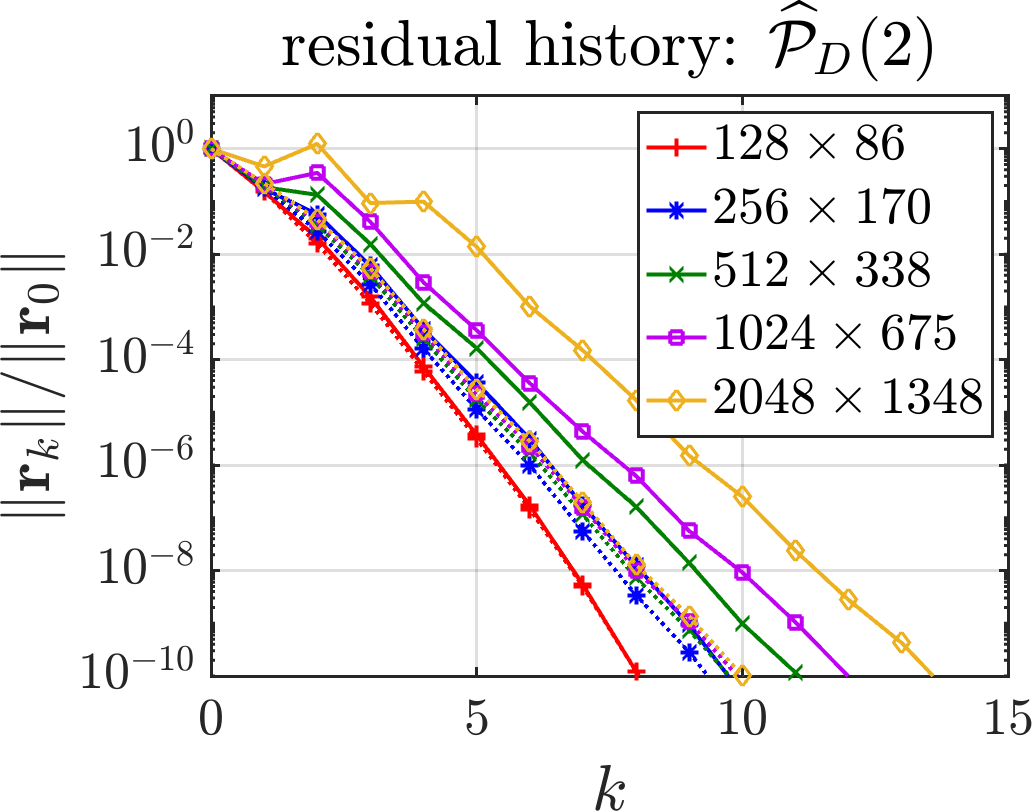}
\hspace{1ex}
\includegraphics[width=0.335\textwidth]{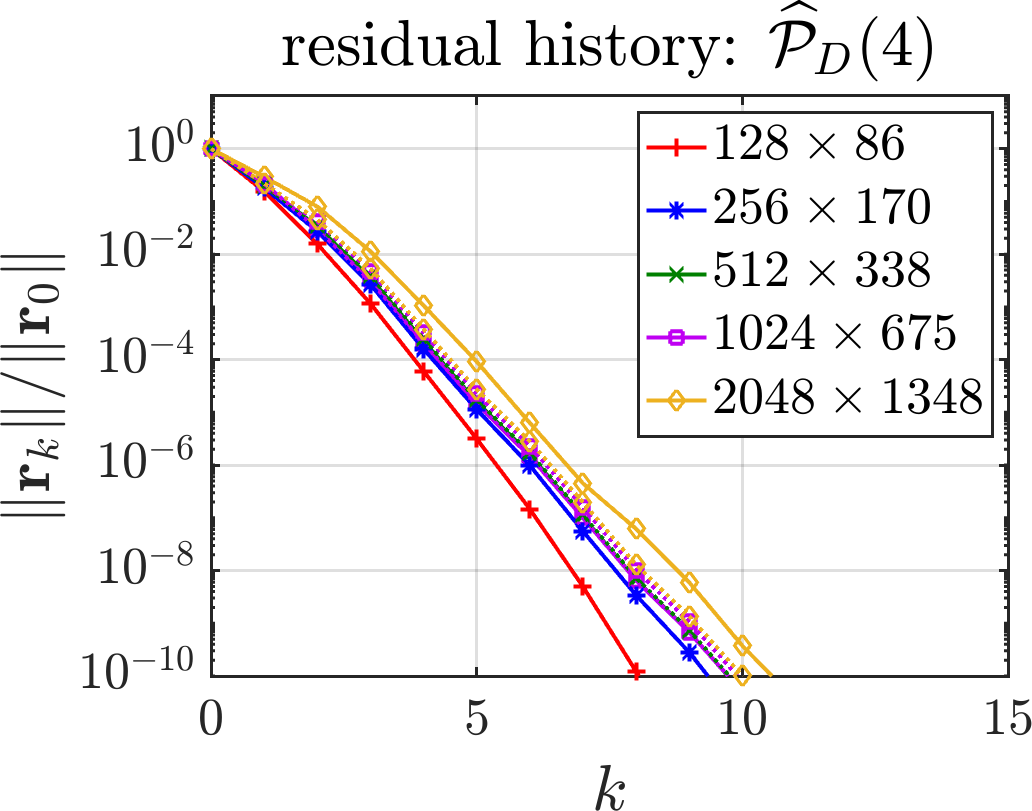}
}
\caption{
Larger-amplitude Sod problem \eqref{SMeq:euler-sod} for \eqref{eq:euler} with $\varepsilon = 0.875$. 
Dotted lines in plots with $\wh{{\cal P}}$ in their titles correspond to exact solution of linearized systems. 
\label{SMfig:euler-sod-fully}
}
\end{figure}

\end{document}